\documentclass[12pt,reqno,letter]{amsart}
\usepackage[
top=4cm, bottom=4cm, inner=3cm, outer=3cm,
marginparwidth=2cm 
]{geometry}

\usepackage{amssymb,latexsym, amsmath, amsxtra, mathrsfs, bm}

\usepackage[dvips]{graphics}
\usepackage[all]{xy}
\usepackage{xcolor}
\usepackage{subcaption}

\usepackage{verbatim}
\usepackage[abs]{overpic}
\usepackage{hyperref}
\usepackage{mathtools}
\usepackage{tikz}  
\usepackage{yhmath}

\usepackage{enumitem}

\usepackage{hyperref}

\setenumerate[1]{label=\upshape(\roman*)}

\DeclareMathAlphabet{\mathbbold}{U}{bbold}{m}{n}

\allowdisplaybreaks[3]

\theoremstyle{plain}
\newtheorem{theorem}{Theorem}[section]
\newtheorem*{theorem*}{Theorem}
\newtheorem*{conj*}{Conjecture}
\newtheorem{lemma}[theorem]{Lemma}
\newtheorem{prop}[theorem]{Proposition}
\newtheorem{cor}[theorem]{Corollary}
\newtheorem{thmx}{Theorem}

\theoremstyle{definition}
\newtheorem{definition}[theorem]{Definition}
\newtheorem{rem}[theorem]{Remark}

\theoremstyle{remark}
\newtheorem*{remark}{Remark}

\numberwithin{equation}{section}
\numberwithin{theorem}{section}
\numberwithin{table}{section}
\numberwithin{figure}{section}

\makeatletter
\def\th@plain{%
	\thm@notefont{}
	\itshape 
}
\def\th@definition{%
	\thm@notefont{}
	\normalfont 
}
\makeatother

\renewcommand{\leq}{\leqslant}

\renewcommand{\geq}{\geqslant}

\newcommand{\diam}  {\operatorname{diam}}

\newcommand{\card} {\operatorname{card}}
\newcommand{\supp}{\operatorname{supp}}

\newcommand{\R}{\mathbb{R}}

\newcommand{\N}{\mathbb{N}}

\providecommand{\abs}[1]{\lvert#1\rvert}
\providecommand{\Absbig}[1]{\bigl\lvert#1\bigr\rvert}

\providecommand{\AbsBig}[1]{\Bigl\lvert#1\Bigr\rvert}

\providecommand{\norm}[1]{\|#1\|}
\providecommand{\Normbig}[1]{\bigl\|#1\bigr\|}

\renewcommand{\:}{\colon}

\renewcommand{\=}{\coloneqq}

\newcommand{\cA}{\mathcal{A}}

\newcommand{\cD}{\mathcal{D}}
\newcommand{\cE}{\mathcal{E}}

\newcommand{\cG}{\mathcal{G}}

\newcommand{\cK}{\mathcal{K}}
\newcommand{\cL}{\mathcal{L}}
\newcommand{\cM}{\mathcal{M}}

\newcommand{\cR}{\mathcal{R}}

\newcommand{\cU}{\mathcal{U}}

\newcommand{\oA}{\overline{A}}

\newcommand{\ox}{\overline{x}}
\newcommand{\oy}{\overline{y}}

\newcommand{\tA}{\widetilde{A}}

\newcommand{\tz}{\widetilde{z}}

\newcommand{\energy}{Q(T,A)}

\newcommand{\cstar}[3][black]
{
    \tikz[baseline=(X.base)]{\node[circle, draw, inner sep=0pt, text=#1, minimum height=#2, font=#3] (X) {$*$};}
}
\newcommand{\ostar}{\!\!\scalebox{0.8}{\cstar{1pt}{\tiny}}\!\!}
\newcommand{\oddiv}{\oslash}

\newcommand{\cdiv}{  
	\begin{tikzpicture}[baseline=(center.center)]  
		\node[circle, draw, minimum size=7.5pt,inner sep=0pt] (center){}; 
		\node at (center) {\scriptsize $\div$}; 
	\end{tikzpicture}  
}
\newcommand{\odiv}{\hspace{-4pt}\raisebox{0.6ex}{\cdiv}\hspace{-4pt}}

\newcommand{\oplusu}[1] {\underset{#1}{\oplus}}

\newcommand{\Rm}{\overline{\underline{\mkern-1mu\R\mkern-1mu}}}
\newcommand{\Rminus}{\underline{\mkern-1mu\R\mkern-1mu}}
\newcommand{\aholder}{C^{0,\alpha}(X,d)}
\newcommand{\fless}{\leq}

\newcommand{\density}{\cD}

\newcommand{\es}[3] {\underline{\mkern-1mu\cE\mkern-1mu}_{#3}(#1,#2)}

\newcommand{\tdim}{\underline{\mkern-1mu\dim\mkern-1mu} \,}

\newcommand{\twidetilde}[1]{\widetilde{\mkern+1mu\underline {\mkern-1mu #1\mkern-1mu}\mkern+1mu}}

\newcommand{\twidehat}[1]{\widehat{\mkern+1mu\underline{\mkern-1mu #1 \mkern-1mu}\mkern+1mu}}

\newcommand{\toverline}[1]{\overline{\mkern+1mu\underline{\mkern-1mu #1 \mkern-1mu}\mkern+1mu}}

\newcommand{\tint}{\underline{\int \mkern-4mu}}

\newcommand{\tai}{\underline{\mkern-1mu 0 \mkern-1mu} }

\newcommand{\tmi}{\underline{\mkern-1mu 1 \mkern-1mu} }

\newcommand{\taif}{\underline{\mkern-1mu \mathbbold{0} \mkern-1mu} }

\newcommand{\tmif}{\underline{\mkern-1mu \mathbbold{1} \mkern-1mu} }

\newcommand{\tinftyf}{\underline{\mkern-1mu \bm{\infty} \mkern-1mu} }

\begin{document}
\title{Tropical thermodynamic formalism}
\author{Zhiqiang~Li \and Yiqing~Sun}
\address{Zhiqiang~Li, School of Mathematical Sciences \& Beijing International Center for Mathematical Research, Peking University, Beijing 100871, China.}
\email{zli@math.pku.edu.cn}
\address{Yiqing~Sun, School of Mathematical Sciences, Peking University, Beijing 100871, China.}
\email{yqsun@stu.pku.edu.cn}

	
	\begin{abstract}
    We investigate the zero-temperature large deviation principle for equilibrium states in the context of distance-expanding maps. The logarithmic-type zero-temperature limit in the large deviation principle induces a tropical algebra structure, which motivates our study of the tropical adjoint Bousch operator $\cL_A^{\ostar}$ since the Bousch operator $\cL_A$ is tropical linear and corresponds to the Ruelle operator $\cR_A$. 
    
    We extend tropical functional analysis, define the adjoint operator $\cL_A^{\ostar}$ as a tropical analog of the adjoint Ruelle operator $\cR_A^*$, and establish the existence and generic uniqueness of tropical eigendensities of $\cL_A^{\ostar}$ associated with the maximal eigenvalue. The Aubry set and the Ma\~{n}\'{e} potential, both originating from weak KAM theory, serve as important tools in the representations of tropical eigendensities. With our notion of tropical completeness and our tropical Riesz representation theorem, $\cL_A^{\ostar}$ can also be seen as a version of the tropical Koopman operator.
	\end{abstract}
	
	\maketitle
	
	\tableofcontents

\section{Introduction}     \label{sct_Introduction}
Large deviation theory characterizes the asymptotic computation of small probabilities on an exponential scale. This topic has gained much interest due to its various applications and the establishment of a general framework by Varadhan \cite{Va66}. For a general account of large deviation theory, see e.g.~\cite{DZ09} and \cite{El85}.

In dynamical systems, the large deviation of empirical means on orbits approximating a certain invariant measure in weak$^*$ topology has been widely studied (see e.g.~\cite{Ki90}, \cite{AP06}, \cite{MN08}, and \cite{CRT19}). Parallel to these studies, in this article, we investigate the \emph{zero-temperature} large deviations of equilibrium states parameterized by an inverse temperature, which have been studied in symbolic dynamics (see~\cite{BLT06} and \cite{Me18}). Specifically, we consider the following problem:
\begin{quote}
\it For which pairs $(T, A)$ of dynamical systems and potentials does the large deviation principle hold for the family of equilibrium states $\{\mu_{\beta A}\}_{\beta\in(1,+\infty)}$ as $\beta\to+\infty$?
\end{quote}
Similar phenomena have been investigated in other settings; see e.g.\ Yilin~Wang's seminal work \cite{Wan19} on the large deviation of SLE$_\kappa$ as $\kappa\to 0^+$.

Note that our formulation of the above problem does not a priori assume the weak$^*$ convergence of $\mu_{\beta A}$ as $\beta\to+\infty$ (cf.~\cite{DV75}). Generally, every weak$^*$ accumulation point of $\{\mu_{\beta A}\}_{\beta\in(1,+\infty)}$ as $\beta\to+\infty$ is a maximizing measure for the potential $A$. 

For a similar phenomenon in Lagrangian systems, every weak$^*$ accumulation point of the invariant measures generated by the twisted Schr\"{o}dinger operators as the viscosity coefficient tends to zero is an action minimizing measure for the Lagrangian (cf.~\cite{An04}).

In this article, we address the aforementioned problem in the context of distance-expanding maps by observing a natural connection: the logarithmic-type (zero-tem\-per\-a\-ture) limits in large deviation principles of equilibrium states induce a tropical algebra structure. This observation motivates our development of a theory of \emph{tropical thermodynamic formalism}, which further enriches the dictionary between ergodic optimization and thermodynamic formalism. 

Since the fundamental work of Cuninghame-Green \cite{Cu79} on tropical linear algebra, the tropical analogs of various mathematical branches such as functional analysis, complex analysis, and algebraic geometry have been extensively investigated (see e.g.~\cite{LMS01}, \cite{HS09}, and \cite{IMS09}). 

In tropical thermodynamic formalism, tropical (max-plus) algebra corresponds to the standard linear algebra over the real numbers, and the Bousch operator in ergodic optimization corresponds to the Ruelle operator in the classical theory of thermodynamic formalism. We systematically investigate this correspondence in the following steps:
\begin{enumerate}[label=\rm{(\roman*)}]	
\smallskip
\item Extending tropical functional analysis (also known as \emph{idempotent analysis}, cf.\ \cite{Ak99}, \cite{CGQ04}, and \cite{LMS01}) to support dynamical applications.

\smallskip
\item Developing adjoint operator theory for the tropical regime.

\smallskip
\item Establishing logarithmic-type zero-temperature limits to operators and eigenmeasures beyond eigenfunctions, creating explicit bridges between thermodynamic and tropical objects.
\end{enumerate} 

To describe these points in more detail, we first review some basic notions of thermodynamic formalism and ergodic optimization.

Thermodynamic formalism in ergodic theory dates back to the works of Sinai, Bowen, Ruelle, and others around the early 1970s \cite{Do68, Si72, Bow75, Ru78}, inspired by statistical mechanics. See also \cite{Br65, Ly82} for early works in complex dynamics. To be more precise, let $T\: X\to X$ be a continuous map on a compact metric space $(X,d)$, and $\varphi\: X\to\R$ be a continuous function. The measure-theoretic pressure is defined by
\begin{equation}\label{eq: measure-theoretic pressure} 
	P_\mu(T,\varphi) \= h_\mu(T) + \int\! \varphi \,\mathrm{d}\mu,
\end{equation}
where $h_\mu(T)$ is the measure-theoretic entropy and $\mu$ is a $T$-invariant Borel probability measure. The measure $\mu$ is called an \emph{equilibrium state} if $\mu$ maximizes $P_\mu(T,\varphi)$ and the maximum is the \emph{topological pressure}, which we denote by $P(T,\varphi)$. In particular, for a constant potential, an equilibrium state is called a \emph{measure of maximal entropy}. Equilibrium states are the central focus of thermodynamic formalism. The Ruelle operator $\cR_{\varphi}$, also known as the Ruelle--Perron--Frobenius operator or the transfer operator, was introduced by Ruelle to study the equilibrium states. For example, it is well known that if $T\:X\to X$ is open, continuous, distance-expanding, and transitive, and $\varphi$ is Lipschitz, then the product of the unique eigenfunction $u_{\varphi}$ of $\cR_{\varphi}$ and the unique eigenmeasure $m_{\varphi}$ of $\cR_{\varphi}^{*}$ (both associated with eigenvalue $e^{P(T,\varphi)}$) is the unique equilibrium state $\mu_{\varphi}$ (up to a multiplicative constant). Moreover, the spectral properties are of significant importance, serving as a foundation for further study of statistical properties and limit theorems. 

 Ergodic optimization originated in the 1990s from the works of Hunt and Ott \cite{HO96a, HO96b}, with motivation from control theory \cite{OGY90, SGYO93}, and the Ph.D.\ thesis of Jenkinson \cite{Je96}. Ergodic optimization seeks to understand \emph{maximizing measures}. For a continuous map $T\: X\to X$ on a compact metric space $X$, let $\cM(X,T)$ denote the set of $T$-invariant Borel probability measures on $X$, and define the \emph{maximal potential energy} of a continuous function $A\:X\to\R$ (known as the potential function) to be
 \begin{equation}\label{eq: defi of energy}
 	\energy\coloneqq\sup\biggl\{\int \! A \, \mathrm{d} \mu:\mu\in \cM(X,T)\biggr\}\in\R.
 \end{equation}
 The supremum is attained due to the weak$^*$-compactness of $\cM(X, T)$. Any measure $\mu\in \cM(X,T)$ that satisfies $\int\! A \, \mathrm{d}\mu=\energy$ is called a \emph{maximizing measure} for $T$ and $A$, and the (nonempty) set of such measures is denoted by 
 \begin{equation*}
 	\cM_{\max}(T,A)\=\biggl\{\mu\in \cM(X,T):\int \! A\, \mathrm{d}\mu=\energy\biggr\}.
 \end{equation*}
Bousch \cite{Bous00} proposed to consider fixed points of an operator $\cL_A$, which we call the \emph{Bousch operator} for the potential $A$. These fixed points are used to reveal the support of maximizing measures of $A$. The operator is also known as the Bousch--Lax operator or the Lax operator in the literature since an analogous construction gives the \emph{Lax--Oleinik semigroups} in the context of Hamiltonian systems. For comprehensive surveys on ergodic optimization, see Jenkinson \cite{Je06, Je19} and Bochi \cite{Boc18}.  

In the aforementioned step~(i), we introduce our notion of \emph{tropical completion} and establish a \emph{tropical Riesz representation theorem} for the tropical completion of the space of continuous functions $C(X,\R)$. A notable difference from the conventional functional analysis is that \emph{density functions} take the place of measures in our tropical functional analysis so that we define the \emph{tropical adjoint Bousch operator} $\cL_A^{\ostar}$ on the space of density functions $\density(X)$ (see (\ref{eq: defi of the density space})). 

In step~(ii), we seek to understand the existence and uniqueness of the tropical eigendensities of $\cL_A^{\ostar}$ associated with eigenvalue $\energy$. In general, the eigendensities are not unique, but we are able to establish representations of them via the Aubry set and the Ma\~{n}\'{e} potential. These two notions originate from the weak KAM theory in Lagrangian systems. For more about weak KAM theory, see e.g.~Contreras \& Iturriaga \cite{CI99} and Kaloshin \& Zhang \cite{KZ20}. We remark that the tropical adjoint Bousch operator $\cL_{A}^{\ostar}$ can be seen as the counterpart of the backward Lax--Oleinik semigroup.

In step~(iii), as we will see in Section~\ref{s: zero-temperature limits}, the logarithmic-type zero-temperature limit of the Ruelle operators is the Bousch operator, and the accumulation points of eigenfunctions (resp.~eigenmeasures) of the Ruelle operators (resp.~the adjoint Ruelle operators) under logarithmic scaling are the tropical eigenfunctions (resp.~eigendensities) of the Bousch operator (resp.~the adjoint Bousch operator).

Now the uniqueness of tropical eigen-objects directly entails the zero-temperature large deviation principle for equilibrium states. We demonstrate a sufficient condition for the uniqueness of tropical eigen-objects that holds for a generic H\"{o}lder potential and derive a characterization theorem for the large deviation principle within this framework.

We remark that although our results are stated for open, continuous, and distance-expanding maps and H\"{o}lder continuous potentials, they can be extended to more general systems and potentials with appropriate modifications.   

\medskip
\subsection{Tropical algebra}
We consider $\Rm \= \R\cup\{+\infty,\,-\infty\}$ equipped with the tropical (max-plus) algebra:
\begin{equation*}
	x\oplus y \= \max\{x,\,y\},\quad x\otimes y \=  x+y,\quad\text{for } x,\, y\in\Rm.
\end{equation*}
Here $-\infty$ is seen as the tropical additive identity, and we adopt the convention that $(+\infty)\otimes(-\infty)=(-\infty)\otimes(+\infty)=-\infty$. For each subset $\cA\subseteq\Rm$, we use $\oplusu{x\in\cA}x$ to denote the supremum in $\Rm$, and $\oplusu{x\in\emptyset}x$ is defined to be $-\infty$. It is straightforward to check that $(\Rm,\oplus,\otimes)$ is a semiring and 
\begin{equation}\label{eq: exchangability of numbers}
    c\otimes \bigl(\oplusu{x\in\cA}x\bigr)=\oplusu{x\in\cA}(c\otimes x)
\end{equation}
for all $\cA\subseteq\Rm$ and $c\in\Rm$. The basis $\{(a,b),[-\infty,a),(b,+\infty]: a,b\in\R\}$ generates the desired topology on $\Rm$. Denote $\Rminus\=\R\cup\{-\infty\}$.

Let $\Rm^{X}$ denote the space of $\Rm$-valued functions on a set $X$. The space $\Rm^X$ with $(u\oplus v)(x) \=  u(x)\oplus v(x)$ for all $u,v\in\Rm^X$ and $(\lambda\otimes u)(x) \= \lambda\otimes u(x)$ for all $u\in\Rm^X$ and $\lambda\in\Rm$ is a $\Rm$-semimodule (see Section~\ref{s: tropical functional analysis} for a review of related notions). Additionally, we define $(u\otimes\lambda)(x)\=u(x)\otimes\lambda$ for all $u\in\Rm^X$ and $\lambda\in\Rm$. The space $\Rm^X$ has the natural order: $u\fless v$ if and only if $u(x)\leq v(x)$ for all $x$ in $X$, i.e., $u\fless v$ if and only if $u\oplus v=v$. For each subset $U\subseteq\Rm^X$, we use $\oplusu{u\in U} u$ to denote the pointwise supremum in $\Rm^X$. We write $(u\otimes v)(x)\=u(x)\otimes v(x)$ for all $u,v\in\Rm^X$.

Let $C(X,Y)$ be the space of continuous functions from a topological space $X$ to a topological space $Y$. 

In the sequel, $\oplus$ always means $\sup$. Denote by $\tai\=-\infty$ (resp.\ $\tmi\=0$) the tropical additive (resp.\ multiplicative) identity of $\Rm$. 

We use $\mathbbold{1}_X$, $\tmif_X$, $\taif_X$, and $\tinftyf_X$ to represent the constant $1$, $\tmi$, $\tai$, and $+\infty$ functions on a set $X$, respectively, and denote $\mathbbold{0}_X \= \tmif_X$. We often omit the subscript $X$ when it is clear from the context.

\subsection{Distance-expanding maps}

A map $T$ on a topological space $X$ is called a \emph{covering map} if $T$ is open, continuous, surjective, and is a local homeomorphism. On a metric space $(X,d)$,  a map $T$ is said to be  \emph{distance-expanding} if there exist constants $\lambda>1$ and $\eta>0$ such that for all $x,\, y\in X$, $d(x,y)\leq 2\eta$ implies $d(T(x),T(y))\geq\lambda d(x,y)$. We say that a covering map $T$ on a metric space $(X,d)$ is an \emph{expanding covering map} if $T$ is distance-expanding.

This article primarily focuses on expanding covering maps on compact metric spaces, though many results apply to nonsurjective cases. Note that for a compact metric space $(X,d)$, the map $T\: X\to X$ constitutes an expanding covering map if and only if it is open, continuous, surjective\footnote{The surjectivity condition guarantees that the Bousch operator $\mathcal{L}_{A}$ preserves $C(X,\R )$, which is often essential as the tropical dual space of $C(X,\R)$ obtained through the tropical Riesz representation theorem (established in Section~\ref{s: tropical functional analysis}) precisely corresponds to (up to a scaling) the set of (large deviation) rate functions---a matching that fails for $C(X,\R \cup\{-\infty\})$. Additionally, when studying invariant objects like equilibrium states (which are inherently supported on the set of nonwandering points), we can safely ignore points violating surjectivity.}, and distance-expanding (cf.~\cite[Section~4.1]{PU10}). Transitivity is generally not required, and we impose this condition only where necessary.

 Let $\aholder$ denote the space of $\alpha$-H\"{o}lder continuous functions $\varphi \: X\rightarrow\R$ with respect to the metric $d$ for $\alpha\in(0,1]$. For a metric space $(X,d)$, denote $B(x,r) \= \{y\in X:d(x,y)<r\}$ for all $x\in X$ and $r>0$. We use $\N$ to denote the set of positive integers and $\N_0 \=\N\cup\{0\}$.

\subsection{Ruelle operators and Bousch operators}\label{ss: Ruelle and Bousch operators}

For a real-valued continuous function $A\in C(X,\R)$, the operator $\cR_{A} \: C(X,\R)\rightarrow C(X,\R)$ given by
\begin{equation}\label{eq: defi of the Ruelle operator}
 u\mapsto \cR_{A}(u)(x)\coloneqq\sum_{y\in T^{-1}(x)}u(y)e^{{A}(y)}
\end{equation} is called the \emph{Ruelle operator} for the potential $A$. For a transitive expanding covering map $T\:X\rightarrow X$ and an $\alpha$-H\"{o}lder continuous potential $A\in\aholder$, it is well known that $\cR_{A}$ (resp.\ its adjoint operator $\cR_{A}^*$) has a unique eigenfunction (resp.\ eigenmeasure) in $C(X,\R)$ up to a constant associated with the eigenvalue of maximal modulus, i.e., $e^{P(T, A)}$, where $P(T, A)$ is the topological pressure of $T$ with respect to the potential $A$. Let $m_A$ be the unique eigenprobability of $\cR_A^*$ associated with eigenvalue $e^{P(T,A)}$ and $u_A$ be the unique eigenfunction of $\cR_A$ in $C(X,\R)$ associated with eigenvalue $e^{P(T,A)}$ satisfying $\int \! u_A \, \mathrm{d} m_A=1$. It is known that $u_A$ is strictly positive and
\begin{equation}\label{eq: equilibrium states}
	\mu_A\=u_A\cdot m_A
	\end{equation} 
    is the unique equilibrium state. Here the transitivity assumption guarantees the uniqueness of $u_A$ and $m_A$, and implies that $u_A$ is strictly positive. See e.g.~\cite[Chapters~3 and 5]{PU10} for more details.

To state the definition of the tropical adjoint Bousch operator, we need the following definition of the space of densities $\density(X)$:
\begin{equation}\label{eq: defi of the density space}
\density(X)\=\{b \: X\rightarrow\Rminus : b \text{ is upper semi-continuous}\}\cup\{\tinftyf\}.
\end{equation}
We say that $b_1\: X\to\Rm$ and $b_2\: X\to\Rm$ are \emph{equivalent} if $\oplusu{x\in X}(f(x)\otimes b_1(x))=\oplusu{x\in X}(f(x)\otimes b_2(x))$ for all $f\in C(X,\R)$. For all $b_1\:X\to\Rm$, there exists a unique $b_2\in\density(X)$ equivalent to $b_1$ (see Remark~\ref{r:density and functional}). 

\begin{definition}[Bousch operator and its tropical adjoint]\label{d: definition of the two operator}
Let $T\:X\rightarrow X$ be a map on a set $X$ and $A\in\Rm^X$.
The \emph{Bousch operator} $\cL_{A} : \Rm^X \rightarrow \Rm^X$ for the potential $A$ is defined by
\begin{equation}\label{eq: defi of the Bousch operator}
	\cL_{A}(u)(x) \= \oplusu{y\in T^{-1}(x)}(u(y)\otimes A(y))
\end{equation} 
for all $u\in\Rm^X$ and $x\in X$. If, in addition, $X$ is a compact metric space, $T$ is continuous, and $A\in C(X,\R)$, then we define the \emph{tropical adjoint Bousch operator} $\cL_A^{\ostar}\: \density(X)\to \density(X)$ acting on the space of densities by
\begin{equation}\label{eq:definition of tropical adjoint}
    \cL_A^{\ostar}(b)(x)\=b(T(x))\otimes A(x)
\end{equation}
for all $b\in \density(X)$ and $x\in X$.
\end{definition}

\begin{remark}
The Bousch operator $\cL_A$ is a tropically continuous linear map (Definition~\ref{d:max-plus linear}), see Lemma~\ref{l: generalized form of exchange lemma}. Similar to the expression for the iterations of the Ruelle operator, for all $n\in\N$, $u\in\Rm^X$, and $x\in X$,
    \begin{equation*}
    \cL_A^n(u)(x)=\oplusu{y\in T^{-n}(x)}\bigl(u(y)\otimes A(y)\otimes A(T(y))\otimes\cdots\otimes A\bigl(T^{n-1}(y)\bigr)\bigr) .
    \end{equation*}
Recall that $\oplusu{x\in \emptyset}x=-\infty$. The Bousch operator $\cL_A$ maps $C(X,\Rminus)$ into itself; under the assumption that $T\:X\to X$ is surjective, it further preserves $C(X,\R)$ (see Proposition~\ref{p:L_continuous}). 
For an explanation for our definition of $\cL_A^{\ostar}$ in (\ref{eq:definition of tropical adjoint}), see Remark~\ref{r:reason for defi of the dual operator}. Note that $\cL_A^{\ostar}$ can also be seen as the tropical analog of the (weighted) Koopman operator\footnote{In thermodynamic formalism, the (weighted) Koopman operator $\cK_A\: L^2(\mu_A)\to L^2(\mu_A)$ for a potential $A\in\aholder$ is given by $\cK_A(f)\coloneqq (f\circ T)\cdot e^A$, where $\mu_A$ is the equilibrium state in (\ref{eq: equilibrium states}). It is easy to verify that $\cK_A$ is the adjoint operator of the Ruelle operator $\cR_A$ on $L^2(\mu_A)$, i.e., $(\cR_A(f),g)=(f,\cK_A(g))$ for all $f,g\in L^2(\mu_A)$. }.
\end{remark}

Let $V$ be a subset of $\Rm^X$.
We define the following related notions:
\begin{enumerate}[label=\rm{(\roman*)}]	
	\smallskip
	\item $Q\in\Rm$ is a \emph{tropical eigenvalue} of $\cL_A$ on $V$ if there exists $u\in V$ such that $\cL_A(u)=Q\otimes u$ and $u(x)\in\R$ for some $x\in X$. Such a function $u\in V$ is called a \emph{tropical eigenfunction} of $\cL_A$ in $V$ (associated with eigenvalue $Q$). The \emph{tropical eigenspace} $\es{\cL_A}{V}{Q}$ of $\cL_A$ for the (tropical) eigenvalue $Q$ consists of all $u\in V$ such that $\cL_A(u)=Q\otimes u$.

    \smallskip
	\item $Q\in\Rm$ is a \emph{tropical eigenvalue} of $\cL_A^{\ostar}$ (on $\density(X)$) if there exists $b\in \density(X)\smallsetminus\{\tinftyf,\,\taif\}$ satisfying $\cL_A^{\ostar}(b)=Q\otimes b$. Such a function $b$ is called a \emph{tropical eigendensity} of $\cL_A^{\ostar}$ (associated with eigenvalue $Q$). The \emph{tropical eigenspace} $\es{\cL_A^{\ostar}}{\density(X)}{Q}$ of $\cL_A^{\ostar}$ for the (tropical) eigenvalue $Q$ consists of all $b\in \density(X)$ such that $\cL_A^{\ostar}(b)=Q\otimes b$.

    \smallskip
    \item A subset $E$ of $\Rm^X$ is \emph{of tropical dimension one}, written as $\tdim E=1$, if there exists $u\in E$ such that $E\smallsetminus\{\taif,\,\tinftyf,\,(+\infty)\otimes u\}=\{a\otimes u:a\in\R\}$ and $u(x)\in\R$ for some $x\in X$.
    
	\smallskip
    \item $u\in C(X,\R)$ is a \emph{sub-action} for the potential $A$ if $\cL_{A}(u)\fless\energy\otimes u$.
\end{enumerate}

In ergodic optimization, a function in $\es{\cL_A}{C(X,\R)}{\energy}$ is called a ``calibrated sub-action'' (see e.g.~\cite{Ga17}).

We call the potential $A\in C(X,\R)$ \emph{uniquely maximizing} if $\cM_{\max}(T, A)$ consists of a single measure. It is known in many contexts that the set of uniquely maximizing potentials in the space $\aholder$ of real-valued $\alpha$-H\"older continuous functions on $(X,d)$ is \emph{generic} (i.e., contains a countable intersection of open and dense subsets of $\aholder$). This fact was proved in a slightly stronger form by Contreras, Lopes, and Thieullen \cite{CLT01} for $C^1$ expanding maps and extended to continuous maps on compact metric spaces and more general potentials by Jenkinson \cite{Je06}. More recently, Contreras \cite{Co16} proved that for an open and dense subset of $A$ in $\aholder$, $\cM_{\max}(T, A)$ consists of a single \emph{periodic} measure for open Lipschitz distance-expanding maps $T$, known as the Yuan--Hunt conjecture \cite{YH99}. Contreras' work is then followed by \cite{HLMXZ19}, completing the uniformly hyperbolic case. Going beyond the setting of uniform hyperbolicity, the first-named author of this article and Zhang \cite{LZ25} proved the analogous result for expanding Thurston maps.

\subsection{The zero-temperature large deviation principle}  \label{ss:LDP}

For all $r>0$, a family of Borel probability measures $\{\nu_{\beta}\}_{\beta\in(r,+\infty)}$ on a topological space $X$ satisfies the \emph{large deviation principle as $\beta\to+\infty$} if there exists a lower semi-continuous function $I \: X \rightarrow [0,+\infty]$ (called the \emph{rate function}) for which the following two inequalities hold:
\begin{align}
		\liminf_{\beta\to+\infty}\frac{1}{\beta}\log\nu_{\beta}(\cG)&\geq-\inf_{x\in \cG}I(x),\quad\text{ for every open set } \cG\subseteq X,   \label{e:LDP_lower_bound}\\
		\limsup_{\beta\to+\infty}\frac{1}{\beta}\log\nu_{\beta}(\cK)&\leq-\inf_{x\in \cK}I(x),\quad\text{ for every closed set } \cK\subseteq X.  \label{e:LDP_upper_bound}
\end{align}
Here $\beta$ is called the \emph{inverse temperature}.
	Note that $I(x)\in [0,+\infty]$ for all $x\in X$, which follows from the fact that $\nu_{\beta}$ is a probability measure.	It immediately follows from Remark~(a) after \cite[Lemma~4.1.4]{DZ09} that if $X$ is a compact metric space, then the rate function is unique as long as a large deviation principle holds.


\medskip
We are now ready to state our main results.

\begin{thmx}[Typical zero-temperature large deviation principle]\label{t: uniquely maximizing implies the large deviation principle}
	Let $T\:X\to X$ be an open continuous distance-expanding map on a compact metric space $(X,d)$, and $\alpha \in (0,1]$. Then for each $\alpha$-H\"{o}lder continuous function $A\:X\to\R$ that has a unique maximizing measure, the following property holds:

    \smallskip
    
    If $\mu_{\beta A}$ is an equilibrium state for $\beta A$ for each $\beta\in(1,+\infty)$, then $\{\mu_{\beta A}\}_{\beta\in(1,+\infty)}$ satisfies the large deviation principle as $\beta\to +\infty$.

    \smallskip

    In particular, in the space of real-valued $\alpha$-H\"older continuous functions on $X$ equipped with the $\alpha$-H\"older norm, the above property holds for a generic subset of $A$, and if in addition $T$ is Lipschitz continuous, then the above property holds for an open and dense subset of $A$.
\end{thmx}

One can describe the rate function for the above large deviation principle in terms of the tropical eigenfunctions of $\cL_A$ and tropical eigendensities of $\cL_A^{\ostar}$ (see~Corollary~\ref{c:uniquely maximizing implies LDP} and Subsection~\ref{ss: proof of A and B}).

\subsection{Logarithmic-type zero-temperature limits, tropical eigenfunctions and eigendensities}
Assume that $T\: X\to X$ is a transitive expanding covering map on a compact metric space $(X,d)$, and $A\in\aholder$. For all $\beta>0$, recall from Subsection~\ref{ss: Ruelle and Bousch operators} that $m_{\beta A}$ is the unique eigenprobability of $\cR_{\beta A}^*$ associated with eigenvalue $e^{P(T,\beta A)}$, $u_{\beta A}$ is the unique eigenfunction of $\cR_{\beta A}$ associated with eigenvalue $e^{P(T,\beta A)}$ satisfying $\int_Xu_{\beta A}\,\mathrm{d}m_{\beta A}=1$, and $\mu_{\beta A}=u_{\beta A}\cdot m_{\beta A}$ is the unique equilibrium state. We denote
\begin{equation} \label{eq:defi tA}
    \tA \= A +\log u_{ A}-\log u_{ A}\circ T-P(T, A)
\end{equation}
and 
\begin{equation}\label{eq:defi of functional at positive temperature} 
	l_{\beta}^{\mu}(f) \= \frac{1}{\beta}\log\int \! e^{\beta f} \, \mathrm{d} \mu_{\beta A}\quad\text{ and }\quad
	l_{\beta}^{m}(f) \= \frac{1}{\beta}\log\int \! e^{\beta f} \, \mathrm{d} m_{\beta A},
\end{equation}
for all $f\in C(X,\R)\cup\{\tinftyf,\,\taif\}$ and $\beta>0$.\footnote{For the constant infinity functions, the values of the functionals are defined by the standard convention for integrals of Borel measurable functions valued in $[0,+\infty]$ and the natural extension that $\log +\infty=+\infty$ and $\log 0=-\infty$. Hence, $l_{\beta}^\mu(\tinftyf)=l_{\beta}^m(\tinftyf)=+\infty$ and $l_{\beta}^{\mu}(\taif)=l_{\beta}^m(\taif)=-\infty$ for all $\beta>0$.}
We call pointwise limits of subsequences of $\bigl\{\frac{1}{\beta}\log u_{\beta A}\bigr\}_{\beta\in(1,+\infty)}$, $\bigl\{l_{\beta}^\mu\bigr\}_{\beta\in(1,+\infty)}$, and $\bigl\{l_{\beta}^m\bigr\}_{\beta\in(1,+\infty)}$, as $\beta\to+\infty$,  ``logarithmic-type zero-temperature limits''.

\begin{thmx}[Characterization of the zero-temperature large deviation principle]\label{t: equi conditions}
Let $T\:X\to X$ be a transitive expanding covering map on a compact metric space $(X,d)$, and $A\: X\to \R$ be $\alpha$-H\"{o}lder continuous with $\alpha \in (0,1]$. If the family of equilibrium states $\{\mu_{\beta A}\}_{\beta\in(1,+\infty)}$ satisfies the large deviation principle as $\beta\to+\infty$, then the following hold:
\begin{enumerate}[label=\rm{(\roman*)}]	
\smallskip
\item $\bigl\{ \widetilde{\beta A} \big/ \beta\bigr\}_{\beta\in(1,+\infty)}$ uniformly converges as $\beta\to+\infty$.
\smallskip
\item $\bigl\{\frac{1}{\beta}\log u_{\beta A}\bigr\}_{\beta\in(1,+\infty)}$ uniformly converges as $\beta\to+\infty$.
\smallskip\item $\bigl\{l_{\beta}^m\bigr\}_{\beta\in(1,+\infty)}$ uniformly converges on every compact subset of $C(X,\R)$ as $\beta\to+\infty$.
\end{enumerate}

Conversely, if both statements~{\rm(ii)} and~{\rm(iii)} are true, then $\{\mu_{\beta A}\}_{\beta\in(1,+\infty)}$ satisfies the large deviation principle with rate function $-(v\otimes b)$ as $\beta\to+\infty$, where $v\in C(X,\R)$ is the limit function in {\rm(ii)} and $b\in\density(X)$ is the density of the tropical linear functional, which is the limit functional in {\rm(iii)}.
\end{thmx}

For the definition of tropical linear functionals and their densities, see Definitions~\ref{d:max-plus linear} and~\ref{d:tropical density}. A version of Theorem~\ref{t: equi conditions} was proved in \cite{Me18} for full shifts via methods depending crucially on symbolic dynamics. Our proofs of Theorems~\ref{t: uniquely maximizing implies the large deviation principle} and \ref{t: equi conditions} rely on developing a tropical analog of the thermodynamic formalism theory, with Theorem~\ref{t:existence and generic uniqueness} being its main component.

\begin{thmx}[Existence and generic uniqueness of tropical eigenfunctions and tropical eigendensities]\label{t:existence and generic uniqueness}
Let $T\:X\rightarrow X$ be an open continuous distance-expanding map on a compact metric space $(X,d)$, and $A \in \aholder$ with $\alpha \in (0,1]$. Then the following hold:
\begin{enumerate}[label=\rm{(\roman*)}]	
	\smallskip
\item {\rm(Existence of tropical eigenfunction.)} For each $u\in C(X,\R)$, define
\begin{equation*}
v_u(x)\coloneqq\limsup_{n\to+\infty}\cL_{\toverline{A}}^n(u)(x)
\end{equation*}
for each $x$ in $X$, where $\toverline{A}\=A-\energy$. Then $v_u$ is in $C(X,\Rminus)$ and is a tropical eigenfunction of $\cL_A$ associated with eigenvalue $\energy$. If $v_{\tmif}(x)=-\infty$ for some $x\in X$, then $\cL_A$ has no tropical eigenfunction in $C(X,\R)$. If $v_{\tmif}\in C(X,\R)$, then $v_u$ is in $\aholder$ for all $u\in C(X,\R)$. Moreover, if $T$ is transitive, then $v_{\tmif}$ is in $\aholder$.
\smallskip
\item {\rm(Generic uniqueness of tropical eigenfunction.)} For a generic potential $A$ in $\aholder$, $\es{\cL_A}{C(X,\Rminus)}{\energy}$ is of tropical dimension one.
\smallskip
\item {\rm(Existence of tropical eigendensity.)} There exists a tropical eigendensity of $\cL_{A}^{\ostar}$ associated with eigenvalue $\energy$.
\smallskip
\item {\rm(Generic uniqueness of tropical eigendensity.)} For a generic potential $A$ in $\aholder$, $\es{\cL_A^{\ostar}}{\density(X)}{\energy}$ is of tropical dimension one.
\smallskip
\item {\rm(Uniqueness of tropical eigenvalue.)} The number $\energy$ is the maximal tropical eigenvalue of $\cL_A$ (resp.~$\cL_A^{\ostar}$) on $C(X,\Rminus)$ (resp.~$\density(X)$), and if $T$ is transitive, then it is the unique tropical eigenvalue of $\cL_A$ (resp.~$\cL_A^{\ostar}$) on $C(X,\Rminus)$ (resp.~$\density(X)$).
 \smallskip
\end{enumerate}
\end{thmx}

Here $\energy$ is defined in (\ref{eq: defi of energy}) and $\aholder$ is the space of real-valued $\alpha$-H\"{o}lder continuous functions on $(X,d)$. We give original proofs of (i) and (iii) following the correspondence between thermodynamic formalism and its tropical counterpart. For (v), we remark that $\cL_A$ may have other tropical eigenvalues on $C(X,\Rminus)$ when $T$ is not transitive. This corresponds to the fact that some eigenfunction of the Ruelle operator may take $0$ and the Ruelle operator may have more than one positive eigenvalue when $T$ is not transitive. 

For (iii), recall that the existence of the eigenmeasure of $\cR_{\varphi}^*$ follows from the Schauder--Tychonoff fixed point theorem. Similarly, we use the tropical completion of $C(X,\R)$ and apply a version of Perron's method (see Proposition~\ref{p:existence of eigenmeasure}). The other ingredient for (iii) is the existence of a continuous real-valued sub-action, known as the Ma\~{n}\'{e} lemma.

In order to study the generic uniqueness of tropical eigenfunctions and tropical eigendensities, we establish the following representations in terms of the Aubry set and the Ma\~{n}\'{e} potential (see Definitions~\ref{d:Aubry set} and \ref{d:Mane potential}, respectively).
\begin{thmx}[Representation of tropical eigenfunctions and eigendensities]\label{t:representation}
	Let $T\:X\rightarrow X$ be an open continuous distance-expanding map on a compact metric space $(X,d)$, and $A\:X\to\R$ be $\alpha$-H\"{o}lder continuous with $\alpha \in (0,1]$. Let $\phi_A(\cdot,\cdot) \: X\times X\rightarrow\Rminus$ be the Ma\~{n}\'{e} potential associated with $A$, and $\Omega_A$ be the Aubry set with respect to $A$. Then the following hold:
\begin{enumerate}[label=\rm{(\roman*)}]	
	\smallskip
	\item Suppose $v$ is in $\es{\cL_A}{C(X,\Rminus)}{\energy}$. Then
	\begin{equation*}
	v(y)=\oplusu{x\in\Omega_A}(v(x)\otimes\phi_A(x,y))
	\end{equation*}
	for every $y$ in $X$. Moreover, for all $c\in C(\Omega_A,\Rminus)$, $f_c(\cdot)\=\oplusu{x\in\Omega_A}(c(x)\otimes\phi_A(x,\cdot))$ is in $\es{\cL_A}{C(X,\Rminus)}{\energy}$. In addition, there exists a unique $c\in C(\Omega_A,\Rminus)$ such that $v=f_c$ and $c(x)\otimes \phi_A(x,y)\leq c(y)$ for all $x,\, y\in\Omega_A$.
    
	\smallskip
	\item \begin{enumerate}
        \item[\rm{(a)}] Suppose $b$ is in $\es{\cL_A^{\ostar}}{\density(X)}{\energy}$. Then $b(\cdot)$ is equivalent to $\oplusu{y\in\Omega_A}(\phi_A(\cdot,y)\otimes b(y))$. Moreover, if $b\neq\tinftyf$, then $b(x)=\oplusu{y\in\Omega_A}(\phi_A(x,y)\otimes b(y))$ for all $x\in X$.

    \item[\rm{(b)}] For all $c\in \Rm^{\Omega_A}$, we have that $\oplusu{y\in\Omega_A}(\phi_A(\cdot,y)\otimes c(y))$ is equivalent to some $b'$ in $\es{\cL_A^{\ostar}}{\density(X)}{\energy}$.
    \end{enumerate}
    
	\smallskip
	\item The entries of $\{\phi_A(\cdot,y)\}_{y\in\Omega_A}$ are tropical eigendensities of $\cL_A^{\ostar}$ associated with eigenvalue $\energy$, and the entries of $\{\phi_A(x,\cdot)\}_{x\in\Omega_A}$ are tropical eigenfunctions of $\cL_A$ in $C(X,\Rminus)$ associated with eigenvalue $\energy$. If $A$ is uniquely maximizing, then the entries of  $\{\phi_A(x,\cdot)\}_{x\in\Omega_A}$ (resp.\ $\{\phi_A(\cdot,y)\}_{y\in\Omega_A}$) are the same up to a tropical multiplicative constant. 
\end{enumerate}
\end{thmx}

While the parts of Theorem~\ref{t:representation} on eigenfunctions in the context of subshifts of finite type have appeared in \cite[Propositions~6.2 and 6.7]{Ga17}, we take the opportunity to generalize it to uniformly expanding systems without assuming transitivity and establish novel counterparts for eigendensities. It is worth mentioning that our proof of Theorem~\ref{t:representation}~(ii) relies on the constructive result Corollary~\ref{c:more_of_constructions} (i.e., Theorem~\ref{t:existence and generic uniqueness}~(i)) for tropical eigenfunctions.

The logarithmic-type zero-temperature limits can be seen as a bridge connecting thermodynamic formalism objects and their tropical counterparts.

Recall that a family of real-valued continuous functions on $X$ (resp.\ functionals on $C(X,\R)$) is a \emph{normal family} if, for every sequence of functions (resp.\ functionals) of this family, there exists a subsequence that is uniformly converging on every compact subset of $X$ (resp.\ $C(X,\R)$).

\begin{thmx}[Logarithmic-type zero-temperature limits of eigenfunctions and eigendensities]\label{t:zero-temperature limit}
	Let $T\:X\rightarrow X$ be a transitive expanding covering map on a compact metric space $(X,d)$, and $A\:X\to\R$ be $\alpha$-H\"{o}lder continuous with $\alpha \in (0,1]$. Then the following hold:
	\begin{enumerate}[label=\rm{(\roman*)}]	
		\smallskip
		\item The family $\bigl\{l_{\beta}^m|_{C(X,\R)}\bigr\}_{\beta\in(1,+\infty)}$ is normal, and the (pointwise) limit of every convergent subsequence $\bigl\{l_{\beta_n}^m\bigr\}_{n\in\N}$ with $\beta_n\to+\infty$ as $n\to+\infty$ is a tropically continuous linear functional whose density in $\density(X)$ is in $\es{\cL_A^{\ostar}}{\density(X)}{\energy}$.
  
		\smallskip
        \item The family $\bigl\{l_{\beta}^{\mu}|_{C(X,\R)}\bigr\}_{\beta\in(1,+\infty)}$ is normal, and the (pointwise) limit of every convergent subsequence $\bigl\{l_{\beta_n}^{\mu}\bigr\}_{n\in\N}$ with $\beta_n\to +\infty$ as $n\to+\infty$ is a tropically continuous linear functional whose density in $\density(X)$ is the tropical product $v \otimes b$ of some $v \in \es{\cL_A}{C(X,\R)}{\energy}$ and some $b \in \es{\cL_A^{\ostar}}{\density(X)}{\energy}$.
	\end{enumerate}
\end{thmx}
Here $l_{\beta}^{m}$ and $l_{\beta}^{\mu}$, whose domains are $C(X,\R)\cup\{\tinftyf,\,\taif\}$, are defined in (\ref{eq:defi of functional at positive temperature}), and $\density(X)$ is defined in (\ref{eq: defi of the density space}). For the definition of tropically continuous linear functionals and their densities, see Definitions~\ref{d:max-plus linear} and~\ref{d:tropical density}. Indeed, every tropical linear functional on $C(X,\R)\cup\{\tinftyf,\,\taif\}$ has a unique density in $\density(X)$ (see Remark~\ref{r:density and functional}).

Theorem~\ref{t:zero-temperature limit}, together with Theorem~\ref{t:representation}, leads to the corresponding rate functions for large deviation principles. In the context of subshifts of finite type, Baraviera, Lopes, and Thieullen \cite{BLT06} established a large deviation principle for $\{\mu_{\beta A}\}_{\beta\in(1,+\infty)}$ under the assumption that the potential $A$ is uniquely maximizing using the ``dual shift'' technique. Our study of subsequential limits does not rely on this technique and explores the situation in which the potential may not be uniquely maximizing.

\begin{thmx}[Logarithmic-type zero-temperature limits of equilibrium states and normalized potentials]\label{t:Log type limit for equilibrium states and altered potentials}
	Let $T\:X\rightarrow X$ be a transitive expanding covering map on a compact metric space $(X,d)$, and $A\:X\to\R$ be $\alpha$-H\"{o}lder continuous with $\alpha \in (0,1]$. Then the following hold:
	\begin{enumerate}[label=\rm{(\roman*)}]	
		\smallskip
		\item The two families $\bigl\{\widetilde{\beta A}\big /\beta\bigr\}_{\beta\in(1,+\infty)}$ and $\bigl\{l_{\beta}^\mu|_{C(X,\R)}\bigr\}_{\beta\in(1,+\infty)}$ are normal.
        \smallskip
		\item Suppose $\bigl\{\frac{1}{\beta_k}\log u_{\beta_kA}\bigr\}_{n\in\N}$ pointwise converges to a function $v\in C(X,\R)$ with $\beta_k\to+\infty$ as $k\to+\infty$. Then $\bigl\{\widetilde{\beta_k A}\big/\beta_k \bigr\}_{k\in\N}$ pointwise converges to $A-\energy+v-v\circ T$ as $k\to+\infty$.
        \smallskip
        \item Suppose $\bigl\{\widetilde{\beta_k A}\big/\beta_k\bigr\}_{k\in\N}$ pointwise converges to a function $\twidetilde{A}\in C(X,\R)$ and $\bigl\{l_{\beta_k}^\mu \bigr\}_{k\in\N}$ pointwise converges to a function $l\:C(X,\R)\cup\{\tinftyf, \, \taif\}\rightarrow\Rm$ with $\beta_k\to+\infty$ as $k\to+\infty$. Then $l$ is a tropically continuous linear functional whose density $b$ in $\density(X)$ satisfies $\cL_{\twidetilde{A}}^{\ostar}(b)=b$.
	\end{enumerate}
\end{thmx}
It is well known that the limit $v$ in Theorem~\ref{t:Log type limit for equilibrium states and altered potentials}~(ii) is in $\es{\cL_A}{C(X, R)}{\energy}$ (cf.\ Lemma~\ref{l: zero tem limits of eigenfunctions}). 

\subsection{Comments and the organization of the article}
As already mentioned, we consider the analysis of the Bousch operator for some H\"{o}lder potential without the surjectivity and transitivity assumptions. It is worth noting that $\cL_A$ may only have tropical eigenfunctions in $C(X,\Rminus)$ as we see in Theorem~\ref{t:existence and generic uniqueness}~(i), while a real-valued (H\"{o}lder continuous) sub-action always exists (Proposition~\ref{p: mane lemma}). The difference lies in the constructions of the two functions. A tropical eigenfunction of $\cL_A$ in $C(X,\Rminus)$ constructed in Theorem~\ref{t:existence and generic uniqueness}~(i) is the limit superior of $\cL_A^n(u)$ as $n\to+\infty$ for some $u\in C^{0,\alpha}(X,d)$ while the sub-action constructed in the proof of the Ma\~{n}\'{e} Lemma in Proposition~\ref{p: mane lemma} is the supremum of $\{\cL_A^n(\tmif)\}_{n\in\N_0}$.

A key feature of the proofs of the above theorems is the application of formalism via tropical algebra and tropical analysis, making the proofs more applicable to other systems. For Theorem~\ref{t:existence and generic uniqueness}~(iii), tropical completeness facilitates an application of a version of Perron's method. For Theorems~\ref{t:zero-temperature limit} and~\ref{t:Log type limit for equilibrium states and altered potentials}, the subsequential limits of $l_{\beta}^m$ and $l_{\beta}^\mu$ as $\beta\to+\infty$ are tropically continuous linear functionals, and thus our tropical Riesz representation theorem (Proposition~\ref{p:dual space equals completion}) applies. Furthermore, it follows from the tropical Riesz representation theorem that the densities of the subsequential limits of $l_{\beta}^m$ are tropical eigendensities of $\cL_A^{\ostar}$. These densities are generally different, and their representations are given in Theorem~\ref{t:representation}. For Theorem~\ref{t:representation}~(ii), for all tropically continuous functionals $l$ and $u\in\aholder$, $l(u)$ is equal to $l(v_u)$, where $v_u$ is in $\es{\cL_A}{C(X,\Rminus)}{\energy}$ (cf.~Theorem~\ref{t:existence and generic uniqueness}~(i)), and thus Theorem~\ref{t:representation}~(i) applies.

The other main ingredient for demonstrating Theorem~\ref{t:representation}~(i) and~(ii) is the properties of the Aubry set and the Ma\~{n}\'{e} potential. Some of these results have already appeared in the context of expanding maps on the circle and subshifts of finite type (see e.g.~\cite{CLT01} and \cite{Ga17}), but since we need to add new ones (namely, Proposition~\ref{p:Mane_potential_properties}~(ii) and Lemma~\ref{l:mane potential gives eigenmeasures}) concerning tropical eigendensities of $\cL_A^{\ostar}$, we provide the new proofs in Subsection~\ref{ss:Mane_potential} and keep the others in Appendix~\ref{ss: Appendix A}. 

The outline of the article is as follows. In Section~\ref{s: tropical functional analysis}, we introduce and investigate tropical functional analysis notions and results, which include the definitions of tropical completeness, tropical dual spaces, and tropical measures. In Section~\ref{s: Bousch operator}, we focus on the Bousch operator for open continuous distance-expanding maps without the surjectivity and transitivity assumptions. In Subsection~\ref{ss:analysis}, we first give a constructive proof of Theorem~\ref{t:existence and generic uniqueness}~(i) (i.e., Proposition~\ref{p:main_construction} and Corollary~\ref{c:more_of_constructions}) and then define the tropical adjoint Bousch operator and prove Theorem~\ref{t:existence and generic uniqueness}~(iii). In Subsection~\ref{ss:Mane_potential}, we recall the concepts of the Aubry set and the Ma\~{n}\'{e} potential, and discover that tropical eigendensities of the tropical adjoint Bousch operator can be represented by the Aubry set and the Ma\~{n}\'{e} potential. Theorem~\ref{t:representation} is proved in Subsection~\ref{ss: proof of D}. A sufficient condition for the uniqueness of tropical eigenfunctions and tropical eigendensities is discussed in Subsection~\ref{ss: sufficient condition for uniqueness}. The proof of Theorem~\ref{t:existence and generic uniqueness} is concluded in Subsection~\ref{ss: proof of C}. In Section~\ref{s: zero-temperature limits}, we investigate the logarithmic-type zero-temperature limits and establish Theorems~\ref{t:zero-temperature limit} and~\ref{t:Log type limit for equilibrium states and altered potentials} for transitive expanding covering maps. By considering the restriction of the map on the set of nonwandering points, we generalize Theorem~\ref{t:zero-temperature limit} to establish Theorem~\ref{t: uniquely maximizing implies the large deviation principle}. Theorem~\ref{t: equi conditions} is established as a corollary of Theorem~\ref{t:Log type limit for equilibrium states and altered potentials}.  We keep the proofs that could be familiar to experts in Appendix~\ref{ss: Appendix A}.

\medskip
\noindent{\bf Acknowledgments.} 
The authors thank Juan Rivera-Letelier for suggestions on references and Mikhail Yu.~Lyubich for his suggestions after listening to the main content of the article during the summer school at Urgench State University, Uzbekistan, in August 2023. The authors were partially supported by Beijing Natural Science Foundation (JQ25001, QY24007) and National Natural Science Foundation of China (12471083, 12101017, 12090010, and 12090015).

\section{Tropical functional analysis}\label{s: tropical functional analysis}
In this section, we introduce the notions of tropical completeness, tropical dual spaces, and tropical measures, and prove general tropical functional analysis results.

\subsection{Tropical spaces and tropical dual spaces} Tropical spaces and tropical dual spaces, especially the tropical completion and tropical dual space of $C(X,\R)$, are studied in this subsection. 

We introduce the following notion of tropical completeness and completion, which is different from the notion of ``normal completion'' in \cite[pp.~703--704]{LMS01}.
\begin{definition}\label{d:completion}
	Let $S\subseteq\Rm^X$ be a set of $\Rm$-valued functions on a set $X$. The set $S$ is said to be \emph{tropically complete} if for each subset $U\subseteq S$, $\oplusu{u\in U}u$ is in $S$. The \emph{tropical completion} of $S$, denoted by $\twidehat{S}$, is defined to be the intersection of all tropically complete subsets of $\Rm^X$ containing $S$.
\end{definition}

In particular, we denote the tropical completion of $C(X,\R)$ by $\twidehat{C}(X,\R)$. Recall that $\oplusu{u\in U}u$ is the pointwise supremum, and a function $u \: X\rightarrow\Rm$ on a topological space $X$ is \emph{upper} (resp.\ \emph{lower}) \emph{semi-continuous} if for every $b\in\Rm$, the set $\{x\in X: u(x)<b\}$ (resp.\ $\{x\in X: u(x)>b\}$) is open.
\begin{rem}\label{r: about completion}
    Let $X$ be a compact metric space. The ``normal completion'' of ${C(X,\R)}$ in \cite{LMS01} is $\operatorname{LSC}(X)\cup\{\tinftyf,\,\taif\}$ (i.e., real-valued lower semi-continuous functions and the two constant infinity functions). In contrast, $\twidehat{C}(X,\R)$ contains lower semi-continuous functions from $X$ to $\R\cup\{+\infty\}$ that take $+\infty$ somewhere but are not equal to $+\infty$ everywhere. Indeed, different kinds of completions can be seen as being related to analogs of different norms on (conventional) linear spaces.

    It is straightforward to check that for all $S\subseteq\Rm^X$,
    \begin{equation}\label{eq: formula of completion}
        \twidehat{S}=\Bigl\{\oplusu{u\in U}u:U\subseteq S\Bigr\}.
    \end{equation}
\end{rem}

\begin{lemma}\label{l: lower semi-continuous functions are oplus of continuous functions}
Let $X$ be a compact metric space. If $g \:X \rightarrow \R\cup\{+\infty\}$ is lower semi-continuous, then there exists a countably infinite family $\{u_i\in C(X,\R): i\in I\}$ such that $g=\oplusu{i\in I}u_i$. 
\end{lemma}

\begin{proof}
It is straightforward to check that every lower semi-continuous function $g\: X\to\R\cup\{+\infty\}$ on the compact metric space $X$ has a lower bound $M\in\R$. It then follows from \cite[Theorem~23.19]{Ke95} that there exists a nondecreasing sequence $\{f_n\}_{n\in\N}$ in $C(X,\R)$ such that $g=\oplusu{n\in\N}f_n$.
\end{proof}

Now we can give a description of $\twidehat{C}(X,\R)$.
\begin{prop} \label{p: completion}
    Let $X$ be a compact metric space. Then
    \begin{equation*}
        \twidehat{C}(X,\R)=\{g \: X \rightarrow \R\cup\{+\infty\}: g \text{ lower semi-continuous}\}\cup\{\taif\}.
    \end{equation*}
\end{prop}

\begin{proof}
	Denote $W \= \{g \: X \rightarrow \R\cup\{+\infty\}: g \text{ lower semi-continuous}\}\cup\{\taif\}$, and we need to show that $W=\twidehat{C}(X,\R)$. By Definition~\ref{d:completion}, it suffices to check that $W$ is tropically complete and each tropically complete subset of $\Rm^{X}$ containing $C(X,\R)$ contains $W$. 

	We first prove the tropical completeness of $W$. For each family $\{g_v\in W: v\in V\}$, $g_v$ is lower semi-continuous for all $v$ in the index set $V$. Thus, $\oplusu{v\in V}g_v$ is lower semi-continuous. Moreover, if there exists an index $v_0$ in $V$ such that $g_{v_0}\neq\taif$, then $g_{v_0}\: X\to\R\cup\{+\infty\}$ since $g_{v_0}\in W$ and it follows that
	$\oplusu{v\in V}g_v \: X \rightarrow \R\cup\{+\infty\}$. Otherwise, $g_v=\taif$ for all $v$ in $V$, and consequently $\oplusu{v\in V}g_v=\taif\in W$. We conclude that $\oplusu{v\in V}g_v\in W$. The tropical completeness of $W$ is now verified.
	
	Now we show that every tropically complete set $W'\subseteq\Rm^{X}$ containing $C(X,\R)$ contains $W$. Note that $\taif=\oplusu{u\in\emptyset}u\in W'$. By Lemma~\ref{l: lower semi-continuous functions are oplus of continuous functions}, for each $g\in W\smallsetminus\{\taif\}$, there exists a family $\{u_i\in C(X,\R): i\in I\}\subseteq W'$ such that $g=\oplusu{i\in I}u_i$. Since $W'$ is tropically complete, it follows that $g\in W'$. We conclude that $W\subseteq W'$.
\end{proof}

Recall that for a semiring $R$, an \emph{$R$-semimodule} is a set $S$ equipped with a binary operation $+\: S\times S \rightarrow  S$ and a map $\times\: R\times S \rightarrow  S$, with the operations for the semiring $R$ also denoted by $+$ and $\times$, provided that the following axioms are satisfied:
\begin{enumerate}[label=\rm{(\roman*)}]	
	\smallskip
	\item $(a+ b)+ c=a+(b+ c)$ for all $a,b,c\in S$.
	\smallskip
	\item $a+ b=b+ a$ for all $a,b\in S$.
	\smallskip
	\item There is an element $0_S$ in $S$ such that $0_S+ a=a$ for all $a\in S$.
	\smallskip
	\item $\lambda\times(a+ b)=(\lambda\times a)+(\lambda\times b)$ for all $\lambda\in R$ and $a,b\in S$.
	\smallskip
	\item $(\lambda_1+\lambda_2)\times a=(\lambda_1\times a)+(\lambda_2\times a)$ for all $\lambda_1,\lambda_2\in R$ and $a\in S$.
	\smallskip
	\item $0_R\times a=0_S$ for all $a\in S$.
	\smallskip
	\item $1_R\times a=a$ for all $a\in S$.
	\smallskip
	\item $(\lambda_1\times\lambda_2)\times a=\lambda_1\times(\lambda_2\times a)$ for all $\lambda_1,\lambda_2\in R$ and $a\in S$.
\end{enumerate}
A map $\cL\:S_1\to S_2$ between the two $R$-semimodules $S_1$ and $S_2$ is called a \emph{semimodule homomorphism} if $\cL(a+b)=\cL(a)+\cL(b)$ and $\cL(\lambda\times a)=\lambda\times\cL(a)$ for all $a,\,b\in S_1$ and $\lambda\in R$. A subset $M$ of an $R$-semimodule $S$ is called a \emph{subsemimodule} if it is closed under the addition and scalar multiplication of $S$.

In this article, we focus on the operation pair $(\oplus,\otimes)$ for $\Rm$. It is straightforward to check that $C(X,\R)\cup\{\tinftyf,\,\taif\}$ and $\twidehat{C}(X,\R)$ are subsemimodules of the $\Rm$-semimodule $\Rm^X$. 

Now we define tropically continuous linear maps using our notion of tropical completion. Note that it follows from (\ref{eq: formula of completion}) and (\ref{eq: exchangability of numbers}) that if $V$ is a subsemimodule of $\Rm^X$, then $\twidehat{V}$ is also a subsemimodule of $\Rm^X$.

\begin{definition}\label{d:max-plus linear}
	Let $X$ be a set and $V, \, W$ be two $\Rm$-semimodules.
\begin{enumerate}[label=\rm{(\roman*)}]	
\smallskip
\item A map $\cL\: V\to W$ is \emph{tropical linear} if it is a semimodule homomorphism.

\smallskip
\item A map $\cL\: V\to\Rm$ is called a \emph{tropical functional}.

\smallskip
\item Assume that $V,\,W$ are two subsemimodules of $\Rm^X$. We call a tropical linear map $\cL\: V\to W$ \emph{tropically continuous} if $\cL$ can be extended uniquely to a tropical linear map $\cL \:\twidehat{V} \rightarrow \twidehat{W}$ satisfying
\begin{equation*}
	\cL\bigl(\oplusu{u\in U}u\bigr)=\oplusu{u\in U}\cL(u)
\end{equation*} 
for every subset $U\subseteq\twidehat{V}$.
\end{enumerate}
When more than one property or term above is mentioned, we often adopt the convention to only use one ``tropical(ly)'', e.g.\ a tropically continuous linear functional and a tropical linear functional.
\end{definition}

Note that we adopt the convention to use an underscore to indicate a tropical object, e.g.\ tropical reals $\Rminus$, tropical eigenspaces $\es{\cL}{V}{\lambda}$, tropical dimension $\tdim$, tropical completion $\twidehat{V}$, and tropical integrals $\tint$ defined below.

Let $V, \,W$ be two subsemimodules of $\Rm^X$, and $\cL$ be a tropical linear map from $V$ to $W$. Suppose $u,\,v\in V$ and $u\fless v$, i.e., $u\oplus v=v$. Then $\cL(v)=\cL(u\oplus v)=\cL(u)\oplus\cL(v)$, i.e., $\cL(u)\fless\cL(v)$. Thus, $\cL$ is an order-preserving map. We record this well-known fact below.

\begin{lemma}\label{l:tropical linear map preserves order}
Let $X$ be a set, $V, \, W\subseteq\Rm^X$ be two subsemimodules of $\Rm^X$, and $\cL$ be a tropical linear map from $V$ to $W$. Then $\cL(u)\fless\cL(v)$ for all $u,\,v\in V$ with $u\fless v$.
\end{lemma}

We discover the following interesting result, which plays a role in the discussion on logarithmic-type zero-temperature limits.
\begin{prop}\label{p:compactness implies continuity}
Let $(X,d)$ be a compact metric space. Then every tropical linear functional $\cL \: C(X,\R)\cup\{\tinftyf,\,\taif\} \rightarrow \Rm$ is tropically continuous.
\end{prop}
\begin{proof}
We need to prove that $\cL$ can be uniquely extended to some $\widehat{\cL} \:\twidehat{C}(X,\R) \rightarrow \Rm$ satisfying $\widehat{\cL}\bigl(\oplusu{u\in U}u\bigr)=\oplusu{u\in U}\widehat{\cL}(u)$ for each subset $U\subseteq\twidehat{C}(X,\R)$ and $\widehat{\cL}(\lambda\otimes g)=\lambda\otimes\widehat{\cL}(g)$ for all $g\in\twidehat{C}(X,\R)$ and $\lambda\in\Rm$. We first construct a tropically continuous linear extension and then verify the uniqueness.

We consider 
\begin{equation}\label{defi: construction of tropically continuous extension}
\widehat{\cL}(g) \= \oplusu{u\in C(X,\R) : u\fless g}\cL(u)
\end{equation}
for all $g\in\twidehat{C}(X,\R)$. For $g\in C(X,\R)$, Lemma~\ref{l:tropical linear map preserves order} implies $\cL(u)\leq\cL(g)$ for all $u\in C(X,\R)$ with $u\fless g$. Note $g\fless g$. Thus, 
\begin{equation}\label{eq: construction of extension of L}
	\widehat{\cL}(g)=\oplusu{u\in C(X,\R) : u\fless g}\cL(u)=\cL(g).
\end{equation}
For $g\in\{\tinftyf,\, \taif\}$, the identities in (\ref{eq: construction of extension of L}) hold since $\tinftyf$ and $\taif$ are the maximal and minimal elements of $C(X,\R)\cup\{\tinftyf,\,\taif\}$, respectively.
We conclude that $\widehat{\cL}$ is an extension of $\cL$. 

Next, we check the linearity and continuity of $\widehat{\cL}$.
Fix $g\in\twidehat{C}(X,\R)$. For $\lambda\in\R$,
\begin{equation}\label{eq:linearity of extension of L}
	\widehat{\cL}(\lambda\otimes g)=\oplusu{u\in C(X,\R) : u\fless\lambda\otimes g}\cL(u)=\oplusu{u\in C(X,\R) : u\fless g}\cL(\lambda\otimes u)=\lambda\otimes\widehat{\cL}(g).
\end{equation}
For $\lambda\in\{+\infty,\,-\infty\}$, it is straightforward to check that $\widehat{\cL}(\lambda\otimes g)=\lambda\otimes\widehat{\cL}(g)$.

To prove $\widehat{\cL}\bigl(\oplusu{u\in U}u\bigr)=\oplusu{u\in U}\widehat{\cL}(u)$ for a fixed arbitrary $U\subseteq\twidehat{C}(X,\R)$, denote $g\coloneqq\oplusu{u\in U}u$. By (\ref{defi: construction of tropically continuous extension}), for $u\in U$, $u\fless g$ implies $\widehat{\cL}(u)\leq\widehat{\cL}(g)$. We conclude that $\oplusu{u\in U}\widehat{\cL}(u)\leq\widehat{\cL}(g)$. Now we need to show $\widehat{\cL}(g)\leq\oplusu{u\in U}\widehat{\cL}(u)$. By (\ref{defi: construction of tropically continuous extension}), it suffices to prove that for all $v$ in $C(X,\R)$ satisfying $v\fless g$, $\cL(v)\leq\oplusu{u\in U}\widehat{\cL}(u)$.

Fix $\epsilon>0$ and $v\in C(X,\R)$ with $v\fless g=\oplusu{u\in U}u$. For every $x\in X$, there exists $u_x\in U$ such that $u_x(x)>v(x)-\epsilon$. Since $u_x\in\twidehat{C}(X,\R)$, by Lemma~\ref{l: lower semi-continuous functions are oplus of continuous functions} there exists $w_x\in C(X,\R)$ such that $w_x\fless u_x$ and $w_x(x)>v(x)-\epsilon$. Now that $w_x$ and $v$ are both in $C(X,\R)$, $w_x>v-2\epsilon$ holds in some neighborhood $B(x,r_x)$. Thus, $\bigcup_{x\in X}B(x,r_x)$ forms an open cover of $X$ and the compactness of $X$ implies that there is a finite cover $X=\bigcup_{i=1}^nB(x_i,r_{x_i})$. We conclude that $v-2\epsilon\fless\oplusu{1\leq i\leq n}w_{x_i}$. 

Note that $\oplusu{1\leq i\leq n}w_{x_i}$ is in $C(X,\R)$ since $w_{x_i}$ is in $C(X,\R)$. Thus it follows from the tropical linearity of $\cL$, $w_x\fless u_x$, and $u_x\in U$ that 
\begin{equation*} 
\cL(v)-2\epsilon\leq\cL\bigl(\oplusu{1\leq i\leq n}w_{x_i}\bigr)=\oplusu{1\leq i\leq n}\cL(w_{x_i})\leq\oplusu{1\leq i\leq n}\widehat{\cL}(u_{x_i})\leq\oplusu{u\in U}\widehat{\cL}(u).
\end{equation*}
Letting $\epsilon$ tend to $0$, we conclude that $\cL(v)\leq\oplusu{u\in U}\widehat{\cL}(u)$, and the tropical continuity follows.
Finally, we verify the uniqueness of the extension. Let $\widetilde{\cL}$ be an extension of $\cL$ satisfying $\widetilde{\cL}\bigl(\mathop{\oplus}\limits_{u\in U}u\bigr)=\mathop{\oplus}\limits_{u\in U}\widetilde{\cL}(u)$ for each subset $U\subseteq\twidehat{C}(X,\R)$ and $\widetilde{\cL}(\lambda\otimes g)=\lambda\otimes\widetilde{\cL}(g)$ for all $g\in\twidehat{C}(X,\R)$ and $\lambda\in\Rm$. 

Fix $g\in\twidehat{C}(X,\R)$, consider $U\coloneqq\{u\in C(X,\R):u\fless g\}$. By Lemma~\ref{l: lower semi-continuous functions are oplus of continuous functions} and Proposition~\ref{p: completion}, we have $\oplusu{u\in U}u=g$. Since $\widetilde{\cL}$ is a tropically continuous linear extension of $\cL$, we see that $\widetilde{\cL}(g)=\oplusu{u\in U}\widetilde{\cL}(u)=\oplusu{u\in C(X,\R) : u\fless g}\cL(u)=\widehat{\cL}(g)$. Uniqueness is now verified.
\end{proof}
\begin{remark}
This proposition suggests that there is no difference between tropical linear functionals on $C(X,\R)\cup\{\tinftyf,\,\taif\}$ and tropically continuous linear functionals on $\twidehat{C}(X,\R)$ for compact metric spaces $X$. Furthermore, we observe from the above proof that we can replace ``every subset $U\subseteq\twidehat{V}$'' with ``every countable subset $U\subseteq\twidehat{V}$'' in Definition~\ref{d:max-plus linear}~(iii) for compact metric spaces.
\end{remark}

\begin{definition}\label{d:dual space}
	Let $X$ be a set and $V$ be a subsemimodule of $\Rm^X$. The \emph{tropical dual space} $V^{\ostar}$ of $V$ is the space consisting of all tropically continuous linear functionals from $\twidehat{V}$ to $\Rm$.
\end{definition}
\begin{rem}\label{r: dual space are semimodules}
The tropical dual space $V^{\ostar}$ of $V$ becomes an $\Rm$-semimodule under the following operations:
\begin{enumerate}
    \smallskip
    \item $(l_1\oplus l_2)(u)\=l_1(u)\oplus l_2(u)$ for all $l_1, \, l_2\in V^{\ostar}$ and $u\in V$;
    \smallskip
    \item $(c\otimes l)(u)\=c\otimes l(u)$ for all $c\in\Rm$, $l\in V^{\ostar}$, and $u\in V$.
\end{enumerate}  
\end{rem}

In the remainder of this subsection, we discuss the connection between the tropical completion and the tropical dual space.

Denote $S\=C(X,\R)\cup\{\tinftyf,\,\taif\}$. Note that $\twidehat{S}=\twidehat{C}(X,\R)$. We thus denote $C(X,\R)^{\ostar}\= S^{\ostar}$ despite the fact that $C(X,\R)$ is not a subsemimodule of $\Rm^X$.

The following definitions are important for the representation of the tropical dual space, as well as some of the analysis in Section~\ref{s: zero-temperature limits}. Similar notions also appear in \cite{CGQ04}, \cite{LMS01}, and \cite{HS09}.\footnote{We remark that \cite{CGQ04} uses $\slash$ to denote the $\scalebox{0.9}{\odiv}$ operation, following the convention in residuation theory. We reserve $\oddiv$ for the pointwise subtraction operation as in \cite{HS09}.}
\begin{definition}\label{d:odiv}
    Let $X$ be a set. For $u,v\in\Rm^{X}$, we define 
	\begin{equation*}
	u \odiv v\=\oplusu{\lambda\otimes v\leq u\atop\lambda\in\Rm}\lambda\in\Rm\quad\text{ and } \quad
	u\oddiv v\=u\otimes(-v)\in\Rm^X.
	\end{equation*}
\end{definition}

Recall that $\oplusu{x\in\emptyset}x=-\infty$. In particular, $a\oddiv b = a\otimes(-b) = a \odiv b$ for $a,\, b \in \Rm$.

\begin{definition}\label{d: tropical linear functional represented by -v}
Let $X$ be a set. For all $u,\, f\in\Rm^X$, denote 
\begin{equation*}
l_u(f) \= -(u\odiv f).
\end{equation*}
\end{definition}

\begin{rem}\label{r: form of l_u}
 It is straightforward to check that 
 \begin{equation*}
     l_u(f)
     =\oplusu{x\in X}(f(x)\otimes(-u(x)))
     =\oplusu{x\in X}(f(x)\oddiv u(x))
 \end{equation*}
 for all $u,\,f\in\Rm^X$. It follows that $l_u\:\Rm^X\to\Rm$ is a tropically continuous linear functional for all $u\in\Rm^X$.
\end{rem}

Our choice of the notion of tropical completion facilitates the following form of \emph{tropical Riesz representation theorem} (cf.~Remark~\ref{r:density and functional}).
\begin{prop}\label{p:dual space equals completion}
    Let $X$ be a compact metric space. The restriction $l_v \: \twidehat{C}(X,\R)\rightarrow \Rm$ is a tropically continuous linear functional in $C(X,\R)^{\ostar}$ for each $v\in\twidehat{C}(X,\R)$, and the map $l_{\bullet}\:\twidehat{C}(X,\R) \rightarrow  C(X,\R)^{\ostar}$ given by $v\mapsto l_v$ is a bijection.
\end{prop}
\begin{proof}
    It is straightforward to check that $l_v\:\twidehat{C}(X,\R)\rightarrow\Rm$ is a tropically continuous linear functional for all $v\in\twidehat{C}(X,\R)$.
    
	Now we show that $l_{\bullet}$ is surjective and then verify its injectivity.
	
	Let $\cL \:\twidehat{C}(X,\R) \rightarrow \Rm$ be a tropically continuous linear functional. First we consider the case where $\cL(f)\in\{+\infty,\,-\infty\}$ for all $f\in\twidehat{C}(X,\R)$. If there exists $u\in C(X,\R)$ such that $\cL(u)=-\infty$, then $\cL(\tinftyf)=\cL((+\infty)\otimes u)=(+\infty)\otimes \cL(u)=(+\infty)\otimes(-\infty)=-\infty$. Note that $\tinftyf$ is the maximal element of $\twidehat{C}(X,\R)$. It follows that $\cL(f)=-\infty$ for all $f\in\twidehat{C}(X,\R)$. 
	If $\cL(u)=+\infty$ for all $u$ in $C(X,\R)$, then it follows from Proposition~\ref{p: completion} and Lemma~\ref{l: lower semi-continuous functions are oplus of continuous functions} that $\cL(g)=+\infty$ for all $g\in\twidehat{C}(X,\R)\smallsetminus\{\taif\}$ and $\cL(\taif)=(-\infty)\otimes\cL(\tmif)=-\infty$. We conclude that the only two tropical linear functionals for the first case are 
	\begin{equation*}
	\cL(f)=-\infty \text{ for all }f\in\twidehat{C}(X,\R)
	\quad
	\text{ and }\quad \cL(f)=\begin{cases}
		-\infty&\text{if}\, f=\taif;\\
		+\infty&\text{if}\, f\neq\taif.\\
	\end{cases}
	\end{equation*}
By Definition~\ref{d: tropical linear functional represented by -v}, the former is $l_{+\infty}$ and the latter is $l_{-\infty}$.
	
	Now we suppose that $\cL(f)\in\R$ for some $f\in\twidehat{C}(X,\R)$. Then the range of $\cL$ must contain $\R$ since $\cL$ is tropical linear. Consider $M\coloneqq\bigl\{f\in\twidehat{C}(X,\R):\cL(f)\leq 0\bigr\}$, and it follows that $\cL(M)\coloneqq\{\cL(f):f\in M\}$ also contains some real number. Since $\twidehat{C}(X,\R)$ is tropically complete, we consider $v\coloneqq\oplusu{f\in M}f\in\twidehat{C}(X,\R)$. 
	
	\smallskip
	\emph{Claim.} $\cL(v)=0$ and $\cL=l_v$. 
	\smallskip
	
	It follows from the tropical continuity of $\cL$ and the definition of $M$ that $\cL(v)=\oplusu{f\in M}\cL(f)\leq 0$. Recall that $\cL(M)\cap\R\neq\emptyset$ and consequently $\cL(v)\in\R$. The tropical linearity of $\cL$ then implies 
    \begin{equation*}
	\cL((-\cL(v))\otimes v)=(-\cL(v))\otimes \cL(v)=0,
	\end{equation*}
    and consequently $(-\cL(v))\otimes v\in M$. Recall that $v=\oplusu{f\in M}f$ is the maximal element of $M$, it follows that $(-\cL(v))\otimes v\fless v$. Thus, 
    \begin{equation*}
	0=(-\cL(v))\otimes \cL(v)=\cL((-\cL(v))\otimes v)\leq \cL(v)
    \end{equation*}
    and we conclude that $\cL(v)=0$.

	Similarly, for each $f\in\twidehat{C}(X,\R)$, $\cL((-\cL(f))\otimes f)=(-\cL(f))\otimes \cL(f)\leq 0$ implies that  $(-\cL(f))\otimes f\in M$. Since $v$ is the maximal element of $M$, we have $(-\cL(f))\otimes f\leq v$, i.e., $-\cL(f)\leq v\odiv f$ (cf.~ Definition~\ref{d:odiv}). 
	Meanwhile, it follows from Definition~\ref{d:odiv} ($(v\odiv f)\otimes f\fless v$) and Lemma~\ref{l:tropical linear map preserves order} that 
	\begin{equation*}
		( v\odiv f ) \otimes \cL(f)=\cL((v\odiv f)\otimes f)\leq \cL(v)=0,
	\end{equation*}
    i.e., $v\odiv f\leq-\cL(f)$. Combining the two inequalities, we see that $\cL=l_v$, establishing the claim.
    
\smallskip

 We conclude that $l_\bullet$ is surjective.
	
	Finally, we need to verify the injectivity of $l_\bullet$, i.e., for each pair of $u,\,v$ in $\twidehat{C}(X,\R)$, if 
    \begin{equation}\label{eq: uniqueness of density representation}
   -(v\odiv f)=-(u\odiv f)
	\end{equation} 
    for all $f$ in $\twidehat{C}(X,\R)$, then $v=u$. Now taking $f=u$ in (\ref{eq: uniqueness of density representation}), we get $u\fless v$. By symmetry, we get $v\fless u$. We conclude that $u=v$ and injectivity follows.
\end{proof}

\subsection{Tropical  measures}
This subsection is devoted to the connection between abstract tropical measures and tropical linear functionals: the function $-v$ appearing in the representation of a tropical linear functional turns out to be the density $b$ of the corresponding tropical measure. 

We recall the following definitions from \cite[Definition~18]{ACGQV94}.
\begin{definition}\label{d:tropical measure}
Let $\mathcal{U}$ be the collection of all open subsets of a topological space $X$. A map $m \:\mathcal{U} \rightarrow \Rm$ is a \emph{tropical measure} if it satisfies the following conditions:
\begin{enumerate}[label=\rm{(\roman*)}]	
\smallskip
\item $m(\emptyset)=-\infty$.
\smallskip
\item $m\bigl(\bigcup_{i\in I}A_i\bigr)=\oplusu{i\in I}m(A_i)$, if $I$ is a countable index set and $A_i\in\mathcal{U}$ for each $i\in I$.
\end{enumerate}
If $m(X)<+\infty$, $m$ is said to be \emph{finite}. If $m(X)=0$, $m$ is called a \emph{tropical probability measure}.
\end{definition} 

\begin{remark}
In \cite[Definition 18]{ACGQV94}, tropical probability measures are called \emph{cost measures}. 
\end{remark}

\begin{definition}\label{d:tropical density}
Let $\mathcal{U}$ be the collection of all open subsets of a topological space $X$. If $b \: X \rightarrow  \Rm$ and $m\:\cU \rightarrow \Rm$ satisfy
\begin{equation*}
m(U)=\mathop{\oplus}\limits_{u\in U}b(u)
\end{equation*}
for every open subset $U$ of $X$, then we call $b$ a \emph{density} of the tropical measure $m$.

For a tropical measure $m$ with a density $b$, an open subset $V$ of $X$, and a function $f\in C(X,\R)$, we define the \emph{tropical integral} of $f$ with respect to $m$ over $V$ by 
\begin{equation*}
\tint_{V}\! f \,\mathrm{d}m\=\oplusu{x\in V}(f(x)\otimes b(x)) .
\end{equation*}
For the fact that tropical integrals are well defined, see Remark~\ref{r:density and functional} below.

A function $b\: X \rightarrow \Rm$ is a \emph{density} of a tropical linear functional $l$ on $C(X,\R)\cup\{\tinftyf,\,\taif\}$ if 
\begin{equation*}
	l(f)=\oplusu{x\in X}(f(x)\otimes b(x))
\end{equation*}
for all $f\in C(X,\R)\cup\{\tinftyf,\,\taif\}$.
\end{definition}

\begin{prop}\label{p: measure has a density}
For each finite tropical measure $m$ on a compact metric space $(X,d)$, there exists a unique upper semi-continuous function $b \: X \rightarrow \Rminus$ such that $b$ is a density of $m$.
\end{prop}
\begin{remark}
Here we present a direct proof in the case where $X$ is compact, which is the only case we need. For more general discussions, see \cite[Theorem~19]{ACGQV94} and \cite[Proposition~3.15 and Corollary~3.22]{Ak99}.
\end{remark}
\begin{proof}
We first verify the existence. By the definition of tropical measures, $m(A\cup B)=m(A)\oplus m(B)$. Thus, $A\subseteq B$ implies $m(A)\leq m(B)$. Then for every $x\in X$, we define $b(x)$ as
\begin{equation} \label{defi: density from measure}
b(x)\coloneqq\lim_{\epsilon\to0^+}m(B(x,\epsilon)).
\end{equation} 
Note that the finiteness of $m$ implies that $b\:X\to\Rminus$. It suffices to show that $b$ is upper semi-continuous and that for every open subset $U$ of $X$,
\begin{equation*}
	m(U)=\oplusu{u\in U}b(u).
    \end{equation*} 

Consider a sequence $\{x_k\}_{k\in\N}$ in $X$ converging to $x$ in $X$ as $k\to+\infty$. Note that 
\begin{equation*}b(x_k)\leq m(B(x_k,d(x,x_k)))\leq m(B(x,2d(x,x_k))).
\end{equation*} 
Combining $x_k\to x$ as $k\to +\infty$ and (\ref{defi: density from measure}), we see that \begin{equation*}\limsup_{k\to+\infty}b(x_k)\leq b(x).
\end{equation*} 
This proves the upper semi-continuity of $b$.

Now suppose that $U$ is open. It follows from (\ref{defi: density from measure}) that $b(u)\leq m(U)$ for every $u$ in $U$. Thus, \begin{equation*}
	m(U)\geq\oplusu{u\in U}b(u).
\end{equation*}
To prove the inverse inequality, we need to use the fact that a compact metric space $(X,d)$ is a second-countable topological space, so that for every open cover of $U$, there exists a countable subcover of $U$. Fix $\epsilon>0$. By (\ref{defi: density from measure}), there exists a neighborhood $B(u,r_u)$ such that $b(u)\otimes\epsilon\geq m(B(u,r_u))$ for every $u$ in $U$. Since $\{B(u,r_u)\}_{u\in U}$ forms an open cover of $U$, we then have a countable subcover $\{B(u_k,r_{u_k})\}_{k\in\N}$. It follows from Definition~\ref{d:tropical measure} that
\begin{equation*}
m(U)\leq\oplusu{k\in\N}m(B(u_k,r_{u_k}))\leq\Bigl(\oplusu{u\in U}b(u)\Bigr)\otimes\epsilon.
\end{equation*}
As $\epsilon \to 0^+$, we get $m(U)\leq\oplusu{u\in U}b(u)$. We conclude that $b$ is a density of $m$.

Finally, it is straightforward to check that if $b \: X \rightarrow \Rminus$ is an upper semi-continuous density of the measure $m$, then $b(x)$ must be equal to the limit of $m(B(x,\epsilon))$ as $\epsilon \to 0^+$. The uniqueness follows.
\end{proof}

\begin{rem}\label{r:density and functional}
Let $X$ be a compact metric space. Recall \begin{equation*}
\density(X) =\{b \: X \rightarrow \Rminus : b \text{ is upper semi-continuous}\}\cup\{\tinftyf\}.
\end{equation*} 
Note that $v\in\twidehat{C}(X,\R)$ if and only if $-v\in \density(X)$ (cf.~Proposition~\ref{p: completion}). Thus, by Proposition~\ref{p:dual space equals completion}, $\density(X)$ consists of densities of tropically continuous linear functionals and $b\mapsto l_{-b}$ gives a bijection from $\density(X)$ to $C(X,\R)^{\ostar}$. By Proposition~\ref{p: measure has a density}, $\density(X)\smallsetminus\{\tinftyf\}$ consists of upper semi-continuous densities of finite tropical measures. Recall the notion of tropical integral. We conclude that Proposition~\ref{p:dual space equals completion}, together with Proposition~\ref{p: measure has a density}, forms a tropical analog of the Riesz representation theorem for $C(X,\R)$ when $X$ is a compact metric space.

The existence of the density provides convenience for studying tropical measures and tropical linear functionals (cf.\ Propositions~\ref{p:compactness implies continuity} and~\ref{p:dual space equals completion}). But one needs to be aware that a tropical measure (resp.~linear functional) can have different densities that may not be in $\density(X)$. Recall that two densities $b_1\: X \rightarrow \Rm$ and $b_2\: X \rightarrow \Rm$ are equivalent if $\oplusu{x\in X}(f(x)\otimes b_1(x))=\oplusu{x\in X}(f(x)\otimes b_2(x))$ for all $f\in C(X,\R)$, which implies that they are the densities of the same tropical linear functional on $C(X,\R)\cup\{\tinftyf,\,\taif\}$. It then follows from Propositions~\ref{p:compactness implies continuity} and \ref{p:dual space equals completion} that for every $b_1\:X\to\Rm$, there exists a unique $b_2\in\density(X)$ equivalent to $b_1$. 

Note that there is a slight difference between (i) $\oplusu{x\in X}(f(x)\otimes b_1(x))=\oplusu{x\in X}(f(x)\otimes b_2(x))$ for all $f\in C(X,\R)$ and (ii) $\oplusu{u\in U}{b_1(u)}=\oplusu{u\in U}b_2(u)$ for every open subset $U$ of $X$. If $\oplusu{x\in X}b_1(x)<+\infty$ and $\oplusu{x\in X}b_2(x)<+\infty$, then the two conditions are equivalent. For every $b_1\: X \rightarrow \Rm$ with $\oplusu{x\in X}b_1(x)=+\infty$, $b_1$ and $\tinftyf$ induce the same tropical linear functional, but they may not induce the same tropical measure. Generally, one can prove that if $\oplusu{u\in U}{b_1(u)}=\oplusu{u\in U}b_2(u)$ for every open subset $U$ of $X$, then $\oplusu{x\in X}(f(x)\otimes b_1(x))=\oplusu{x\in X}(f(x)\otimes b_2(x))$ for all $f\in C(X,\R)$ and consequently for all $f\in\twidehat{C}(X,\R)$. 

To justify that the tropical integrals in Definition~\ref{d:tropical density} are well defined, we assume that $\oplusu{u\in U}b_1(u)=\oplusu{u\in U}b_2(u)$ for every open subset $U$ of $X$. Fix an open subset $V$ of $X$ and a function $h\in C(X,\R)$. For each $x\in V$ and each $\epsilon>0$, there exists $\delta>0$ such that $B(x,\delta)\subseteq V$ and $\abs{h(x)-h(y)}\leq \epsilon$ for all $y\in B(x,\delta)$. It follows that 
\begin{align*}
h(x)\otimes b_1(x)
&\leq \oplusu{y\in B(x,\delta)}(h(y)\otimes b_1(y))\\
&\leq \epsilon\otimes h(x)\otimes \Bigl(\oplusu{y\in B(x,\delta)}b_1(y)\Bigr)\\
&=\epsilon\otimes h(x)\otimes\Bigl(\oplusu{y\in B(x,\delta)}b_2(y)\Bigr)\\
&\leq (2\epsilon)\otimes\Bigl(\oplusu{y\in B(x,\delta)}(h(y)\otimes b_2(y))\Bigr)\\
&\leq (2\epsilon)\otimes\Bigl(\oplusu{y\in V}(h_2(y)\otimes b_2(y))\Bigr).
\end{align*}
We conclude that $\oplusu{x\in V}(h(x)\otimes b_1(x))\leq\oplusu{x\in V}(h(x)\otimes b_2(x))$. By symmetry, we see that $\oplusu{x\in V}(h(x)\otimes b_1(y))=\oplusu{x\in V}(h(x)\otimes b_2(y))$.

By Remark~\ref{r: dual space are semimodules}, $\density(X)$ inherits the structure of an $\Rm$-semimodule from $C(X,\R)^{\ostar}$ through the bijection $b\mapsto l_{-b}$. It is important to note that the operations $(\oplus^{\ostar},\otimes^{\ostar})$ in this inherited structure are not defined pointwise. Rather, they can be characterized by the following relations:
\begin{enumerate}
\smallskip
\item For each $\cA\subseteq \density(X)$, $\underset{b\in\cA\;\;}{\oplus^{\ostar}} b$ is the unique density in $\density(X)$ equivalent to $\oplusu{b\in\cA}b$.
\smallskip
\item For each $c\in\Rm$ and each $b\in \density(X)$, $c\otimes^{\ostar}b$ is the unique density in $\density(X)$ equivalent to $c\otimes b$. For example, for each $b\in\density(X)\smallsetminus\{\taif\}$, $(+\infty)\otimes^{\ostar}b=\tinftyf$  even if $b(x)=\tai$ for some $x\in X$.
\end{enumerate} 
\end{rem}

For future reference, we can also define the counterparts of invariant measures and ergodic measures. We will not need these notions in the current article. 

\begin{definition}
Let $T \: X \rightarrow  X$ be a continuous map on a compact metric space $X$. Let $m$ be a finite tropical measure on $X$ with the upper semi-continuous density $b \: X \rightarrow \Rminus$. 
\begin{enumerate}[label=\rm{(\roman*)}]	
	\smallskip
	\item $m$ is \emph{$T$-invariant} if for every point $x$ in $X$, $b(x)=\oplusu{y\in T^{-1}(x)}b(y)$.
	\smallskip
	\item $m$ is \emph{ergodic} if $m$ is \emph{$T$-invariant} and for every point $x$ in $X$, $\lim_{k\to+\infty}b\bigr(T^k(x)\bigr)$ is either $\oplusu{y\in X}b(y)$ or $\tai$.
\end{enumerate}
\end{definition}

\section{Bousch operator and its tropical adjoint}\label{s: Bousch operator}
In parallel to the classical Ruelle operator theory, we investigate in this section the tropical eigenfunctions of $\cL_A$ and tropical eigendensities of $\cL_A^{\ostar}$ for H\"{o}lder continuous potentials $A$.

In Subsection~\ref{ss:analysis}, we present analysis similar to that in thermodynamic formalism and establish constructive results for tropical eigenfunctions of $\cL_A$ and tropical eigendensities of $\cL_A^{\ostar}$. Properties of the Aubry set and the Ma\~{n}\'{e} potential, and the representations of tropical eigenfunctions of $\cL_A$ are recalled in Subsection~\ref{ss:Mane_potential}. In Subsection~\ref{ss: proof of D}, we establish the representations of tropical eigendensities of $\cL_A^{\ostar}$ and prove Theorem~\ref{t:representation}. In Subsection~\ref{ss: sufficient condition for uniqueness}, we discuss a sufficient condition for the uniqueness of tropical eigenfunctions and tropical eigendensities, leading to a proof of Theorem~\ref{t:existence and generic uniqueness} in Subsection~\ref{ss: proof of C}.

\subsection{Analysis in a tropical thermodynamic approach}   \label{ss:analysis}

This subsection provides a tropical counterpart to certain aspects of thermodynamic formalism.  

The construction of tropical eigenfunctions of $\cL_A$ associated with eigenvalue $\energy$ is presented in Proposition~\ref{p:main_construction} and Corollary~\ref{c:more_of_constructions}. We generalize Proposition~\ref{p:main_construction} to establish Corollary~\ref{c:more_of_constructions} as a tropical counterpart of convergence theorems for the Ruelle operators despite the potential nonuniqueness of tropical eigenfunctions. As noted earlier, the construction in this corollary is key for obtaining general representations of tropical eigendensities. Since the map $T$ is not assumed to be transitive, our construction only yields tropical eigenfunctions in $C(X,\Rminus)$ and $\cL_A$ may have no tropical eigenfunctions in $C(X,\R)$. In comparison, we use the existence of a sub-action (known as the Ma\~{n}\'{e} Lemma) with a H\"older seminorm bound from Proposition~\ref{p: mane lemma} without assuming transitivity or surjectivity (cf.\ \cite{Bous11, STY24}) for the proof of Proposition~\ref{p:existence of eigenmeasure}. We then establish the existence of a tropical eigendensity of $\cL_A^{\ostar}$ associated with eigenvalue $\energy$ in Proposition~\ref{p:existence of eigenmeasure} and conclude the subsection with a discussion on the uniqueness of tropical eigenvalue in Proposition~\ref{p:uniqueness of tropical eigenvalue}. 

We remark that although specific propositions and proofs may differ from those in thermodynamic formalism, the underlying reasoning closely parallels the framework of thermodynamic formalism.

Recall that $\aholder$ denotes the space of $\alpha$-H\"{o}lder continuous functions $\varphi \: X \rightarrow \R$ for $\alpha\in(0,1]$. For each $\varphi\in \aholder$ and each $\epsilon \in (0,+\infty]$, denote 
\begin{equation}
		\abs{\varphi}_{d^\alpha,\epsilon}   \= \sup \biggl\{\frac{\abs{\varphi(x)-\varphi(y)} }{d(x,y)^\alpha} : x,\, y\in X, \,  0<d(x,y)<\epsilon \biggr\} 
  \quad \text{ and } \quad
  \abs{\varphi}_{d^\alpha}   \= \abs{\varphi}_{d^\alpha,+\infty}  .  \label{e:Lip_seminorm_local}
\end{equation}
Recall that $\mathbbold{0}$ is the constant zero function on X. For a potential $A$ in $\Rm
^X$, denote 
\begin{equation}\label{eq: defi for S_n} 
	S_nA(x)\= A(x)\otimes A(T(x))\otimes\cdots\otimes A\bigl(T^{n-1}(x)\bigr)
\end{equation}
for all $n\in\N$ and $x\in X$. We adopt the convention that $S_0A\=\mathbbold{0}$. For all $A\in C(X,\R)$, denote 
\begin{equation}\label{eq: defi of toA}
    \toverline{A}\=A-\energy.
\end{equation}
For each $u\in C(X,\R)$ and each $\epsilon>0$, denote
\begin{equation*}
    \omega_u(\epsilon)\=\sup\{\abs{u(x)-u(y)}:x,\, y\in X,\,d(x,y)\leq\epsilon\}.
\end{equation*}

To formulate the main results of this section, we first list some fundamental properties of the Bousch operator in Lemma~\ref{l: generalized form of exchange lemma} and review distance-expanding maps in Lemma~\ref{l:expanding systems}. Lemma~\ref{l: generalized form of exchange lemma} generalizes \cite[Lemma~6.1]{LZ25}. Lemma~\ref{l:expanding systems} is well known (see e.g.~\cite[Lemmas~4.1.2 and 4.1.4]{PU10}) and we omit its proof. 

\begin{lemma}\label{l: generalized form of exchange lemma}
Let $T\: X\to X$ be a map on a set $X$ and $A\in\Rm^X$. 
Then the following hold:
\begin{enumerate}
    \item $\cL_A(c\otimes u)=c\otimes\cL_A(u)$ for all $c\in\Rm$ and $u\in\Rm^X$.
    \smallskip
    \item $\cL_A^n(u)(x)=\oplusu{y\in T^{-n}(x)}(u(y)\otimes S_nA(y))$ for all $n\in\N$, $u\in\Rm^X$, and $x\in X$.
    \smallskip
    \item $\cL_A\bigl(\oplusu{u\in U}u\bigr)(x)=\oplusu{u\in U}\cL_A(u)(x)$ for all $U\subseteq\Rm^X$ and $x\in X$.
    \smallskip
    \item Suppose $T$ is finite-to-one and $A\in\Rminus^X$. Then $\lim\limits_{i\to+\infty}\cL_A(u_i)(x)=\cL_A\bigl(\lim\limits_{i\to+\infty}u_i\bigl)(x)$ for each $x\in X$ and each pointwise convergent sequence $\{u_i\}_{i\in\N}$ of functions in $\Rminus^{X}$.
\end{enumerate}
\end{lemma}

\begin{proof}
(i)--(iii) follow from (\ref{eq: defi of the Bousch operator}) and (\ref{eq: exchangability of numbers}).

\smallskip
(iv) Let $v:X\to\Rminus$ be the pointwise limit of a sequence $\{u_i\}_{i\in\N}$ of functions in $\Rminus^X$ as $i\to+\infty$. Fix arbitrary $x\in X$ and $\epsilon>0$. Since $T$ is finite-to-one, we can find $N\in\N$ such that for each integer $n\geq N$ and each $y\in T^{-1}(x)$, $\Absbig{e^{u_n(y)}-e^{v(y)}}<\epsilon$. It then follows from (\ref{eq: defi of the Bousch operator}) that 
    \begin{align*}
        \abs{\exp(\cL_A(u_n)(x))-\exp(\cL_A(v)(x))}
        & \leq \oplusu{y\in T^{-1}(x)}\Absbig{e^{u_n(y)}e^{A(y)}-e^{v(y)}e^{A(y)}} \\
        & \leq \epsilon\cdot\Bigl(\oplusu{y\in T^{-1}(x)}e^{A(y)}\Bigr)
    \end{align*} 
    for all $n\geq N$. Note that $\oplusu{y\in T^{-1}(x)}e^{A(y)}<+\infty$ since $T$ is finite-to-one and $A\in\Rminus^X$.

    We conclude that $\lim_{i\to+\infty}\cL_A(u_i)(x)=\cL_A(v)(x)$ for each $x\in X$, and (iv) is now verified.
\end{proof}

\begin{lemma}\label{l:expanding systems}
	Let $T\:X \rightarrow  X$ be an open continuous distance-expanding map on a compact metric space $(X,d)$. Let $\lambda>1$ and $\eta>0$ denote the constants in the distance-expanding property of $T$. Then there exists a constant $\xi>0$ such that for each $x$ in $X$, $B(T(x),\xi)\subseteq T(B(x,\eta))$. Moreover, the restriction $T|_{B(x,\eta)}$ is injective, and its inverse map $T_x^{-1} \: B(T(x),\xi) \rightarrow  B(x,\eta)$ has the property that $d \bigl( T_x^{-1}(y),T_x^{-1}(z) \bigr)\leq\lambda^{-1}d(y,z)$. Furthermore, 
    \begin{equation*}
 \sup\bigl\{\card{T^{-1}(x)}:x\in X\bigr\}\eqqcolon N<+\infty.
	\end{equation*} 
\end{lemma}
\begin{rem}\label{r: defi of n times pull back}
For each $n\in\N$, we denote $T_x^{-n}\: B(T^n(x),\xi) \rightarrow  B(x,\lambda^{-n}\xi)$ as the composition of inverse maps $T_{T^i(x)}^{-1}$, $i = 0,\, \dots, \, n-1$. In the remainder of this article, we will fix a choice of $\lambda$, $\eta$, and $\xi$ for a distance-expanding map $T$. Since in our article, the metric $d$ is fixed for every compact space $X$, we will say that a quantity depends on $T$ if it depends on $T$ and $d$.
\end{rem}

Now we formulate the main results Proposition~\ref{p:main_construction} to Proposition~\ref{p:uniqueness of tropical eigenvalue} of this subsection, whose proofs are presented after some technical preparations in Proposition~\ref{p:L_continuous} to Lemma~\ref{l: uniform upper bound without transi and surje}.

\begin{prop}[Construction of tropical eigenfunctions]  \label{p:main_construction}	
Let $T\:X \rightarrow  X$ be an open continuous distance-expanding map on a compact metric space $(X,d)$, and $A \in \aholder$ with $\alpha \in (0,1]$. Let $\xi>0$ be the constant from Lemma~\ref{l:expanding systems}. For each $u\in C(X,\R)$, define 
\begin{equation}\label{eq: defi of v_u}
		v_{u}(x)\coloneqq\limsup_{n\to+\infty}\cL_{\toverline{A}}^n(u)(x)
\end{equation}
for all $x$ in $X$. Then $v_{\tmif}$ satisfies the following properties:
\begin{enumerate}[label=\rm{(\roman*)}]	
		\smallskip
		\item There exists $L_2>0$ depending only on $T$ and $\alpha$ such that $v_{\tmif}\fless L_2\abs{A}_{d^\alpha}$, and for all $x,\, y\in X$ with $d(x,y)<\xi$, 
        \begin{equation}  \label{eq: v0_local_estimate}
        v_{\tmif}(x)\leq v_{\tmif}(y)+\abs{A}_{d^\alpha}d(x,y)^\alpha(\lambda^{\alpha}-1)^{-1}.
        \end{equation}  
	
		\item $v_{\tmif}\in C(X,\Rminus)$ is a tropical eigenfunction of $\cL_A$ associated with eigenvalue $\energy$.
        \smallskip
        \item
        Assume that $T$ is transitive. Then there exists a constant $C_3>0$ depending only on $T$ and $\alpha$ such that $v_{\tmif}\in\aholder$ with 
        \begin{equation*}
        \norm{v_{\tmif}}_{C^0}\leq C_3\abs{A}_{d^\alpha}(\diam X)^\alpha\quad \text{ and }\quad 
        \abs{v_{\tmif}}_{d^\alpha}\leq C_3\abs{A}_{d^\alpha}.
        \end{equation*}
\end{enumerate}

More precisely, we can set $L_2$ to be the constant $L_1$ from Lemma~\ref{l: uniform upper bound without transi and surje} and set $C_3$ to be the constant $C_2$ from Lemma~\ref{l:distortion}.
\end{prop}

Corollary~\ref{c:more_of_constructions} generalizes Proposition~\ref{p:main_construction} and will play a key role in the proof of Theorem~\ref{t:representation}.

\begin{cor}\label{c:more_of_constructions}
In the setting of Proposition~\ref{p:main_construction}, for each $u\in C(X,\R)$, $v_u$ satisfies the following properties:
\begin{enumerate}[label=\rm{(\roman*)}]	
	\smallskip
	\item $v_{\tmif}-\norm{u}_{C^0}\fless v_{u}\fless v_{\tmif}+\norm{u}_{C^0}$, $v_u\fless L_2\abs{A}_{d^\alpha}+\norm{u}_{C^0}$, and for all $x,\, y\in X$ with $d(x,y)<\xi$,
        \begin{equation}\label{ineq:v_u_local_estimate}
        v_{u}(x)\leq v_{u}(y)+\abs{A}_{d^\alpha}d(x,y)^\alpha(\lambda^{\alpha}-1)^{-1}.
        \end{equation} 
        
    \item $v_u\in C(X,\Rminus)$ is a  tropical eigenfunction of $\cL_A$ associated with eigenvalue $\energy$.
    \smallskip
	\item 
    Assume that $T$ is transitive. Then $v_u\in\aholder$ with 
    \begin{equation*}
    \norm{v_u}_{C^0}\leq C_3\abs{A}_{d^\alpha}(\diam X)^\alpha+\norm{u}_{C^0}\quad\text{ and }\quad 
    \abs{v_u}_{d^\alpha,\xi}\leq \abs{A}_{d^\alpha}(\lambda^{\alpha}-1)^{-1}.
    \end{equation*}

\end{enumerate}
\end{cor}

An important ingredient for the proof of Proposition~\ref{p:existence of eigenmeasure} below is a version of the Ma\'{n}\~{e} lemma in Proposition~\ref{p: mane lemma} without the transitivity or surjectivity assumptions. Despite the potential nonexistence of tropical eigenfunctions in $C(X,\R)$, the Ma\'{n}\~{e} lemma below guarantees the existence of a sub-action with a seminorm bound and serves as a fundamental tool for many other aspects of ergodic optimization. 

\begin{prop}[Ma\'{n}\~e lemma]\label{p: mane lemma}
Let $T\: X\to X$ be an open continuous distance-expanding map on a compact metric space $(X,d)$. Let $\lambda>1$ and $\eta>0$ be the constants in the distance-expanding property of $T$, $\xi>0$ be the constant from Lemma~\ref{l:expanding systems}, and $\alpha\in (0,1]$. Then there exists $L>0$ depending only on $T$ and $\alpha$ such that for all $A\in\aholder$, there exists a sub-action $v_A\in\aholder$ for $A$ with 
\begin{equation*}
\abs{v_A}_{d^\alpha,\xi}\leq\abs{A}_{d^\alpha}(\lambda^{\alpha}-1)^{-1}\quad \text{ and }\quad 
\abs{v_A}_{d^\alpha}\leq L\abs{A}_{d^\alpha}.
\end{equation*}
\end{prop}

\begin{prop}\label{p:existence of eigenmeasure}
	Let $T\:X \rightarrow  X$ be an open continuous distance-expanding map on a compact metric space $(X,d)$, and $A \in \aholder$ with $\alpha \in (0,1]$. There exists $b\in \density(X)\smallsetminus\{\tinftyf,\,\taif\}$ such that \begin{equation*}
		\cL_{A}^{\ostar}(b)=\energy\otimes b.
		\end{equation*}
\end{prop}

We discuss the uniqueness of tropical eigenvalues in Proposition~\ref{p:uniqueness of tropical eigenvalue}. Due to the lack of transitivity, $\cL_A$ (resp.~$\cL_A^{\ostar}$) may have more than one eigenvalue on $C(X,\Rminus
)$ (resp.~$\density(X)$).
\begin{prop}[Uniqueness of tropical eigenvalues]
	\label{p:uniqueness of tropical eigenvalue}
	Let $T\:X \rightarrow  X$ be an open continuous distance-expanding map on a compact metric space $(X,d)$, and $A \in \aholder$ with $\alpha \in (0,1]$. Then the following hold:
	\begin{enumerate}[label=\rm{(\roman*)}]	
		\smallskip
        \item If there exists $u\in C(X,\R)$ and $\lambda\in\Rm$ such that $\cL_A(u)=\lambda\otimes u$, then $\lambda=\energy$.
        \smallskip
		\item If there exists $u\in C(X,\Rminus)\smallsetminus\{\taif\}$ and $\lambda\in\Rm$ such that $\cL_A(u)=\lambda\otimes u$, then $\lambda\leq\energy$. Moreover, if $T$ is transitive, then $u\in C(X,\R)$ and $\lambda=\energy$.
		\smallskip
		\item If there exists $b\in \density(X)\smallsetminus\{\tinftyf,\,\taif\}$ and $\lambda\in\Rm$ such that $\cL_A^{\ostar}(b)=\lambda\otimes b$, then $\lambda\leq\energy$. Moreover, if $T$ is transitive, then $\lambda=\energy$.
	\end{enumerate}
\end{prop}

Before providing the proofs of Proposition~\ref{p:main_construction} to 
Proposition~\ref{p:uniqueness of tropical eigenvalue}, we discuss the following technical facts, whose proofs can be found in Appendix~\ref{ss: Appendix A}.

\begin{prop}  \label{p:L_continuous}
	Let $T\:X \rightarrow  X$ be an open continuous distance-expanding map on a compact metric space $(X,d)$, and $A\in C(X,\R)$. Then $\cL_{A} (C(X,\Rminus)) \subseteq C(X,\Rminus)$. 
	
	Assume, in addition, that $T$ is surjective. Then $\cL_A (C(X,\R)) \subseteq C(X,\R)$, and if $A\in \aholder$ with $\alpha \in (0,1]$ then $\cL_A \bigl( \aholder \bigr) \subseteq \aholder$.
\end{prop}

\begin{lemma}  \label{l:distortion}
	Let $T\:X \rightarrow  X$ be an open continuous distance-expanding map on a compact metric space $(X,d)$, and $A \in \aholder$ with $\alpha \in (0,1]$. Let $\xi>0$ be the constant in Lemma~\ref{l:expanding systems}. Then for all $x,\, y\in X$ with $d(x,y)<\xi$, $n\in\N$, and $u\in C(X,\R)$, we have
        \begin{equation}\label{ineq: equi estimate for L_A_n(u)}
            \cL_A^n(u)(x)\leq \cL_A^n(u)(y)+\abs{A}_{d^\alpha}d(x,y)^\alpha(\lambda^{\alpha}-1)^{-1}+\omega_u(\lambda^{-n}d(x,y)) ,
        \end{equation}
        and if $x_1\in T^{-n}(x)$, then
        \begin{equation}\label{ineq:difference of the nearby trajectory sum}
        \Absbig{S_nA(x_1)-S_nA\bigl(T_{x_1}^{-n}(y)\bigr)}\leq\abs{A}_{d^\alpha}d(x,y)^\alpha(\lambda^{\alpha}-1)^{-1}.
	   \end{equation} 
Moreover, if $T$ is transitive, then the following hold:
\begin{enumerate}[label=\rm{(\roman*)}]	
\smallskip
\item There exists a constant $C_1>0$ such that for all $n\in\N$ and $x,\, y\in X$,
	\begin{equation*}
		\AbsBig{\oplusu{\ox\in T^{-n}(x)}S_nA(\ox)-\oplusu{\oy\in T^{-n}(y)}S_nA(\oy)}\leq C_1.
	\end{equation*}

\item There exists a constant $C_2>0$ depending only on $T$ and $\alpha$ such that for all $A,\, u\in \aholder$, and $n\in\N$,
	\begin{equation*}
	\abs{\cL_{A}^n(u)}_{d^\alpha}\leq C_2(\abs{A}_{d^\alpha}+\abs{u}_{d^\alpha}).
	\end{equation*}
\end{enumerate}
\end{lemma}

\begin{lemma}\label{l:projective Hilbert metric estimate}
	Let $T\:X \rightarrow  X$ be an open continuous distance-expanding map on a compact metric space $(X,d)$, and $A \in \aholder$ with $\alpha \in (0,1]$. Then the Bousch operator $\cL_A$ satisfies 
    \begin{equation}
		\cL_A^k(u_2)\leq\norm{u_1-u_2}_{C^0}\otimes\cL_A^k(u_1)
	\end{equation}
    for all $k\in\N$ and $u_1,\, u_2\in C(X,\R)$. Moreover, if $T$ is surjective, then for all $u_1,\, u_2\in C(X,\R)$,
    \begin{equation}
        \norm{\cL_A(u_1)-\cL_A(u_2)}_{C^0}\leq\norm{u_1-u_2}_{C^0}.
    \end{equation}
\end{lemma}

\begin{proof}
	Note that $u_2\fless u_1\otimes\norm{u_2-u_1}_{C^0}$ for all $u_1,\, u_2\in C(X,\R)$. By Lemma~\ref{l: generalized form of exchange lemma}, $\cL_A$ is tropical linear (cf.~Definition~\ref{d:max-plus linear}). It follows that $\cL_A^k$ is tropical linear for all $k\in\N$. Thus, Lemma~\ref{l:tropical linear map preserves order} implies that 
	$\cL_A^k(u_2)\fless\norm{u_1-u_2}_{C^0}\otimes\cL_A^k(u_1)$ for all $k\in\N$ and $u_1,\, u_2\in C(X,\R)$. Moreover, if $T$ is surjective, then Proposition~\ref{p:L_continuous} implies that $\cL_A(u)\in C(X,\R)$ for all $u\in C(X,\R)$. Thus, it follows that $\norm{\cL_A(u_1)-\cL_A(u_2)}_{C^0}\leq\norm{u_1-u_2}_{C^0}$.
\end{proof}

\begin{lemma}\label{l: uniform upper bound without transi and surje}Let $T\: X\to X$ be an open continuous distance-expanding map on a compact metric space $(X,d)$, and $\alpha\in(0,1]$. Then there exists $L_1>0$ depending only on $T$ and $\alpha$ such that for all $x\in X$, $n\in\N$, and $A\in\aholder$, we have
$S_n\toverline{A}(x)\leq L_1\abs{A}_{d^\alpha}$.
\end{lemma}
\begin{proof}
     Let $\lambda>1$ and $\eta>0$ be constants in the distance-expanding property of $T$, and $\xi$ be the constant in Lemma~\ref{l:expanding systems}. Denote $\beta\=\min\{\eta,\,\xi\}$ and $\gamma\=\min\{(\lambda-1)\beta,\,\xi\}$. Since $X$ is compact, let $N_{\gamma}$ be the maximum cardinality of a $\gamma$-separated subset of $X$. Fix $A\in\aholder$.

    Now fix a point $x\in X$ and an integer $n>N_\gamma$. It follows that there exist integers $0\leq i<j\leq n-1$ such that $d\bigl(T^i(x),T^j(x)\bigr)<\gamma$. Consider 
    \begin{align*}
    i_0&\=\min\bigl\{i \in \N_0 : 0\leq i\leq n-1 \\
    & \qquad\qquad\qquad\quad  \text{ and there exists }i<j\leq n-1\text{ such that }d\bigl(T^i(x),T^j(x)\bigr)<\gamma\bigr\}, \\
    j_0&\=\max\bigl\{j \in \N_0 : i_0<j\leq n-1,\,d\bigl(T^{i_0}(x),T^j(x)\bigr)<\gamma\bigr\}.
    \end{align*}
    Now we can recursively define $i_{k+1},\, j_{k+1}$ when $i_k,\, j_k$ have been defined by the following procedure. If there exist integers $j_k<i<j\leq n-1$ such that $d\bigl(T^i(x),T^j(x)\bigr)<\gamma$, then consider
    \begin{align*}
        i_{k+1}&\=\min\bigl\{i \in \N_0 : j_k<i\leq n-1 \\
        & \qquad\qquad\qquad\quad  \text{ and there exists }i<j\leq n-1\text{ such that }d\bigl(T^i(x),T^j(x)\bigr)<\gamma\bigr\}, \\
        j_{k+1}&\=\max\bigl\{j \in \N_0 : i_{k+1}<j\leq n-1,d\bigl(T^{i_{k+1}}(x),T^j(x)\bigr)<\gamma\bigr\}.
    \end{align*}
    We stop defining if there do not exist integers $j_k<i<j\leq n-1$ with $d\bigl(T^i(x),T^j(x)\bigr)<\gamma$. 
    
    We get a series of indices $0\leq i_0<j_0<i_1<j_1<\cdots<i_k<j_k\leq n-1$ with $k\in\N_0$. Denote $I_1\=\{i_l:0\leq l\leq k\}$ and $I_2\=\{m:0\leq m\leq n-1\}\smallsetminus\bigcup_{l=0}^{k}[i_l,j_l]$. Note that it follows from the definition of these indices that $\bigl\{T^i(x):i\in I_1\bigr\}$ and $\bigl\{T^i(x):i\in I_2\bigr\}$ are both $\gamma$-separated subsets of $X$. Thus we see that $k+1=\# I_1\leq N_\gamma$, $\# I_2\leq N_\gamma$, and
    \begin{align*}
    S_n\toverline{A}(x)
    &\leq \sum_{i\in I_2}\toverline{A}\bigl(T^i(x)\bigr)+\sum_{l=0}^{k} S_{j_l-i_l+1}\toverline{A}\bigl(T^{i_l}(x)\bigr) \\
    &\leq 2N_\gamma\norm{\toverline{A}}_{C^0}+\sum_{l=0}^{k} S_{j_l-i_l}\toverline{A}\bigl(T^{i_l}(x)\bigr).
    \end{align*}

    It follows from \cite[Corollaries~4.2.4 and 4.2.5]{PU10} that for each integer $0\leq l\leq k$ there exists a periodic point $w_l\in X$ with period $j_l-i_l$ such that for each integer $0\leq p\leq j_l-i_l$, $d\bigl(T^{p+i_l}(x),T^p(w_l)\bigr)\leq\beta\leq\xi$ . Thus it follows from (\ref{ineq:difference of the nearby trajectory sum}) that
    \begin{equation*}
        S_{j_l-i_l}\toverline{A}\bigl(T^{i_l}(x)\bigr)\leq S_{j_l-i_l}\toverline{A}(w_l)+\abs{A}_{d^\alpha}\xi^\alpha(\lambda^{\alpha}-1)^{-1}
    \end{equation*}
    for each integer $0\leq l\leq k$.
    Since $w_l$ is a periodic point with period $j_l-i_l$, it follows from the definition of $\energy$ that $S_{j_l-i_l}A(w_l)\leq(j_l-i_l)\energy$ for each integer $0\leq l\leq k$. 
    
    We conclude that for all $n\in\N$,
    \begin{equation*}
        S_n\toverline{A}(x)\leq 2N_\gamma\norm{\toverline{A}}_{C^0}+N_\gamma\abs{A}_{d^\alpha}\xi^\alpha(\lambda^{\alpha}-1)^{-1}.
    \end{equation*}

    Note that $Q\bigl(T,\toverline{A}\bigr)=\energy-\energy=0$, together with the compactness of $X$, implies that $\min\{\toverline{A}(x): x\in X \}\leq 0$ and $\max\{ \toverline{A}(x) : x\in X \}\geq 0$. It follows that $\norm{\toverline{A}}_{C^0}\leq\abs{A}_{d^\alpha}(\diam X)^\alpha$ and $S_n\toverline{A}(x)\leq \abs{A}_{d^\alpha}L_1$, where 
    \begin{equation}
    	L_1\=2N_{\gamma}(\diam X)^\alpha +N_\gamma\xi^{\alpha}(\lambda^{\alpha}-1)^{-1}.
    \end{equation}
    Note that $L_1$ depends only on $T$ and $\alpha$. Finally, it is straightforward to check that $S_n\toverline{A}(x)\leq L_1\abs{A}_{d^\alpha}$ for all $x\in X$, $n\in\N$, and $A\in\aholder$.
\end{proof}

Now we are ready to prove Proposition~\ref{p:main_construction} and Corollary~\ref{c:more_of_constructions}.
The proof of the part of Proposition~\ref{p:main_construction} with the transitivity assumption simplifies the approach in \cite[Proposition~6.4]{LZ25}.
\begin{proof}[Proof of Proposition~\ref{p:main_construction}]

    \smallskip
    (i) Let $L_1$ be the constant in Lemma~\ref{l: uniform upper bound without transi and surje} depending only on $T$ and $\alpha$. It immediately follows from Lemmas~\ref{l: generalized form of exchange lemma}~(ii) and~\ref{l: uniform upper bound without transi and surje} that $v_{\tmif}\fless L_1\abs{A}_{d^\alpha}$ and we can take $L_2\=L_1$. Note that (\ref{eq: v0_local_estimate}) follows from (\ref{ineq: equi estimate for L_A_n(u)}).
    
    \smallskip
    
    (ii) It follows from (i) that $v_{\tmif}\in C(X,\Rminus)$. It then follows from Lemma~\ref{l: generalized form of exchange lemma}~(iii)(iv) that 
    \begin{equation*}
        \cL_{\toverline{A}}(v_{\tmif})
        =\lim_{n\to+\infty}\cL_{\toverline{A}}\Bigl(\sup_{k\geq n}\cL_{\toverline{A}}^k(\tmif)\Bigr)
        =\lim_{n\to+\infty}\sup_{k\geq n}\cL_{\toverline{A}}^{k+1}(\tmif)=v_{\tmif}.
    \end{equation*}
    
    Now it suffices to show that $v_{\tmif}\neq\taif$, i.e., $v_{\tmif}(y)\in\R$ for some $y\in X$. Denote $a_n\=\sup_{x\in X}\cL_{A}^n(\tmif)(x)$ for each $n\in\N$. By Proposition~\ref{p:L_continuous}, $\cL_A^n(\tmif)\in C(X,\Rminus)$. Thus there exists $x_n\in X$ such that $\cL_A^n(\tmif)(x_n)=a_n$ since $X$ is compact. By Lemma~\ref{l: generalized form of exchange lemma}~(ii),
    $a_n=\sup_{x\in X}S_nA(x)\in\R$. Thus, by \cite[Proposition~2.1]{Je06}, 
    \begin{equation*}
       \energy=\limsup_{n\to+\infty}\frac{a_n}{n}.
    \end{equation*}
    Observe that $\{a_n\}_{n\in\N}$ is subadditive. It follows that $\energy=\lim_{n\to+\infty}\frac{a_n}{n}=\inf \{\frac{a_n}{n}: n\in\N\}\in\R$, and consequently $\cL_{\toverline{A}}^n(\tmif)(x_n)=a_n-n\energy\geq 0$ for all $n\in\N$. Since $X$ is compact, there exists a subsequence $\{x_{n_k}\}_{k\in\N}$ (with $\lim_{k\to+\infty}n_k=+\infty$) converging to $y\in X$ as $k\to+\infty$. Thus, there exists $N\in\N$ such that $d(x_{n_k},y)<\xi$ for all $k\geq N$. It then follows from (\ref{ineq: equi estimate for L_A_n(u)}) and $\cL_{\toverline{A}}^n(\tmif)(x_n)\geq 0$ that 
    \begin{align*}
        \cL_{\toverline{A}}^{n_k}(\tmif)(y)
        & \geq\cL_{\toverline{A}}^{n_k}(\tmif)(x_{n_k})-\abs{A}_{d^\alpha}d(y,x_{n_k})^\alpha(\lambda^\alpha-1)^{-1} \\
        & \geq -\abs{A}_{d^\alpha}d(y,x_{n_k})^\alpha(\lambda^\alpha-1)^{-1}.
    \end{align*}
    Hence, 
    \begin{align*}
        v_{\tmif}(y)
        = \limsup_{n\to+\infty}\cL_{\toverline{A}}^n(\tmif)(y)
        & \geq \limsup_{k\to+\infty}\cL_{\toverline{A}}^{n_k}(\tmif)(y) \\
        & \geq \limsup_{k\to+\infty} \bigl( -\abs{A}_{d^\alpha}d(y,x_{n_k})^\alpha(\lambda^\alpha-1)^{-1} \bigr)
         = 0.
    \end{align*}
    Since $v_{\tmif}\in C(X,\Rminus)$, we see that $v_{\tmif}(y)\in\R$.
    
    \smallskip
    
    (iii) Let $C_2>0$ be the constant from Lemma~\ref{l:distortion}. Denote $D_1\=C_2\abs{A}_{d^\alpha}(\diam X)^\alpha$.  
    
    Assume that $T$ is transitive. It follows from Lemma~\ref{l:distortion}~(ii) that $\abs{\cL_{A}^n(\tmif)}_{d^\alpha}\leq C_2\abs{A}_{d^\alpha}$ for all $n\in\N$. Recall that $a_n=\sup_{x\in X}\cL_A^n(\tmif)(x)\in\R$. Thus, for all $n\in\N$,
    \begin{equation}\label{ineq: variation implies lower bound for L_A^n}
    	(a_n-D_1)\otimes\tmif\fless\cL_{A}^n(\tmif).
    \end{equation}
    Then it follows from the tropical linearity of $\cL_A$ (cf.\ Lemma~\ref{l: generalized form of exchange lemma}~(i)(iii)) and Lemma~\ref{l:tropical linear map preserves order} that for all $n,\, m\in\N$,
    \begin{equation*}
    	(a_n-D_1)\otimes\cL_{A}^m(\tmif)=\cL_{A}^m((a_n-D_1)\otimes\tmif)\fless\cL_{A}^{n+m}(\tmif).
    \end{equation*}
    It follows that $(a_n-D_1)\otimes a_m\leq a_{n+m}$ for all $n,\, m\in\N$, i.e., $\{D_1-a_n\}_{n\in\N}$ is subadditive. Thus, 
    \begin{equation*}
    	\inf_{n\in\N}(D_1-a_n)/n=\lim_{n\to+\infty}(D_1-a_n)/n=-\energy.
    \end{equation*}
    We conclude that $0\leq a_n-n\energy\leq D_1$ for all $n\in\N$. Thus by (\ref{ineq: variation implies lower bound for L_A^n}), we have for all $n\in\N$, $\Normbig{\cL_{\toverline{A}}^n(\tmif)}_{C^0}\leq D_1$.
   
    Recall that $\Absbig{\cL_{\toverline{A}}^n(\tmif)}_{d^\alpha}=\abs{\cL_{A}^n(\tmif)}_{d^\alpha}\leq C_2\abs{A}_{d^\alpha}$ for all $n\in\N$. We conclude that the sequence $\bigl\{\cL_{\toverline{A}}^n(\tmif)\bigr\}_{n\in\N}$ is equicontinuous and uniformly bounded. Consequently, 
    $\bigl\{\sup_{k\geq n}\cL_{\toverline{A}}^k(\tmif)\bigr\}_{n\in\N}$ is also equicontinuous and uniformly bounded. 
    
    Thus, $v_{\tmif}$, as the pointwise decreasing limit of $\sup_{k\geq n}\cL_{\toverline{A}}^k(\tmif)$ as $n\to +\infty$, is the uniform limit of $\sup_{k\geq n}\cL_{\toverline{A}}^k(\tmif)$ by the Arzel\`{a}--Ascoli theorem. Now it follows clearly that $\norm{v_{\tmif}}_{C^0}\leq D_1$ and $\abs{v_{\tmif}}_{d^\alpha}\leq C_2\abs{A}_{d^\alpha}$. Now we can take $C_3\=C_2$, and the proof is now complete.
\end{proof}

\begin{proof}[Proof of Corollary~\ref{c:more_of_constructions}]
Fix an arbitrary $u\in C(X,\R)$.

\smallskip
(i) By Lemma~\ref{l:projective Hilbert metric estimate}, we have for all $n\in\N$, $\cL_{\toverline{A}}^n(\tmif)-\norm{u-\tmif}_{C^0}\fless \cL_{\toverline{A}}^n(u)\fless\norm{u-\tmif}_{C^0}+\cL_{\toverline{A}}^n(\tmif)$. Consequently, $v_{\tmif}-\norm{u-\tmif}_{C^0}\fless v_u\fless \norm{u-\tmif}_{C^0}+v_{\tmif}$.
It then follows from Proposition~\ref{p:main_construction}~(i) that $v_u\leq L_2\abs{A}_{d^\alpha}+\norm{u}_{C^0}$. Since $u\in C(X,\R)$ and $X$ is compact, we see that $u$ is uniformly continuous and thus $\lim_{\epsilon\to 0^+}\omega_{u}(\epsilon)=0$. Now (\ref{ineq:v_u_local_estimate}) follows from (\ref{ineq: equi estimate for L_A_n(u)}) and $\lim_{n\to+\infty}\omega_{u}(\lambda^{-n}\xi)=0$.

\smallskip
(ii) It follows from (i) that $v_u\in C(X,\Rminus)$. That $\cL_{\toverline{A}}(v_u)=v_u$ follows from a similar argument as that of Proposition~\ref{p:main_construction}~(ii). It follows from $v_{\tmif}-\norm{u}_{C^0}\leq v_u\leq v_{\tmif}+\norm{u}_{C^0}$ and Proposition~\ref{p:main_construction}~(ii) that $v_u(y)\in\R$ for some $y\in\R$.

\smallskip
(iii) Recall from (ii) that $v_u\in C(X,\Rminus)$. Assume that $T$ is transitive. It then follows from (i) and Proposition~\ref{p:main_construction}~(iii) that $v_u\in C(X,\R)$ with $\norm{v_u}_{C^0}\leq C_3\abs{A}_{d^\alpha}(\diam X)^\alpha+\norm{u}_{C^0}$. Furthermore, it follows from (i) that $\abs{v_u}_{d^\alpha,\xi}\leq\abs{A}_{d^\alpha}(\lambda^{\alpha}-1)^{-1}$ and consequently $v_u\in \aholder$.
\end{proof}

\begin{remark}
It follows from Corollary~\ref{c:more_of_constructions}~(i) that $v_{u}^{-1}(-\infty) = v_{\tmif}^{-1}(-\infty)$ for every $u \in C(X,\R )$. When $T$ is not transitive, examples with $v_{\tmif}^{-1}(-\infty) \neq \emptyset$ are readily constructible. For example, take two transitive expanding covering maps $T_1\:X_1\to X_1$ and $T_2\: X_2\to X_2$, and consider the map $T\: X_1\bigsqcup X_2\to X_1\bigsqcup X_2$ on the disjoint union that acts as $T_i$ on the component $X_i$ for each $i\in\{1,\, 2\}$. Consider the potential $A$ that takes $i$ on the component $X_i$ for each $i\in\{1,\, 2\}$. It is straightforward to check that $\energy=2$ and $v_{\tmif}|_{X_1}=\taif_{X_1}$.

This indicates that the preceding constructions may only provide tropical eigenfunctions in $C(X,\Rminus)$ associated with eigenvalue $\energy$. The analogous phenomenon in thermodynamic formalism arises for eigenfunctions of the Ruelle operator $\mathcal{R}_{A}$ associated with eigenvalue $\exp(P(T, A))$: without transitivity, these eigenfunctions may lack strict positivity and vanish at certain points. This correspondence reflects the algebraic similarity that $-\infty=\tai$ in the tropical algebra over $\Rm$ corresponds to $0$ in $\R$.
\end{remark}

Now we present a direct proof of Proposition~\ref{p: mane lemma} without assuming transitivity or surjectivity. While adopting the construction of a sub-action $v_A$ from \cite[Proposition~11]{CLT01} (where $C^1$ expanding maps on the circle are considered), our argument crucially relies on the estimates in Lemma~\ref{l: uniform upper bound without transi and surje}, which are essential for handling the enhanced generality for our map $T$. 
\begin{proof}[Proof of Proposition~\ref{p: mane lemma}]
    Consider $v_A(x)\=\max\bigl\{0 , \, \sup_{n\geq 1}\max_{y\in T^{-n}(x)}S_n\toverline{A}(y) \bigr\}$ for all $x\in X$. Let $L_1$ be the constant in Lemma~\ref{l: uniform upper bound without transi and surje} depending only on $T$ and $\alpha$. It follows that $0\leq v_A(x)\leq L_1\abs{A}_{d^\alpha}$ for all $x\in X$. Moreover, it follows from (\ref{ineq:difference of the nearby trajectory sum}) that 
    \begin{equation*}
 \abs{v_A}_{d^\alpha,\xi}\leq\abs{A}_{d^\alpha}(\lambda^{\alpha}-1)^{-1}.
    \end{equation*}Since $X$ is compact, we conclude that $v_A\in\aholder$ and 
    \begin{align*}
        \abs{v_A}_{d^\alpha}
        &\leq \max\Bigl\{\abs{v_A}_{d^\alpha,\xi},\,\xi^{-\alpha}\sup_{x,\, y\in X}\abs{v_A(x)-v_A(y)}\Bigr\} \\
        &\leq \max\bigl\{\abs{A}_{d^\alpha}(\lambda^{\alpha}-1)^{-1},\,\xi^{-\alpha}L_1\abs{A}_{d^\alpha}\bigr\} \\
        & =L\abs{A}_{d^\alpha},
    \end{align*}
    where the constant $L\=\max\bigl\{(\lambda^{\alpha}-1)^{-1},\,\xi^{-\alpha}L_1\bigr\}$ depends only on $T$ and $\alpha$.

    Now it suffices to show that $v_A$ is a sub-action, i.e.,
    $v_A(x)+\toverline{A}(x)\leq v_A(T(x))$ for all $x\in X$. Note that for all $n\geq 1$ and $y\in T^{-n}(x)\subseteq T^{-n-1}(T(x))$, $S_n\toverline{A}(y)+\toverline{A}(x)=S_{n+1}\toverline{A}(y)$. It follows that 
    \begin{equation*}
        \begin{aligned}
            v_A(x)+\toverline{A}(x)
            &=\max\Bigl\{\sup_{n\geq 1}\max_{y\in T^{-n}(x)}S_{n+1}\toverline{A}(y),\,\toverline{A}(x)\Bigr\}\\
            &\leq\max\Bigl\{\sup_{m\geq 2}\max_{y\in T^{-m}(T(x))}S_m\toverline{A}(y),\,\max_{y\in T^{-1}(T(x))}\toverline{A}(y)\Bigr\}\\
            &=\sup_{m\geq 1}\max_{y\in T^{-m}(T(x))}S_m\toverline{A}(y)\\
            &\leq v_A(T(x))
        \end{aligned}
    \end{equation*}for all $x\in X$. The proof is now complete.
\end{proof}

Before constructing a tropical eigendensity of $\cL_A^{\ostar}$ associated with eigenvalue $\energy$, we discuss in the following remark the reason for the definition of $\cL_{A}^{\ostar}$ in (\ref{eq:definition of tropical adjoint}).

\begin{rem}\label{r:reason for defi of the dual operator}
Due to Remark~\ref{r:density and functional}, we define the tropical adjoint operators on the space $\density(X)$. It is clear from (\ref{eq:definition of tropical adjoint}) that $b\in \density(X)$ implies $\mathcal{L}_A^{\ostar}(b)\in \density(X)$. 
	
Recall from Proposition~\ref{p:dual space equals completion} that $b\mapsto l_{-b}$ gives a bijection from $\density(X)$ to $C(X,\R)^{\ostar}$. We denote $c\=\cL_A^{\ostar}(b)$. Then by Remark~\ref{r: form of l_u},
\begin{equation}\label{eq: reason for adjoint}
\begin{aligned}
l_{-c}(u)
&=\oplusu{x\in X} (u(x)\otimes\cL_{A}^{\ostar}(b)(x) )\\
&=\oplusu{x\in X}(u(x)\otimes b(T(x))\otimes A(x))\\ 
&=\oplusu{x\in X}\Bigl(\oplusu{y\in T^{-1}(x)}(u(y)\otimes A(y))\otimes b(x)\Bigr) \\
&=\oplusu{x\in X}(\cL_{A}(u)(x)\otimes b(x))\\
&=l_{-b}(\cL_A(u))
\end{aligned}
\end{equation}
for all $u\in C(X,\R)$. By identifying $b$ with $l_{-b}$, $\cL_A^{\ostar}$ can be seen as a map from $C(X,\R)^{\ostar}$ to $C(X,\R)^{\ostar}$, i.e., $\cL_{A}^{\ostar}(l_{-b})=l_{-c}$. Now the above identities imply that
\begin{equation}\label{eq: defi of adjoint using functional}
\cL_{A}^{\ostar}(l)(u)=l(\cL_{A}(u))
\end{equation}
for all $l\in C(X,\R)^{\ostar}$ and $u\in C(X,\R)$. To avoid confusion with notations, we will not use this identification in this article.

For $b\in \density(X)\smallsetminus\{\tinftyf\}$, let $m_b$ be the finite tropical measure satisfying $m_b(U)=\oplusu{x\in U}b(x)$ for every open subset $U$ of $X$. Recall Definition~\ref{d:tropical density} and $c=\cL_{A}^{\ostar}(b)$. Then for every open subset $U$ of $X$,
\begin{equation}\label{eq: calculation of adjoint on measure}
\begin{aligned}
m_c(U)
&=\oplusu{x\in U}\cL_{A}^{\ostar}(b)(x)=\oplusu{x\in U}(b(T(x))\otimes A(x))\\
&=\oplusu{x\in T(U)}\Bigl(b(x)\otimes \Bigl(\oplusu{y\in T^{-1}(x)}A(y)\Bigr)\Bigr)
=\tint_{T(U)}\Bigl(\oplusu{y\in T^{-1}(x)}A(y)\Bigr)\,\mathrm{d}m_b .
\end{aligned}
\end{equation}

By identifying $b$ with $m_b$, $\cL_{A}^{\ostar}$ can be seen as a map on the set of finite tropical measures, i.e.,
\begin{equation}\label{eq: defi for adjoint using measure}
\cL_{A}^{\ostar}(m)(U)=\tint_{T(U)}\oplusu{y\in T^{-1}(x)}A(y)\,\mathrm{d}m
\end{equation}
for every open subset $U$ of $X$. Note that the above identity can serve as the defining identity for $\cL_A^{\ostar}$ on the set of finite tropical measures.
\end{rem}

Here we establish the existence of a tropical eigendensity of $\cL_A^{\ostar}$ associated with eigenvalue $\energy$ and prove Proposition~\ref{p:existence of eigenmeasure}.
\begin{proof}[Proof of Proposition~\ref{p:existence of eigenmeasure}]
	 We consider $v\coloneqq-b$ (cf.~Remark~\ref{r:density and functional}) and it suffices to find a solution of $v\circ T=v+\toverline{A}$ in $\twidehat{C}(X,\R)\smallsetminus\{\tinftyf,\,\taif\}$.
	
	Let $v_A\in C(X,\R)$ be the sub-action for $A\in\aholder$ from Proposition~\ref{p: mane lemma}. So 
	\begin{equation}\label{ineq: v_0 is a sub-action}
	v_A+\toverline{A}\fless v_A\circ T.
	\end{equation}
	Consider $w\coloneqq v-v_A$ and $\varphi\coloneqq v_A+\toverline{A}-v_A\circ T\in C(X,\R)$. It suffices to find $w\in\twidehat{C}(X,\R)\smallsetminus\{\tinftyf,\,\taif\}$ such that 
	\begin{equation}\label{eq: target equation}
    w\circ T=w+\varphi.
    \end{equation}
	
	Note that $\varphi\fless\mathbbold{0}$ by (\ref{ineq: v_0 is a sub-action}) and
	\begin{equation*}
	Q(T,\varphi)=\sup_{\mu\in\cM(X,T)}\int \! \varphi \, \mathrm{d}\mu=\max_{\mu\in\cM(X,T)}\int \! \varphi \, \mathrm{d} \mu=\max_{\mu\in\cM(X,T)}\int \! \toverline{A} \, \mathrm{d} \mu=Q(T,\toverline{A})=0.
	\end{equation*}
	The supremum is attained due to the weak$^*$-compactness of $\cM(X,T)$. We fix a maximizing measure $\mu\in \cM(X,T)$ for $\varphi$. 
	
	It follows from $C(X,\R)\ni\varphi\fless\mathbbold{0}$ and $\int\varphi\,\mathrm{d}\mu=0$ that $\supp(\mu)\subseteq\varphi^{-1}(0)$. Moreover, since $\mu$ is $T$-invariant, $T(\supp(\mu))\subseteq\supp(\mu)$.  We use these two properties to construct a solution $w=w_0$ of (\ref{eq: target equation}). 
	
	Consider $S\coloneqq\bigl\{w\in\twidehat{C}(X,\R):w+\varphi\fless w\circ T,\, w|_{\supp(\mu)}=0,\, \mathbbold{0}\fless w\bigr\}$ and $w_0\coloneqq\oplusu{w\in S}w\in\twidehat{C}(X,\R)$ (since $\twidehat{C}(X,\R)$ is tropically complete). Since $\varphi\fless\mathbbold{0}$, we see that $\mathbbold{0}\in S$. Consequently, $S$ is not empty and $w_0\neq\taif$. It is straightforward to check that $w_0\in S$, and it follows from $w_0|_{\supp(\mu)}=0$ that $w_0\neq\tinftyf$.
    
    To show that $w_0\circ T=w_0+\varphi$, we consider $w_1\coloneqq w_0\circ T-\varphi$. 
	
	 Since $w_0\in S$, we have $\mathbbold{0}\fless w_0\fless w_1$. So $w_0\circ T\fless w_1\circ T$, i.e.,  $w_1+\varphi\fless w_1\circ T$. Recall that $\supp(\mu)\subseteq\varphi^{-1}(0)$ and $T(\supp(\mu))\subseteq \supp(\mu)$. It follows that  $w_1|_{\supp(\mu)}=(w_0\circ T-\varphi)|_{\supp(\mu)}=0$. We conclude that $w_1\in S$. 
	 
	Since $w_0$ is the maximal element of $S$, we see that $w_1 \fless w_0$ and consequently $w_1=w_0$. We conclude that $w_0\in\twidehat{C}(X,\R)\smallsetminus\{\tinftyf,\,\taif\}$ and $w_0\circ T=w_0+\varphi$.
\end{proof}

We conclude this subsection with the following proof of Proposition~\ref{p:uniqueness of tropical eigenvalue}.

\begin{proof}[Proof of Proposition~\ref{p:uniqueness of tropical eigenvalue}]
Suppose $\cL_{A}(u)=\lambda\otimes u$ for some $u\in C(X,\Rminus)\smallsetminus\{\taif\}$ and some $\lambda\in\Rm$. Note that $u_0\=\oplusu{x\in X}u(x)\in\R$. It then follows from (\ref{eq: exchangability of numbers}) that
\begin{align*}
    \lim_{n\to+\infty}n^{-1}\Bigl(\oplusu{x\in X}\cL_A^n(u)(x)\Bigr) 
    & =\lim_{n\to+\infty}n^{-1}\Bigl(\oplusu{x\in X}((n\lambda)\otimes u(x))\Bigr) \\
    & =\lim_{n\to+\infty}n^{-1}((n\lambda)\otimes u_0) \\
    & =\lambda.
\end{align*}

For (i), further assume that $u\in C(X,\R)$. It follows from Lemma~\ref{l:projective Hilbert metric estimate} that \begin{equation*}
	\cL_{\toverline{A}}^k(\tmif)-\norm{u-\tmif}_{C^0}\fless\cL_{\toverline{A}}^k(u)\fless \norm{u-\tmif}_{C^0}+\cL_{\toverline{A}}^k(\tmif)
\end{equation*}
for all $k\in\N$. Thus it follows from \cite[Proposition~2.1]{Je06} that 
\begin{equation*}
    \lambda=\lim_{n\to+\infty}n^{-1}\oplusu{x\in X}\cL_A^n(u)(x)=\lim_{n\to+\infty}n^{-1}\oplusu{x\in X}\cL_A^n(\tmif)(x)=\energy.
\end{equation*} Now (i) is verified.

\smallskip
For (ii), it follows from Lemma~\ref{l:projective Hilbert metric estimate} that $\cL_A^k(u)\fless u_0\otimes\cL_A^k(\tmif)$ for all $k\in\N$. Thus it follows from (\ref{eq: exchangability of numbers}) and \cite[Proposition~2.1]{Je06} that
\begin{equation*}
    \lambda=\lim_{n\to+\infty}n^{-1}\oplusu{x\in X}\cL_A^n(u)(x)\leq\liminf_{n\to+\infty}n^{-1}\Bigl(u_0\otimes\Bigl(\oplusu{x\in X}\cL_A^n(\tmif)(x)\Bigr)\Bigr)=\energy.
\end{equation*}

Assume further that $T$ is transitive. We claim that $u\in C(X,\R)$.

Since $u\in C(X,\Rminus)\smallsetminus\{\taif\}$, we see that $\cL_A(u)\neq\taif$ and consequently $\lambda\neq \tai$. Suppose that $u(x_0)=\tai$ for some $x_0\in X$. Then it follows from $\cL_A(u)=\lambda\otimes u$, $\lambda\neq\tai$, and $A\in\aholder$ that $u(y)=\tai$ for all $y\in \bigcup_{n\in\N}T^{-n}(x_0)$. Since $T$ is transitive, $\bigcup_{n\in\N}T^{-n}(x_0)$ is dense in $X$. Thus, it follows from the continuity of $u$ that $u=\taif$, which is a contradiction. 

We conclude that $u(x)\neq\tai$ for all $x\in X$ and thus $u\in C(X,\R)$. It then follows from~(i) that $\lambda=\energy$. Now~(ii) is verified.

\smallskip
For (iii), suppose that $\cL_A^{\ostar}(b)=\lambda\otimes b$ for some $b\in \density(X)\smallsetminus\{\tinftyf,\,\taif\}$ and some $\lambda\in\Rm$. It follows from (\ref{eq: reason for adjoint}) and (\ref{eq: exchangability of numbers}) that
\begin{equation}\label{eq: eigendensity b asso with lambda}
	\oplusu{x\in X}(\cL_A(v)(x)\otimes b(x))=\oplusu{x\in X}\bigl(v(x)\otimes\cL_A^{\ostar}(b)(x)\bigr)=\lambda\otimes\Bigl(\oplusu{x\in X}(v(x)\otimes b(x))\Bigr)
\end{equation} for all $v\in C(X,\R )$. Thus, for all $n\in\N$,
\begin{equation*}
    \oplusu{x\in X}(\cL_A^n(\tmif)(x)\otimes b(x))=(n\lambda)\otimes\Bigl(\oplusu{x\in X}b(x)\Bigr).
\end{equation*}
Since $b\in \density(X)\smallsetminus\{\tinftyf,\,\taif\}$, we see that $b_0\=\oplusu{x\in X}b(x)\in\R$ and 
\begin{align*}
        \lambda
        =\lim_{n\to+\infty}n^{-1}((n\lambda)\otimes b_0)
        &=\lim_{n\to+\infty}n^{-1}\Bigl(\oplusu{x\in X}(\cL_A^{n}(\tmif)(x)\otimes b(x))\Bigr)\\
        &\leq \lim_{n\to+\infty}n^{-1}\Bigl(\Bigl(\oplusu{x\in X}\cL_A^n(\tmif)(x)\Bigr)\otimes b_0\Bigr)
        =\energy.
\end{align*}
Here the last identity follows from \cite[Proposition~2.1]{Je06}. Assume that $T$ is transitive. Recall from Proposition~\ref{p:main_construction}~(ii)(iii) that $v_{\tmif}\in\es{\cL_A}{C(X,\R)}{\energy}$. Take $v\=v_{\tmif}$ in~(\ref{eq: eigendensity b asso with lambda}). Since $b\in\density(X)\smallsetminus\{\tinftyf,\,\taif\}$, it follows that $\oplusu{x\in X} (v_{\tmif}(x)\otimes b(x) )\in\R$ and thus $\lambda=\energy$. Now (iii) is verified. 
\end{proof}

\subsection{Ma\~{n}\'{e} potential and representation}   \label{ss:Mane_potential}
In this subsection, properties of the Aubry set and the Ma\~{n}\'{e} potential are recalled. We establish Proposition~\ref{p:Mane_potential_properties} and Lemma~\ref{l:mane potential gives eigenmeasures} to prepare for the proof of the representation of tropical eigendensities in Subsection~\ref{ss: proof of D}.
\begin{definition}[Aubry set]\label{d:Aubry set}

Let $T\:X \rightarrow  X$ be an open continuous distance-expanding map on a compact metric space $(X,d)$. For a continuous potential $A\in C(X,\R)$, we call $x\in X$ an \emph{Aubry point} if for every $\epsilon>0$, there exists $y\in X$ and $n\in\N$ such that \begin{equation*}
	d(x,y)\leq\epsilon,\quad d(T^n(y),x)\leq\epsilon,\quad\text{ and } \quad\abs{S_n\toverline{A}(y)}\leq\epsilon.
\end{equation*}
The collection of all Aubry points in $X$ is called the \emph{Aubry set} and is denoted by $\Omega_A$. 
\end{definition}

Some basic properties about the Aubry set are discussed in \cite[Chapter~4]{Ga17} in the setting of subshifts of finite type, e.g., the nonemptiness, closedness, and $T$-invariance. The proofs for these facts work in our context and are omitted here.

\begin{lemma}\label{l: maximizing measures supported on Aubry set}
    Let $T\:X \rightarrow  X$ be an open continuous distance-expanding map on a compact metric space $(X,d)$, and $A\in\aholder$. Then an invariant probability measure is a maximizing measure for $A$ if and only if it is supported on the Aubry set $\Omega_A$.
\end{lemma}
Denote the set of nonwandering points by $\Omega(T)$. The above lemma can be obtained through the existence of a sub-action $v_A$ (cf.~Proposition~\ref{p: mane lemma}) and the following observations (cf.~\cite[Chapter 4]{Ga17}):
\begin{enumerate}[label=\rm{(\roman*)}]	
\smallskip
\item $\Omega_A=\Omega_{A+v_A-v_A\circ T}\subseteq (A+v_A-v_A\circ T)^{-1}(\energy)\eqqcolon K$ and
\smallskip
\item $\bigcap_{n\in N_0}T^{-n}(K)\cap\Omega(T)\subseteq\Omega_{A+v_A-v_A\circ T}$.
\end{enumerate}
A detailed proof of 
Lemma~\ref{l: maximizing measures supported on Aubry set} is contained in Appendix~\ref{ss: Appendix A}.

\begin{definition}[Ma\~{n}\'{e} potential]\label{d:Mane potential}
Let $T\:X \rightarrow  X$ be an open continuous distance-expanding map on a compact metric space $(X,d)$. For a potential $A\in \aholder$, the \emph{Ma\~{n}\'{e} potential} associated with $A$ is the function $\phi_A$ defined on $X\times X$ given by 
\begin{equation*}
\phi_A(x,y)\coloneqq\lim_{\epsilon\to0^+}\oplusu{n\in\N}\oplusu{d(z,x)\leq\epsilon\atop d(T^n(z),y)\leq\epsilon}S_n\toverline{A}(z)=\lim_{\epsilon\to0^+}\sup_{n\in\N}\sup_{d(z,x)\leq\epsilon\atop d(T^n(z),y)\leq\epsilon}S_n\toverline{A}(z).
\end{equation*}
\end{definition}

\begin{rem}\label{r: regularity of Mane potential}
It is straightforward to check that $\phi_A(\cdot,\cdot)$ is upper semi-continuous (see Lemma~\ref{l: mane potential is upper semi-continuous}). Recall from Lemma~\ref{l: uniform upper bound without transi and surje} that there exists $D>0$ such that $S_n \toverline{A}(x)\leq D$ for all $n\in\N$ and $x\in X$. Thus $\phi_A(\cdot,\cdot) \: X\times X \rightarrow \Rminus$ and $\phi_A(\cdot,y)\in\density(X)\smallsetminus\{\tinftyf\}$ for all $y\in X$.
\end{rem}

\begin{prop}[Properties of the Ma\~{n}\'{e} potential] \label{p:Mane_potential_properties}
Let $T\:X \rightarrow  X$ be an open continuous distance-expanding map on a compact metric space $(X,d)$, and $A \in \aholder$ with $\alpha \in (0,1]$. Then the following hold for all $x,\, y,\, z\in X$:
\begin{enumerate}[label=\rm{(\roman*)}]	
\smallskip
\item $u(x)\otimes\phi_A(x,y)\leq u(y)$ for all $u\in \es{\cL_A}{C(X,\Rminus)}{\energy}$.

\smallskip
\item $\phi_A(x,y)\otimes b(y)\leq b(x)$ for all $b\in\es{\cL_A^{\ostar}}{\density(X)}{\energy}$.

\smallskip
\item $\phi_A(x,z)\geq\phi_A(x,y)\otimes\phi_A(y,z)$.
\smallskip
\item $\phi_A(x,x)\leq 0$, and $\phi_A(x,x)=0$ if and only if $x\in\Omega_A$.
\smallskip
\item Assume $x\in\Omega_A$. Then $\phi_A(x,\cdot)$ is in $\es{\cL_A}{C(X,\Rminus)}{\energy}$ and 
\begin{equation*}
    \phi_A(x,z)\leq \phi_A(x,w)+\abs{A}_{d^\alpha}(\lambda^{\alpha}-1)^{-1}d(z,w)^\alpha
\end{equation*}
for all $w\in X$ satisfying $d(z,w)<\xi$, where $\xi>0$ is the constant in Lemma~\ref{l:expanding systems}. Moreover, if $T$ is transitive, then $\phi_A(x,\cdot)\in \aholder$.
\end{enumerate}
\end{prop}

While statements~(i), (iii), (iv), and~(v)  have appeared for some different settings in the literature (cf.~\cite[Proposition~23]{CLT01} and \cite[Proposition~5.2]{Ga17}), our formulations extend these results to functions in $C(X,\Rminus)$—a generalization resulting from relaxing the transitivity requirement for $T$. The proof of (i) requires only minor adaptations from existing arguments, and (v) follows reasoning similar to \cite{CLT01}; these are therefore deferred to Appendix~\ref{ss: Appendix A}. Here we focus on proving (ii)--(iv): (iii) and~(iv) are included here due to their concise derivations from definitions, while (ii) presents a novel result for tropical eigendensities.

\begin{proof}[Proof of Proposition~\ref{p:Mane_potential_properties}~(ii)--(iv)]
(ii) Since $b\in\es{\cL_A^{\ostar}}{\density(X)}{\energy}$, $b$ is upper semi-continuous and $b(T(x))\otimes\toverline{A}(x)=b(x)$ for all $x\in X$. Thus, for all $x\in X$ and $n\in\N$,
\begin{equation}\label{eq:tropical eigendensity along the trajectory}
	b(T^n(x))\otimes S_n\toverline{A}(x)=b(x) .
\end{equation}

Denote $\widetilde{\phi}_A(x,y)\=\lim\limits_{\epsilon\to0^+}\oplusu{n\in\N}\oplusu{d(y_0,x)\leq\epsilon\atop T^n(y_0)=y}S_n\toverline{A}(y_0)$. 

\smallskip
\emph{Claim.} $\widetilde{\phi}_A\equiv\phi_A$.
\smallskip

We first prove the claim. It immediately follows from Definition~\ref{d:Mane potential} that 
$\phi_A(x,y)
\geq \widetilde\phi_A(x,y)$ for all $x,\, y\in X$.
So it suffices to show that $\widetilde\phi_A(x,y)\geq\phi_A(x,y)$ for all $x,\, y\in X$.

Fix $x,\, y\in X$. Consider $n\in\N$, and $\epsilon\in(0,\xi)$, where $\xi$ is the constant in Lemma~\ref{l:expanding systems}. For every $z\in X$ satisfying $d(z,x)\leq\epsilon/2$ and $d(T^n(z),y)\leq\epsilon/2$, it follows from Lemma~\ref{l:expanding systems} that there exists $y_0\in X$ satisfying that $T^n(y_0)=y$ and that for each integer $0\leq i\leq n$, $d\bigl(T^i(z),T^i(y_0)\bigr)\leq\lambda^{-n+i}\epsilon/2$. Thus, by (\ref{ineq:difference of the nearby trajectory sum}),
 \begin{equation*}
\abs{S_n\toverline{A}(z)-S_n\toverline{A}(y_0)}
\leq\frac{\epsilon^\alpha\abs{A}_{d^\alpha}}{2^\alpha(\lambda^\alpha-1)}.
\end{equation*} 
Note that $d(y_0,x)\leq d(z,x)+d(z,y_0)\leq\epsilon$.
We conclude that for all $n\in\N$ and $\epsilon\in(0,\xi)$,
\begin{equation*}
\oplusu{d(z,x)\leq\epsilon/2\atop d(T^n(z),y)\leq\epsilon/2}S_n\toverline{A}(z)
\leq\Bigl(\oplusu{d(y_0,x)\leq\epsilon\atop T^n(y_0)=y}S_n\toverline{A}(y_0)\Bigr)\otimes\frac{\epsilon^\alpha\abs{A}_{d^\alpha}}{2^\alpha(\lambda^\alpha-1)} .
\end{equation*}

Thus, 
\begin{align*}
\phi_A(x,y)
&=\lim_{\epsilon\to0^+}\oplusu{n\in\N}\oplusu{d(z,x)\leq\epsilon/2\atop d(T^n(z),y)\leq\epsilon/2}S_n\toverline{A}(z) \\
&\leq \lim_{\epsilon\to0^+}\Bigl(\oplusu{n\in\N}\oplusu{d(y_0,x)\leq\epsilon\atop T^n(y_0)=y}S_n\toverline{A}(y_0)\Bigr)\otimes\frac{\epsilon^\alpha\abs{A}_{d^\alpha}}{2^\alpha(\lambda^\alpha-1)} \\
&=\widetilde{\phi}_A(x,y) ,
\end{align*} 
and the claim follows.

\smallskip
It follows from the claim, (\ref{eq:tropical eigendensity along the trajectory}), and the upper semi-continuity of $b$ that
\begin{align*}
\phi_A(x,y)\otimes b(y)
& =\lim_{\epsilon\to0^+}\oplusu{n\in\N}\oplusu{d(z,x)\leq\epsilon\atop T^n(z)=y}(S_n\toverline{A}(z)\otimes b(y))  \\
& =\lim_{\epsilon\to 0^+}\oplusu{n\in\N}\oplusu{d(z,x)\leq\epsilon\atop T^n(z)=y}b(z)
\leq\lim_{\epsilon\to 0^+}\oplusu{d(z,x)
\leq\epsilon}b(z)
\leq b(x).
\end{align*}

\smallskip

(iii) In the following argument, we use Lemma~\ref{l:expanding systems} to connect two trajectories when the end of one trajectory is close to the beginning of the other. 

Fix $x,\, y,\, z\in X$, and $\epsilon\in(0,\xi/2)$. Assume that $w_1\in B(x,\epsilon)$ and $n\in\N$ satisfy $d(T^n(w_1),y)\leq\epsilon$, and that $w_2\in B(y,\epsilon)$ and $m\in\N$ satisfy $d(T^m(w_2),z)\leq\epsilon$. Note that $d(w_2,T^n(w_1))\leq d(w_2,y)+d(y,T^n(w_1))\leq 2\epsilon<\xi$.
Lemma~\ref{l:expanding systems} then implies that 
\begin{equation*}
	d\bigl(T_{w_1}^{-n}(w_2),x\bigr)\leq d\bigl(T_{w_1}^{-n}(w_2),w_1\bigr)+d(w_1,x)\leq d(w_2,T^n(w_1))+d(w_1,x)\leq 3\epsilon .
\end{equation*} 
Moreover, it follows from (\ref{ineq:difference of the nearby trajectory sum}) that 
\begin{equation*}
	\Absbig{S_n\toverline{A}(w_1)-S_n\toverline{A}\bigl(T_{w_1}^{-n}(w_2)\bigr)}\leq(\lambda^{\alpha}-1)^{-1}\abs{A}_{d^\alpha}(2\epsilon)^\alpha.
\end{equation*}
We conclude that $S_{n+m}\toverline{A}\bigl(T_{w_1}^{-n}(w_2)\bigr)+(\lambda^{\alpha}-1)^{-1}\abs{A}_{d^\alpha}(2\epsilon)^{\alpha}\geq S_n\toverline{A}(w_1)+S_m\toverline{A}(w_2)$ and thus
\begin{equation*}
 \begin{aligned}
        &\Bigl(\oplusu{n\in\N}\oplusu{d(w_1,x)\leq\epsilon/2\atop d(T^n(w_1),y)\leq\epsilon/2}S_n\toverline{A}(w_1)\Bigr)+\Bigl(\oplusu{m\in\N}\oplusu{d(w_2,y)\leq\epsilon/2\atop d(T^m(w_2),z)\leq\epsilon/2}S_m\toverline{A}(w_2)\Bigr) \\
        &\qquad \leq (\lambda^\alpha-1)^{-1}\abs{A}_{d^\alpha}(2\epsilon)^\alpha+\Bigl(\oplusu{k\in\N}\oplusu{d(w,x)\leq 3\epsilon\atop d(T^k(w),z)\leq 3\epsilon}S_{k}\toverline{A}(w)\Bigr). 
 \end{aligned}
\end{equation*}
Now it is straightforward to check that (iii) follows from the definition of $\phi_A$ (cf.\ Definition~\ref{d:Mane potential}).

\smallskip

(iv) It follows from (iii) that $\phi_A(x,x)\otimes\phi_A(x,x)\leq\phi_A(x,x)$ for all $x\in X$. Recall that $\phi_A:X\times X\to\Rminus$ (cf.~Remark~\ref{r: regularity of Mane potential}). It follows that $\phi_A(x,x)\leq 0$ for all $x\in X$. The second part of the statement is a direct consequence of Lemma~\ref{l:expanding systems}, (\ref{ineq:difference of the nearby trajectory sum}), and the definitions of the Aubry set and the Ma\~{n}\'{e} potential. 
\end{proof}

\begin{prop}[Representation of eigenfunctions]   \label{p:eigenfunction}
Let $T\:X \rightarrow  X$ be an open continuous distance-expanding map on a compact metric space $(X,d)$, and $A \in \aholder$ with $\alpha \in (0,1]$. Then for all $u\in \es{\cL_A}{C(X,\Rminus)}{\energy}$ and $y\in X$, we have 
\begin{equation*}
u(y)=\oplusu{x\in\Omega_A}(u(x)\otimes\phi_A(x,y)).
\end{equation*}
\end{prop}
A version of this proposition for subshifts of finite type is established in \cite[Proposition~6.2~
(iii)]{Ga17}. We extend it to functions in $C(X,\Rminus)$ and include a proof in the present setting in Appendix~\ref{ss: Appendix A} for the reader's convenience.

\begin{lemma} [Ma\~n\'e potential as eigendensity] \label{l:mane potential gives eigenmeasures}
Let $T\:X \rightarrow  X$ be an open continuous distance-expanding map on a compact metric space $(X,d)$, and $A \in \aholder$ with $\alpha \in (0,1]$. The Ma\~{n}\'{e} potential satisfies 
\begin{equation}\label{eq: Mane potential (_,y) almost eigendensity}
	\phi_A(T(x),y)=\phi_A(x,y)-\toverline{A}(x)
\end{equation}
 for all $x,\, y\in X$ with $T(x)\neq y$. Moreover, for every $y\in\Omega_A$, the equality (\ref{eq: Mane potential (_,y) almost eigendensity}) holds for every $x\in X$, i.e., $\phi_A(\cdot,y)$ is in $\es{\cL_A^{\ostar}}{\density(X)}{\energy}$.
\end{lemma}

While the first part of this lemma incorporates ideas from \cite[Proposition~5.3]{Ga17} for subshifts of finite type, the second part verifies (\ref{eq: Mane potential (_,y) almost eigendensity}) for points in the Aubry set, yielding a novel construction of tropical eigendensities, essential for their representations.

\begin{proof}
Let $\xi>0$ be the constant in Lemma~\ref{l:expanding systems}. The condition $T(x)\neq y$ implies that for all $\epsilon\in \bigl(0,\frac{1}{2}\min \bigl\{\xi,\, d\bigl(x,T^{-1}(y) \bigr)\bigr\}\bigr)$, if $z\in B(x,\epsilon)$ and $n\in\N$ satisfy $d(T^n(z),y)<\epsilon$, then $n>1$. Thus, (\ref{eq: Mane potential (_,y) almost eigendensity}) immediately follows from the definition of $\phi_A$ (cf.\ Definition~\ref{d:Mane potential}) and the continuity of $T$. This proves the first part of the lemma.

For the second part of the lemma, note that $\phi_A(\cdot,\cdot) \: X\times X \rightarrow \Rminus$ is upper semi-continuous (see Lemma~\ref{l: mane potential is upper semi-continuous}). Thus, $\phi_A(\cdot,y)\in\density(X)$ for all $y\in \Omega_A$, and it suffices to show $\phi_A(T(x),y)=\phi_A(x,y)-\toverline{A}(x)$ for all $y\in\Omega_A$ and $x\in T^{-1}(y)$.
 Fix $y\in\Omega_A$ and $x\in T^{-1}(y)$. By Proposition~\ref{p:Mane_potential_properties}~(iv), we have $\phi_A(T(x),y)=\phi_A(y,y)=0$. Thus, it suffices to prove $\phi_A(x,T(x))=\toverline{A}(x)$. 
 
 \smallskip
 \emph{Claim.} $\phi_A(x,T(x))=\toverline{A}(x)$ for all $x\in X$.
 \smallskip
 
 Fix $x\in X$. On the one hand, $S_n\toverline{A}(z)=\toverline{A}(x)$ for $n\=1$, $z\=x$, and it follows from Definition~\ref{d:Mane potential} that $\phi_A(x,T(x))\geq \toverline{A}(x)$. 

On the other hand, since $T$ is continuous, for every $\epsilon>0$, there exists $\eta(\epsilon)\in(0,\epsilon)$ such that $d(y,x)\leq\eta(\epsilon)$ implies $d(T(y),T(x))\leq\epsilon$ for all $y\in X$. Thus,
\begin{align*}
&\oplusu{n\in\N}\oplusu{d(z,x)\leq\eta(\epsilon)\atop d(T^n(z),T(x))
	\leq\eta(\epsilon)}S_n\toverline{A}(z) \\
&\qquad \leq\Bigl(\oplusu{d(z,x)\leq\eta(\epsilon)}\toverline{A}(z)\Bigr)\otimes\Bigl(\oplusu{n\in\N}\oplusu{d(z,x)\leq\eta(\epsilon)\atop d(T^{n}(z),T(x))\leq\eta(\epsilon)}S_{n-1}\toverline{A}(T(z))\Bigr)\\
&\qquad\leq\Bigl(\oplusu{d(z,x)\leq\eta(\epsilon)}\toverline{A}(z)\Bigr)\otimes\Bigl(\oplusu{m\in\N_0}\oplusu{d(\tz,T(x))\leq\epsilon\atop d(T^m(\tz),T(x))\leq\epsilon}S_m\toverline{A}(\tz)\Bigr).
\end{align*}

Recall that $S_0\toverline{A}(\tz)=0$ for all $\tz\in X$. As $\epsilon\to 0^+$, we get
\begin{equation*}
\phi_A(x,T(x))\leq\toverline{A}(x)\otimes(0\oplus\phi_A(T(x),T(x))).
\end{equation*}
By Proposition~\ref{p:Mane_potential_properties}~(iv), $\phi_A(T(x),T(x))\leq 0$. We conclude that $\phi_A(x,T(x))\leq\toverline{A}(x)$, and the claim is now verified. 
\end{proof}

\begin{lemma}\label{l: the linear constant between Mane(x,_) or Mane(_,y)}
Let $T\:X \rightarrow  X$ be an open continuous distance-expanding map on a compact metric space $(X,d)$, and $A \in \aholder$ with $\alpha \in (0,1]$. If $\phi_A(x,y)\otimes\phi_A(y,x)=0$ for all $x,\, y\in\Omega_A$, then 
\begin{equation*}
    \phi_A(x,\cdot)=\phi_A(x,y)\otimes\phi_A(y,\cdot) \quad \text{ and } \quad
    \phi_A(\cdot,x)=\phi_A(\cdot,y)\otimes\phi_A(y,x)
\end{equation*}
for all $x,\, y\in\Omega_A$. 
\end{lemma}
\begin{proof}
Fix $x,\, y\in\Omega_A$. 
By Proposition~\ref{p:Mane_potential_properties}~(iii), we have 
\begin{equation*}
	\phi_A(y,x)\otimes \phi_A(x,\cdot)\fless \phi_A(y,\cdot)\quad\text{ and }\quad\phi_A(x,y)\otimes \phi_A(y,\cdot)\fless\phi_A(x,\cdot).
\end{equation*}
The above two inequalities, together with $\phi_A(x,y)\otimes\phi_A(y,x)=0$, imply 
$\phi_A(x,\cdot)=\phi_A(x,y)\otimes \phi_A(y,\cdot)$.
By Proposition~\ref{p:Mane_potential_properties}~(iii), we have
\begin{align*}
	\phi_A(\cdot,x)\otimes\phi_A(x,y)\fless\phi_A(\cdot,y)\quad\text{ and }\quad\phi_A(\cdot,y)\otimes\phi_A(y,x)\leq\phi_A(\cdot,x).
\end{align*}
Therefore, $\phi_A(\cdot,x)=\phi_A(\cdot,y)\otimes\phi_A(y,x)$ follows from the above two inequalities and $\phi_A(x,y)\otimes\phi_A(y,x)=0$.
\end{proof}

The following characterizations will be useful in the proof of Theorem~\ref{t:representation}~(iii).

\begin{prop}\label{p: equivalence for uniqueness of Mane (x,_) and (_,y)}
Let $T\:X \rightarrow  X$ be an open continuous distance-expanding map on a compact metric space $(X,d)$, and $A \in \aholder$ with $\alpha \in (0,1]$. Then the following statements are equivalent:
\begin{enumerate}[label=\rm{(\roman*)}]	
\smallskip
\item The entries of $\{\phi_A(x,\cdot)\}_{x\in\Omega_A}$ are the same up to a tropical multiplicative constant.
\smallskip
\item The entries of $\{\phi_A(\cdot,y)\}_{y\in\Omega_A}$ are the same up to a tropical multiplicative constant.
\smallskip
\item For all $x,\, y\in\Omega_A$, $\phi_A(x,y)\otimes\phi_A(y,x)=0$.
\end{enumerate}
\end{prop}
\begin{proof}
	Fix $x,\, y\in\Omega_A$.
	
	To see that (i) implies (iii), suppose $\phi_A(x,\cdot)\otimes c=\phi_A(y,\cdot)$. It follows that $\phi_A(x,x)\otimes c=\phi_A(y,x)$ and $\phi_A(x,y)\otimes c=\phi_A(y,y)$. By Proposition~\ref{p:Mane_potential_properties}~(iv), we get $c=\phi_A(y,x)$ and $\phi_A(x,y)\otimes c=0$. Thus, $\phi_A(x,y)\otimes\phi_A(y,x)=0$. 

	To see that (ii) implies (iii), suppose $\phi_A(\cdot,x)\otimes d=\phi_A(\cdot,y)$. It follows that $\phi_A(x,x)\otimes d=\phi_A(x,y)$ and $\phi_A(y,x)\otimes d=\phi_A(y,y)$. By Proposition~\ref{p:Mane_potential_properties}~(iv), we get $d=\phi_A(x,y)$ and $\phi_A(y,x)\otimes d=0$. Thus, $\phi_A(y,x)\otimes\phi_A(x,y)=0$. 

	That (iii) implies (i) and (ii) follows from Lemma~\ref{l: the linear constant between Mane(x,_) or Mane(_,y)}.
\end{proof}

\subsection{Proof of Theorem~\ref{t:representation}}\label{ss: proof of D}
We discover the representation of tropical eigendensities of $\cL_A^{\ostar}$ associated with eigenvalue $\energy$ through the duality in $\phi_A(\cdot,\cdot)$. 

\begin{proof}[\bf Proof of Theorem~\ref{t:representation}]
(i) It follows from Proposition~\ref{p:eigenfunction} that $v(\cdot)=\oplusu{x\in\Omega_A}(v(x)\otimes\phi_A(x,\cdot))$ for all $v\in\es{\cL_A}{C(X,\Rminus)}{\energy}$. 

Fix an arbitrary $c\in C(\Omega_A,\Rminus)$. Recall $f_c(\cdot)\=\oplusu{x\in\Omega_A}(c(x)\otimes\phi_A(x,\cdot))$. We show that $f_c$ is in $\es{\cL_A}{C(X,\Rminus)}{\energy}$. Note that $c$ has an upper bound since $X$ is compact. Since $\phi_A$ also has an upper bound (cf.~Remark~\ref{r: regularity of Mane potential}), it follows that $f_c\:X\to\Rminus$. Let $\xi$ be the constant in Lemma~\ref{l:expanding systems}. It follows from Proposition~\ref{p:Mane_potential_properties}~(v) that for all $z_1,\, z_2\in X$ with $d(z_1,z_2)<\xi$,
\begin{equation*}
    f_c(z_1)\leq f_c(z_2)+\abs{A}_{d^\alpha}(\lambda^{\alpha}-1)^{-1}d(z_1,z_2)^\alpha
\end{equation*}and that for all $z\in X$, 
\begin{align*}
    \cL_A(f_c)(z)
    &=\oplusu{y\in T^{-1}(z)}(f_c(y)\otimes A(y)) \\
    &=\oplusu{y\in T^{-1}(z)}\oplusu{x\in\Omega_A}(c(x)\otimes\phi_A(x,y)\otimes A(y))\\
    &=\oplusu{x\in\Omega_A}\oplusu{y\in T^{-1}(z)}(c(x)\otimes\phi_A(x,y)\otimes A(y))\\
    &=\oplusu{x\in\Omega_A}(c(x)\otimes\cL_A(\phi_A(x,\cdot))(z))\\
    &=\oplusu{x\in\Omega_A}(c(x)\otimes\phi_A(x,z)\otimes\energy)\\
    &=\energy\otimes f_c(z).
\end{align*}
We conclude that $f_c$ is in $\es{\cL_A}{C(X,\Rminus)}{\energy}$. 

In addition, for some $v\in\es{\cL_A}{C(X,\Rminus)}{\energy}$, if there exists $c\in C(\Omega_A,\Rminus)$ such that $v=f_c$ and $c(x)\otimes\phi_A(x,y)\leq c(y)$ for all $x,\, y\in\Omega_A$, then it follows from this inequality and $\phi_A(x,x)=0$ for $x\in\Omega_A$ (cf.~Proposition~\ref{p:Mane_potential_properties}~(iv)) that for all $z\in\Omega_A$,
\begin{equation*}
    c(z)=\oplusu{x\in\Omega_A}(c(x)\otimes\phi_A(x,z))=f_c(z)=v(z),
\end{equation*}
i.e., $c=v|_{\Omega_A}$. Since $v\in \es{\cL_A}{C(X,\Rminus)}{\energy}$, its restriction $v|_{\Omega_A}\in C(\Omega_A,\Rminus)$ satisfies $v|_{\Omega_A}(x)\otimes\phi_A(x,y)\leq v|_{\Omega_A}(y)$ for all $x,\, y\in\Omega_A$ (by Proposition~\ref{p:Mane_potential_properties}~(i)) and $f_{v|_{\Omega_A}}=v$ (by Proposition~\ref{p:eigenfunction}). Now (i) is verified.

\smallskip
(ii) For (a), assume that $b\in\es{\cL_A^{\ostar}}{\density(X)}{\energy}$. We need to show 
\begin{equation}\label{eq: representation of eigendensity as functionals}
\oplusu{x\in X}(u(x)\otimes b(x))=\oplusu{x\in X,y\in\Omega_A}(u(x)\otimes\phi_A(x,y)\otimes b(y))
\end{equation} for all $u\in C(X,\R)$. Now fix an arbitrary $u\in C(X,\R)$. 

We first reduce the proof of $(\ref{eq: representation of eigendensity as functionals})$ to functions in $\es{\cL_A}{C(X,\Rminus)}{\energy}$ using Corollary~\ref{c:more_of_constructions} and then apply Proposition~\ref{p:eigenfunction}.

For the left-hand side of (\ref{eq: representation of eigendensity as functionals}), since $b\in\es{\cL_A^{\ostar}}{\density(X)}{\energy}$, it follows from (\ref{eq: reason for adjoint}) that 
\begin{equation}\label{eq: functional induced by eigendensity}
	\oplusu{x\in X}(\cL_A(u)(x)\otimes b(x))=\oplusu{y\in X}\bigl(u(y)\otimes \cL_{A}^{\ostar}(b)(y)\bigr)=\oplusu{y\in X}(u(y)\otimes\energy\otimes b(y)).
\end{equation} 

By repeated use of (\ref{eq: functional induced by eigendensity}), we get 
\begin{equation}\label{eq: pre for limit for iterations of L_A_n(u)}
\oplusu{x\in X}(u(x)\otimes b(x))
=\oplusu{x\in X}\bigl(\cL_{\toverline{A}}^n(u)(x)\otimes b(x)\bigr)
=\oplusu{x\in X}\Bigl(\Bigl(\oplusu{m\geq n}\cL_{\toverline{A}}^m(u)(x)\Bigr)\otimes b(x)\Bigr)
\end{equation}
for all $n\in\N$. Recall that the modulus of continuity $\omega_u(\cdot)\:(0,+\infty)\to\R$ is nondecreasing. It then follows from Lemmas~\ref{l: uniform upper bound without transi and surje} and~\ref{l: generalized form of exchange lemma}~(ii) that $\bigl\{\oplusu{m\geq n}\cL_{\toverline{A}}^m(u)\bigr\}_{n\in\N}$ has a uniform upper bound. It follows from (\ref{ineq: equi estimate for L_A_n(u)}) that $\bigl\{\oplusu{m\geq n}\cL_{\toverline{A}}^m(u)\bigr\}_{n\in\N}$ is equicontinuous. Recall that $v_u$ is the pointwise decreasing limit of $\oplusu{m\geq n}\cL_{\toverline{A}}^m(u)$ as $n\to+\infty$ (cf.~(\ref{eq: defi of v_u})). Thus, $\bigl\{\exp \bigl(\oplusu{m\geq n}\cL_{\toverline{A}}^m(u)\bigr)\bigr\}_{n\in\N}$ is a normal family and uniformly converges to $e^{v_u}$ as $n\to+\infty$. It then follows from (\ref{eq: pre for limit for iterations of L_A_n(u)}) that 
\begin{equation*}
    \begin{aligned}
    \exp\Bigl(\oplusu{x\in X}(u(x)\otimes b(x))\Bigr)&=\lim_{n\to+\infty}\exp\Bigl(\oplusu{x\in X}\Bigl(\Bigl(\oplusu{m\geq n}\cL_{\toverline{A}}^m(u)(x)\Bigr)\otimes b(x)\Bigr)\Bigr)\\
    &=\lim_{n\to+\infty}\oplusu{x\in X}\Bigl(\exp\Bigl(\oplusu{m\geq n}\cL_{\toverline{A}}^m(u)(x)\Bigr)\cdot e^{b(x)}\Bigr)\\
    &=\oplusu{x\in X}\bigl(e^{v_u(x)}\cdot e^{b(x)}\bigr)\\
    &=\exp\Bigl(\oplusu{x\in X}(v_u(x)\otimes b(x))\Bigr),
    \end{aligned}
\end{equation*}
where the third equality holds since $\exp b$ either is equal to $\tinftyf$ (when $b=\tinftyf$) or has an upper bound (when $b\in\density(X)\smallsetminus\{\tinftyf\}$).

We conclude that 
\begin{equation}\label{eq: reduction to eigenfunctions}
	\oplusu{x\in X}(u(x)\otimes b(x))=\oplusu{x\in X}(v_u(x)\otimes b(x)).
\end{equation}

For the right-hand side of (\ref{eq: representation of eigendensity as functionals}), by Lemma~\ref{l:mane potential gives eigenmeasures}, $\phi_A(\cdot,y)\in\es{\cL_A^{\ostar}}{\density(X)}{\energy}$ for all $y\in\Omega_A$. It follows from Proposition~\ref{p:Mane_potential_properties}~(iv) that $\phi_A(y,y)=0$ and $\phi_A(\cdot,y)\in\density(X)\smallsetminus\{\tinftyf\}$ for all $y\in\Omega_A$. Thus, we can substitute $b(\cdot)$ in (\ref{eq: reduction to eigenfunctions}) with  $\phi_A(\cdot,y)$ and it follows that for all $y\in\Omega_A$,
\begin{equation*}
\oplusu{x\in X}(u(x)\otimes \phi_A(x,y))=\oplusu{x\in X}(v_u(x)\otimes\phi_A(x,y)).
\end{equation*} 
Hence,
\begin{align*}
\oplusu{x\in X,\, y\in\Omega_A}(u(x)\otimes\phi_A(x,y)\otimes b(y))
&=\oplusu{y\in\Omega_A}\Bigl(\oplusu{x\in X}(\phi_A(x,y)\otimes u(x))\otimes b(y)\Bigr)\\
&=\oplusu{y\in\Omega_A}\Bigl(\oplusu{x\in X}(\phi_A(x,y)\otimes v_u(x))\otimes b(y)\Bigr)\\
&=\oplusu{x\in X,\, y\in\Omega_A}(v_u(x)\otimes\phi_A(x,y)\otimes b(y)).
\end{align*}

We have finished the first step of reduction. Since $v_u\in\es{\cL_A}{C(X,\Rminus)}{\energy}$ (cf.~Corollary~\ref{c:more_of_constructions}~(ii)), it suffices to prove (\ref{eq: representation of eigendensity as functionals}) for all $v\in\es{\cL_A}{C(X,\Rminus)}{\energy}$.

It follows from the discussions below that for all $v\in\es{\cL_A}{C(X,\Rminus)}{\energy}$,
\begin{align*}
\oplusu{x\in X,\, y\in\Omega_A}(v(x)\otimes\phi_A(x,y)\otimes b(y))
&=\oplusu{x\in X,\, y,\, z\in\Omega_A}(v(z)\otimes\phi_A(z,x)\otimes\phi_A(x,y)\otimes b(y))\\
&=\oplusu{y,\, z\in\Omega_A}(v(z)\otimes \phi_A(z,y)\otimes b(y))\\
&=\oplusu{y\in X,\, z\in\Omega_A}(v(z)\otimes \phi_A(z,y) \otimes b(y))\\
&=\oplusu{y\in X}(v(y)\otimes b(y)).
\end{align*}
Here the first and the fourth identities follow from $v(x)=\oplusu{z\in\Omega_A}(v(z)\otimes\phi_A(z,x))$ for all $v\in\es{\cL_A}{C(X,\Rminus)}{\energy}$ and $x\in X$ (cf.~ Proposition~\ref{p:eigenfunction}). The second identity follows from $\phi_A(z,y)\leq \phi_A(z,x)\otimes\phi_A(x,y)$ for $x,\, y,\, z\in X$ (cf.~Proposition~\ref {p:Mane_potential_properties}~(iii)) and $\phi_A(z,z)=0$ for $z\in\Omega_A$ (cf.~Proposition~\ref{p:Mane_potential_properties}~(iv)).
For the third identity, we remark that
$\phi_A(z,y)\otimes b(y)\leq b(z)$ for all $y,\, z\in X$ (cf.~Proposition~\ref{p:Mane_potential_properties}~(ii)) and if $z\in \Omega_A$, then the equality is attained at $y=z\in\Omega_A$ as $\phi_A(z,z)=0$ (cf.~Proposition~\ref{p:Mane_potential_properties}~(iv)). Thus, for all $z\in\Omega_A$,
\begin{equation*}
\oplusu{y\in\Omega_A}(\phi_A(z,y)\otimes b(y))=\oplusu{y\in X}(\phi_A(z,y)\otimes b(y))=b(z).
\end{equation*}
We conclude that (\ref{eq: representation of eigendensity as functionals}) holds for all $v\in\es{\cL_A}{C(X,\Rminus)}{\energy}$ and thus for all $u\in C(X,\R)$.

Now further assume that $b\neq\tinftyf$. We need to show that $b(x)=\oplusu{y\in\Omega_A}(\phi_A(x,y)\otimes b(y))$ for all $x\in X$. Recall that for all $b_1\:X\to\Rm$, there exists a unique $b_2\in\density(X)$ equivalent to $b_1$ (cf.~Remark~\ref{r:density and functional}). Since we have shown $b(\cdot)$ is equivalent to $\oplusu{y\in\Omega_A}(\phi_A(\cdot,y)\otimes b(y))$, it suffices to show that $\oplusu{y\in\Omega_A}(\phi_A(\cdot,y)\otimes b(y))\in\density(X)$. Since $b\in\density(X)\smallsetminus\{\tinftyf\}$, $b$ has an upper bound. Note that $\phi_A$ also has an upper bound (cf.~Remark~\ref{r: regularity of Mane potential}). We see that $\oplusu{y\in\Omega_A}(\phi_A(\cdot,y)\otimes b(y))\: X\to\Rminus$ has an upper bound. Thus, it suffices to show that $\oplusu{y\in\Omega_A}(\phi_A(\cdot,y)\otimes b(y))$ is upper semi-continuous, i.e., for all $x\in X$,
\begin{equation}\label{ineq: upper semi-continuity of representation}
    \lim_{\epsilon\to 0^+}\oplusu{z\in B(x,\epsilon)}\oplusu{y\in\Omega_A}(\phi_A(z,y)\otimes b(y))\leq\oplusu{y\in\Omega_A}(\phi_A(x,y)\otimes b(y)).
\end{equation}

Fix $x\in X$. Note that the limit in the left-hand side of (\ref{ineq: upper semi-continuity of representation}) is a nondecreasing limit. It suffices to show that for all $\delta>0$, there exists $\epsilon>0$ such that 
\begin{equation*}
    \oplusu{z\in B(x,\epsilon)}\oplusu{y\in\Omega_A}(\phi_A(z,y)\otimes b(y))\leq\log\Bigl(\exp\Bigl(\oplusu{y\in\Omega_A}(\phi_A(x,y)\otimes b(y))\Bigr)+\delta\Bigr) .
\end{equation*}
We argue by contradiction. Suppose that for some $\delta>0$ and all $n\in\N$, there exists $z_n\in B(x,1/n)$ such that 
\begin{equation*}
    \oplusu{y\in\Omega_A}(\phi_A(z_n,y)\otimes b(y))>\log\Bigl(\exp\Bigl(\oplusu{y\in\Omega_A}(\phi_A(x,y)\otimes b(y))\Bigr)+\delta\Bigr).
\end{equation*} Note that $\phi_A$ is upper semi-continuous (see Lemma~\ref{l: mane potential is upper semi-continuous}) and $b\in\density(X)$ is also upper semi-continuous. Recall that $\Omega_A\subseteq X$ is closed and $X$ is compact. It follows that for each $n\in\N$, there exists $y_n\in\Omega_A$ such that $\phi_A(z_n,y_n)\otimes b(y_n)=\oplusu{y\in\Omega_A}(\phi_A(z_n,y)\otimes b(y))$. Furthermore, there exists a subsequence $\{y_{n_k}\}_{k\in\N}$ (with $\lim_{k\to+\infty}n_k=+\infty$) converging to $y'\in\Omega_A$ as $k\to+\infty$. It then follows from the upper semi-continuity of $\phi_A$ and $b$ that
\begin{equation}\label{ineq: contra for upper semi of representation}
\begin{aligned}
    \log\Bigl(\exp\Bigl(\oplusu{y\in\Omega_A}(\phi_A(x,y)\otimes b(y))\Bigr)+\delta\Bigr)&\leq\limsup_{k\to+\infty}\oplusu{y\in\Omega_A}(\phi_A(z_{n_k},y)\otimes b(y))\\&=\limsup_{k\to+\infty}\phi_A(z_{n_k},y_{n_k})\otimes b(y_{n_k})\\
    &\leq \phi_A(x,y')\otimes b(y')\\&\leq\oplusu{y\in\Omega_A}(\phi_A(x,y)\otimes b(y)).
\end{aligned}
\end{equation}
 Recall that $\oplusu{y\in\Omega_A}(\phi_A(\cdot,y)\otimes b(y))\: X\to\Rminus$. It follows that (\ref{ineq: contra for upper semi of representation}) forms a contradiction and (\ref{ineq: upper semi-continuity of representation}) is verified. We conclude that $\oplusu{y\in\Omega_A}(\phi_A(\cdot,y)\otimes b(y))\in\density(X)$ and $b(x)=\oplusu{y\in\Omega_A}(\phi_A(x,y)\otimes b(y))$ for all $x\in X$.

For (b), fix $c\in\Rm^{\Omega_A}$ and denote $K_c(f)\=\oplusu{x\in X,y\in\Omega_A}(f(x)\otimes\phi_A(x,y)\otimes c(y))$ for all $f\in \Rm^X$. It is clear from the definition that $K_c$ is a tropical linear functional, and consequently, by Proposition~\ref{p:compactness implies continuity}, $K_c$ is tropically continuous. Since $\phi_A(\cdot,y)\in\es{\cL_A^{\ostar}}{\density(X)}{\energy}$ for all $y\in\Omega_A$ (cf.~Lemma~\ref{l:mane potential gives eigenmeasures}), it follows from (\ref{eq: reason for adjoint}) that for all $f\in C(X,\R)$,
\begin{align*}
    K_c\bigl(\cL_{\toverline{A}}(f)\bigr)
    &=\oplusu{y\in \Omega_A}\Bigl(\Bigl(\oplusu{x\in X}(\cL_{\toverline{A}}(f)(x)\otimes\phi_A(x,y))\Bigr)\otimes c(y)\Bigr)\\
    &=\oplusu{y\in\Omega_A}\Bigl(\Bigl(\oplusu{x\in X}(f(x)\otimes\phi_A(x,y))\Bigr)\otimes c(y)\Bigr) \\
    &=K_c(f).
\end{align*} 
It then follows from (\ref{eq: reason for adjoint}) that the density $b'\in \density(X)$ of $K_c$ is equivalent to $\cL_{\toverline{A}}^{\ostar}(b')$. Recall that $\cL_{\toverline{A}}^{\ostar}(b')\in\density(X)$ (cf.~(\ref{eq:definition of tropical adjoint})) and that for every $b_1\:X\to\Rm$, there exists a unique $b_2\in\density(X)$ equivalent to $b_1$ (cf.~Remark~\ref{r:density and functional}). Thus, $b'=\cL_{\toverline{A}}^{\ostar}(b')$ and $b'\in\es{\cL_A^{\ostar}}{\density(X)}{\energy}$. We conclude that $\oplusu{y\in\Omega_A}(\phi_A(\cdot,y)\otimes c(y))$ is equivalent to $b'\in\es{\cL_A^{\ostar}}{\density(X)}{\energy}$. Now (ii) is verified.

\smallskip

(iii) It follows from Proposition~\ref{p:Mane_potential_properties}~(v) and Lemma~\ref{l:mane potential gives eigenmeasures} that  for all $x\in\Omega_A$, we have $\phi_A(x,\cdot)\in\es{\cL_A}{C(X,\Rminus)}{\energy}$ and $\phi_A(\cdot,x)\in\es{\cL_A^{\ostar}}{\density(X)}{\energy}$. It follows from Proposition~\ref{p:Mane_potential_properties}~(iv) that $\phi_A(x,x)=0\in\R$ for all $x\in\Omega_A$. We conclude that for all $x\in\Omega_A$, $\phi_A(x,\cdot)\neq\taif$ and $\phi_A(\cdot,x)\not\in\{\tinftyf,\,\taif\}$. Now the first part of (iii) is verified.

For the remaining part of (iii), by Proposition~\ref{p: equivalence for uniqueness of Mane (x,_) and (_,y)}, it suffices to show that if $A$ has a unique maximizing measure, then $\phi_A(x,y)\otimes\phi_A(y,x)=0$ for all $x,\, y\in\Omega_A$. 

Note that it follows from Proposition~\ref{p:Mane_potential_properties}~(iii)(iv) that for all $x,\, y\in X$,
\begin{equation}\label{ineq: mane potential x-y-x smaller than zero}
\phi_A(x,y)\otimes\phi_A(y,x)\leq\phi_A(x,x)\leq 0.
\end{equation}
It follows from (\ref{eq: Mane potential (_,y) almost eigendensity}) in Lemma~\ref{l:mane potential gives eigenmeasures} that $\phi_A(T^n(x),x)=\phi_A\bigl(T^{n-1}(x),x\bigr)-\toverline{A}\bigl(T^{n-1}(x)\bigr)$ for all $n\in\N$ and $x\in \Omega_A$. Since $\phi_A(x,x)=0$ for $x\in\Omega_A$ (cf.~Proposition~\ref{p:Mane_potential_properties}~(iv)), $\phi_A(T^n(x),x)=\phi_A(x,x)-S_n\toverline{A}(x)=-S_n\toverline{A}(x)$ for all $n\in\N$ and $x\in\Omega_A$. By Definition~\ref{d:Mane potential}, we see that $\phi_A(x,T^n(x))\geq S_n\toverline{A}(x)$ for all $n\in\N$ and $x\in X$. Thus, $\phi_A(x,T^n(x))\otimes\phi_A(T^n(x),x)\geq S_n\toverline{A}(x)\otimes(-S_n\toverline{A}(x))=0$ for all $n\in\N$ and $x\in\Omega_A$. It follows from this inequality and (\ref{ineq: mane potential x-y-x smaller than zero}) that $\phi_A(x,T^n(x))\otimes\phi_A(T^n(x),x)=0$ for all $n\in\N$ and $x\in\Omega_A$. 

Denote $L_x\=\overline{\{T^n(x):n\in\N\}}$ for all $x\in\Omega_A$. It then follows from the upper semi-continuity of $\phi_A$ (see Lemma~\ref{l: mane potential is upper semi-continuous}) and (\ref{ineq: mane potential x-y-x smaller than zero}) that $\phi_A(x,z)\otimes\phi_A(z,x)=0$ for all $z\in L_x$. Note that $\phi_A(x,y)\otimes\phi_A(y,x)=0$ defines an equivalence relation between points in $\Omega_A$ by Proposition~\ref{p:Mane_potential_properties}~(iii)(iv). 

We conclude that if $\phi_A(x,y)\otimes\phi_A(y,x)\neq 0$ for some $x,\, y\in\Omega_A$, then $L_x\cap L_y=\emptyset$. Recall from Lemma~\ref{l: maximizing measures supported on Aubry set} that an invariant probability measure is a maximizing measure for $A$ if and only if it is supported on $\Omega_A$. Since $L_x$ and $L_y$ are both compact invariant subsets of $\Omega_A$, we see that there are at least two maximizing measures for $A$ supported respectively on $L_x$ and $L_y$, which contradicts the assumption that $A$ has a unique maximizing measure. 
\end{proof}

\subsection{Uniqueness of eigenfunction and eigendensity}\label{ss: sufficient condition for uniqueness}

\begin{prop} [Sufficient condition for uniqueness]\label{p:condition_for_uniqueness}
Let $T\:X \rightarrow  X$ be an open continuous distance-expanding map on a compact metric space $(X,d)$, and $A \in \aholder$ with $\alpha \in (0,1]$. Assume that $A$ is uniquely maximizing. Then 
\begin{equation*} 
   \tdim\es{\cL_A}{C(X,\Rminus)}{\energy}
   =1
   = \tdim\es{\cL_A^{\ostar}}{\density(X)}{\energy}.
\end{equation*}
\end{prop}
The part on eigenfunctions of $\cL_A$ in $C(X,\R)$ for transitive systems is known (cf.~\cite[Lemma~C]{Bous00}). Here we generalize it to eigenfunctions in $C(X,\Rminus)$ and establish the part on eigendensities using the relation $0=\phi_A(x,y)\otimes\phi_A(y,x)$.

\begin{proof}
Since $A$ is uniquely maximizing, by Theorem~\ref{t:representation}~(iii), Proposition~\ref{p: equivalence for uniqueness of Mane (x,_) and (_,y)}, and Lemma~\ref{l: the linear constant between Mane(x,_) or Mane(_,y)}, we conclude that the entries of $\{\phi_A(x,\cdot)\}_{x\in\Omega_A}$ (resp.~$\{\phi_A(\cdot,x)\}_{x\in\Omega_A}$) are the same tropical eigenfunction of $\cL_A$ in $C(X,\Rminus)$ (resp.~tropical eigendensity of $\cL_{A}^{\ostar}$ in $\density(X)$) up to a tropical multiplicative constant, and
\begin{equation}\label{eq: in_proof_the linear constant}
	\phi_A(\cdot,x)=\phi_A(\cdot,y)\otimes\phi_A(y,x)\quad\text{ and }\quad\phi_A(x,\cdot)=\phi_A(x,y)\otimes\phi_A(y,\cdot)
\end{equation}for all $x,\, y\in\Omega_A$. Fix $x_0\in \Omega_A$ as $\Omega_A\neq\emptyset$. We establish the two equalities in Proposition~\ref{p:condition_for_uniqueness} below.
 
Fix an arbitrary $u\in\es{\cL_A}{C(X,\Rminus)}{\energy}$. It follows from Theorem~\ref{t:representation}~(i) and (\ref{eq: in_proof_the linear constant}) that 
\begin{equation}
\begin{aligned}
    u(\cdot)=\oplusu{x\in\Omega_A}(u(x)\otimes\phi_A(x,\cdot))&=\oplusu{x\in\Omega_A}(u(x)\otimes\phi_A(x,x_0)\otimes\phi_A(x_0,\cdot))\\
    &=\Bigl(\oplusu{x\in\Omega_A}(u(x)\otimes\phi_A(x,x_0))\Bigr)\otimes\phi_A(x_0,\cdot).
\end{aligned}
\end{equation}Denote $d\=\oplusu{x\in\Omega_A}(u(x)\otimes\phi_A(x,x_0))\in\Rm$. It follows that $u(\cdot)=d\otimes\phi_A(x_0,\cdot)$. We conclude that $\es{\cL_A}{C(X,\Rminus)}{\energy}\subseteq\bigl\{a\otimes\phi_A(x_0,\cdot):a\in\Rm\bigr\}$. Since $\phi_A(x_0,x_0)=0$ and $\phi_A(x_0,\cdot)\in C(X,\Rminus)$ (cf.~Proposition~\ref{p:Mane_potential_properties}~(iv)(v)), it is straightforward to check that $a\otimes\phi_A(x_0,\cdot)\in\es{\cL_A}{C(X,\Rminus)}{\energy}$ for all $a\in\Rminus$ and the first equality is now verified.

For the second equality, fix an arbitrary $b\in\es{\cL_A^{\ostar}}{\density(X)}{\energy}\smallsetminus\{\tinftyf\}$. It follows from Theorem~\ref{t:representation}~(ii)~(a) and (\ref{eq: in_proof_the linear constant}) that for all $x\in X$,
\begin{equation}\label{eq: in_proof_uniqueness of eigendensity}
\begin{aligned}
b(x)
=\oplusu{y\in\Omega_A}(\phi_A(x,y)\otimes b(y))
&=\oplusu{y\in\Omega_A}(\phi_A(x,x_0)\otimes\phi_A(x_0,y)\otimes b(y))\\
&=\Bigl(\oplusu{y\in\Omega_A}(\phi_A(x_0,y)\otimes b(y))\Bigr)\otimes\phi_A(x,x_0).
\end{aligned}
\end{equation}
Since $\phi_A$ has an upper bound (cf.~Remark~\ref{r: regularity of Mane potential}) and $b\in\density(X)\smallsetminus\{\tinftyf\}$, it follows that $c\=\oplusu{y\in\Omega_A}(\phi_A(x_0,y)\otimes b(y))\in\Rminus$. We conclude that $b(\cdot)=c\otimes \phi_A(\cdot,x_0)$ and $\es{\cL_A^{\ostar}}{\density(X)}{\energy}\smallsetminus\{\tinftyf\}\subseteq\{a\otimes\phi_A(\cdot,x_0):a\in\Rminus\}$. Since $\phi_A(x_0,x_0)=0$ (cf.~Proposition~\ref{p:Mane_potential_properties}~(iv)) and $\phi_A(\cdot,x_0)\in\density(X)$ (cf.~Lemma~\ref{l:mane potential gives eigenmeasures}), it is straightforward to check that $\{a\otimes\phi_A(\cdot,x_0):a\in\Rminus\}\subseteq\es{\cL_A^{\ostar}}{\density(X)}{\energy}$, and the second equality is now verified.
\end{proof}

\subsection{Proof of Theorem~\ref{t:existence and generic uniqueness}}\label{ss: proof of C}
\begin{proof}[\bf Proof of Theorem~\ref{t:existence and generic uniqueness}]

For (i), it follows from Corollary~\ref{c:more_of_constructions}~(ii) that $v_u$ is a tropical eigenfunction of $\cL_A$ in $C(X,\Rminus)$ associated with eigenvalue $\energy$. For the rest of (i), it follows from Proposition~\ref{p:uniqueness of tropical eigenvalue}~(i) and (\ref{eq: defi of v_u}) that if $u\in C(X,\R)$ is a tropical eigenfunction of $\cL_A$, then $\cL_A(u)=\energy\otimes u$ and $v_u=u$. Since $v_{\tmif}-\norm{u}_{C^0}\leq v_u\leq v_{\tmif}+\norm{u}_{C^0}$ (cf.~Corollary~\ref{c:more_of_constructions}~(i)), it follows that $v_u^{-1}(-\infty)=v_{\tmif}^{-1}(-\infty)$ for all $u\in C(X,\R)$. We conclude that if $v_{\tmif}^{-1}(-\infty)\neq\emptyset$, then $\cL_A$ has no tropical eigenfunction in $C(X,\R)$. If $v_{\tmif}\in C(X,\R)$, then by Corollary~\ref{c:more_of_constructions}~(i), for all $u\in C(X,\R)$, $v_u\in\aholder$. Finally, if $T$ is transitive, then it follows from Proposition~\ref{p:main_construction}~(iii) that $v_{\tmif}\in\aholder$.

For (iii), it directly follows from Proposition~\ref{p:existence of eigenmeasure}. 

For (ii) and (iv), we recall that it follows from \cite[Theorem~3.2]{Je06} that for a generic set of potentials $A$ in $\aholder$, $A$ has a unique maximizing measure. Thus, (ii) and (iv) follow from Proposition~\ref{p:condition_for_uniqueness}.

For (v), it follows from (i) and (iii) that $\energy$ is a tropical eigenvalue of $\cL_A$ (resp.~$\cL_A^{\ostar}$) on $C(X,\Rminus)$ (resp.~$\density(X)$). The rest of (v) follows from Proposition~{\ref{p:uniqueness of tropical eigenvalue}}~(ii)(iii).
\end{proof}

\section{Zero-temperature limits}\label{s: zero-temperature limits}

The following two kinds of zero-temperature limits have been investigated in the literature (see e.g.~\cite{BLL13} and \cite[Section~4]{Je19}). One is to study the weak$^*$ limits of the equilibrium states $\{\mu_{\beta A}\}_{\beta\in(1,+\infty)}$ as the inverse temperature $\beta\to+\infty$. The other is to study the accumulation points (in $C^0$ topology) of $\bigl\{\frac{1}{\beta}\log u_{\beta A}\bigr\}_{\beta\in(1,+\infty)}$ as $\beta\to+\infty$ and the accumulation points are generally tropical eigenfunctions of $\cL_{A}$ (associated with eigenvalue $\energy$). It is also natural to consider the logarithmic-type zero-temperature limits of the equilibrium states $\{\mu_{\beta A}\}_{\beta\in(1,+\infty)}$. 

In this section, we establish Theorems~\ref{t:zero-temperature limit} and~\ref{t:Log type limit for equilibrium states and altered potentials} in Subsection~\ref{ss: proof of E and F}, and consequently Theorems~\ref{t: uniquely maximizing implies the large deviation principle} and \ref{t: equi conditions} in Subsection~\ref{ss: proof of A and B}. 

Assume that $T$ is a transitive expanding covering map and $\alpha\in(0,1]$. For all $\beta>0$ and $A\in\aholder$, recall that $m_{\beta A}$ is the unique Borel probability measure satisfying $\cR_{\beta A}^*(m_{\beta A})=e^{P(T,\beta A)} m_{\beta A}$, and that $u_{\beta A}$ is the unique eigenfunction of $\cR_{\beta A}$ associated with eigenvalue $e^{P(T,\beta A)}$ satisfying $\int u_{\beta A}\,\mathrm{d}m_{\beta A}=1$. Note that $u_{\beta A}$ is strictly positive. We need the following well-known result concerning $\bigl\{\frac{1}{\beta}\log u_{\beta A}\bigr\}_{\beta\in(1,+\infty)}$ (see e.g.~\cite[Theorem~1]{Sa99}). 

\begin{lemma}\label{l: zero tem limits of eigenfunctions}
    Let $T\:X\rightarrow X$ be a transitive expanding covering map on a compact metric space $(X,d)$, and $A\:X\to\R$ be $\alpha$-H\"{o}lder continuous with $\alpha \in (0,1]$. Then the family $\bigl\{\frac{1}{\beta}\log u_{\beta A}\bigr\}_{\beta\in(1,+\infty)}$ is normal and the (uniform) limit of every convergent subsequence $\bigl\{\frac{1}{\beta_n}\log u_{\beta_n A}\bigr\}_{n\in\N}$ with $\beta_n\to+\infty$ as $n\to+\infty$ is in $\es{\cL_A}{C(X,\R)}{\energy}$.
\end{lemma}

Note that by \cite[Theorems~4.3.1 and~4.4.2]{DZ09}, when $X$ is a compact metric space, a family of probability measures $\{\nu_\beta\}_{\beta\in(r,+\infty)}$ satisfies the large deviation principle as $\beta\to+\infty$ with rate function $I$ (cf.\ Subsection~\ref{ss:LDP}) if and only if for each $f\in C(X,\R)$,
	\begin{equation}\label{eq:LDP in functionals}
		\lim_{\beta\to+\infty}\frac{1}{\beta}\log\int\!e^{\beta f}\,\mathrm{d}\nu_{\beta}=\sup_{x\in X}(f(x)-I(x)) .
	\end{equation}

Recall that for all $f\in C(X,\R)\cup\{\tinftyf,\,\taif\}$ and $\beta>0$,
\begin{equation*} 
	l_{\beta}^{\mu}(f)=\frac{1}{\beta}\log\int \! e^{\beta f} \, \mathrm{d} \mu_{\beta A}\quad\text{ and }\quad
 	l_{\beta}^{m}(f)=\frac{1}{\beta}\log\int \! e^{\beta f} \, \mathrm{d} m_{\beta A}.
 \end{equation*}

\begin{rem}\label{r: l_beta at infinity}
Note that for all $\beta>0$,
\begin{equation}\label{eq: l_beta at infinity}
l_{\beta}^m(\tinftyf)=l_{\beta}^\mu(\tinftyf)=+\infty\quad\text{ and }\quad l_{\beta}^m(\taif)=l_{\beta}^\mu(\taif)=\tai=-\infty.
\end{equation}
Thus, if $\bigl\{l_{\beta_k}^m|_{C(X,\R)}\bigr\}_{k\in\N}$ $\bigl($resp.~$\bigl\{l_{\beta_k}^\mu|_{C(X,\R)}\bigr\}_{k\in\N}$$\bigr)$ is pointwise convergent as $k\to+\infty$, then $\bigl\{l_{\beta_k}^m\bigr\}_{k\in\N}$ $\bigl($resp.~$\bigl\{l_{\beta_k}^\mu\bigr\}_{k\in\N}$$\bigr)$ is also pointwise convergent on $C(X,\R)\cup\{\tinftyf,\,\taif\}$ as $k\to+\infty$.
\end{rem}
It is straightforward to check that for all $f,\, g\in C(X,\R)$ and $\beta>0$,
\begin{align}
	\Absbig{l_{\beta}^{m}(f)} &\leq \norm{f}_{C^0},\quad
	\Absbig{l_{\beta}^m(f)-l_{\beta}^m(g)}\leq\norm{f-g}_{C^0}, \label{ineq:regularity of eigenmeasure_functionals at positive temperature} \\
	\Absbig{l_{\beta}^{\mu}(f)} & \leq\norm{f}_{C^0},\quad
	\Absbig{l_{\beta}^{\mu}(f)-l_{\beta}^{\mu}(g)}\leq\norm{f-g}_{C^0}.  \label{ineq:regularity of equilibrium_functionals at positive temperature}
\end{align}

Assuming that $X$ is a compact metric space, we see that $C(X,\R)$ is separable. Thus, by the Arzel\`a--Ascoli theorem (cf.~\cite[Theorem~X.2.2]{Bour07}), an equicontinuous family of real-valued continuous functionals on $C(X,\R)$ that is uniformly bounded on every compact subset of $C(X,\R)$ is a normal family. 

Thus, we only need to verify the equicontinuity and the uniform boundedness on compact subsets of $C(X,\R)$ when showing the normality of a given family below. Recall that $\tmif = \mathbbold{0}$ and $\mathbbold{1}$ are used to represent the constant $0$ and $1$ functions on $X$.

\subsection{Proofs of Theorems~\ref{t:zero-temperature limit} and \ref{t:Log type limit for equilibrium states and altered potentials}}\label{ss: proof of E and F}
\begin{proof}[\bf Proof of Theorem~\ref{t:zero-temperature limit}]
\smallskip
(i) It follows from (\ref{ineq:regularity of eigenmeasure_functionals at positive temperature}) that $\bigl\{l_{\beta}^m|_{C(X,\R)}\bigr\}_{\beta\in(1,+\infty)}$ is uniformly bounded on every bounded subset of $C(X,\R)$, equicontinuous, and consequently normal. By (\ref{eq: l_beta at infinity}), if $\bigl\{l_{\beta_k}^m|_{C(X,\R)}\bigr\}_{k\in\N}$ is pointwise convergent as $k\to+\infty$, then so is $\bigl\{l_{\beta_k}^m\bigr\}_{k\in\N}$.

Now suppose $l\:C(X,\R)\cup\{\tinftyf,\,\taif\}\to\Rm$ is the pointwise limit of a convergent subsequence $\bigl\{l_{\beta_k}^m\bigr\}_{k\in\N}$ with $\beta_k\to+\infty$ as $k\to+\infty$.

\smallskip
\emph{Claim.} The function $l\: C(X,\R)\cup\{\tinftyf,\,\taif\}\to\Rm$ is a tropical linear functional.
\smallskip

 It follows from the definition of $l_{\beta}^{m}$ that $l_{\beta}^m(a\otimes f)=a\otimes l_{\beta}^m(f)$ for all $a\in\R$, $f\in C(X,\R)$, and $\beta>0$. Thus, $l(a\otimes f)=a\otimes l(f)$ for all $a\in\R$ and $f\in C(X,\R)$. It follow from (\ref{eq: l_beta at infinity}) that $l(\taif)=\tai$ and $l(\tinftyf)=+\infty$. It is straightforward to check that $l(\tai\otimes f)=\tai\otimes l(f)$, $l((+\infty)\otimes f)=(+\infty)\otimes l(f)$, $l(\taif\oplus f)=l(\taif)\oplus l(f)$, and $l(\tinftyf\oplus f)=l(\tinftyf)\oplus l(f)$ for all $f\in C(X,\R)\cup\{\tinftyf,\,\taif\}$.

Now it suffices to prove $l(f\oplus g)=l(f)\oplus l(g)$ for all $f,\, g\in C(X,\R)$. We introduce the \emph{plus operation at inverse temperature $\beta>0$}: 
\begin{equation*} 
	h_1\oplusu{\beta}h_2\= \frac{1}{\beta}\log\bigl(e^{\beta h_1}+e^{\beta h_2}\bigr),
\end{equation*}
for $h_1,h_2\in C(X,\R)$.

 It immediately follows that for all $n\in\N$, $h_i\in C(X,\R)$ with $i = 1,\, \dots,\,  n$, and $\beta >0$,
 \begin{equation}\label{ineq: difference of several beta sum}
     \oplusu{1\leq i\leq n}h_i\leq h_1\oplusu{\beta}h_2\oplusu{\beta}\cdots\oplusu{\beta}h_n\leq\frac{\log n}{\beta}\otimes \Bigl(\oplusu{1\leq i\leq n}h_i\Bigr) ,
 \end{equation}
 and for all $h_1,h_2\in C(X,\R)$, $l_{\beta}^m\bigl(h_1\oplusu{\beta}h_2\bigr)=l_{\beta}^m(h_1)\oplusu{\beta}l_{\beta}^m(h_2)$. Thus, for all $k\in\N$ and $f,g\in C(X,\R)$,
\begin{align*}
l_{\beta_k}^m(f\oplus g)
&\leq l_{\beta_k}^m\Bigl(f\oplusu{\beta_k}g\Bigr)\leq l_{\beta_k}^m(f\oplus g)\otimes\frac{\log 2}{\beta_k} ,\\
l_{\beta_k}^m\Bigl(f\oplusu{\beta_k}g\Bigr)
&=l_{\beta_k}^m(f)\oplusu{\beta_k}l_{\beta_k}^m(g) ,\\
l_{\beta_k}^m(f)\oplus l_{\beta_k}^m(g)
&\leq l_{\beta_k}^m(f)\oplusu{\beta_k}l_{\beta_k}^m(g)\leq \bigl(l_{\beta_k}^m(f)\oplus l_{\beta_k}^m(g)\bigr)\otimes\frac{\log 2}{\beta_k}.
\end{align*}
We conclude that $l_{\beta_k}^m(f\oplus g)-\frac{\log 2}{\beta_k}
\leq l_{\beta_k}^m(f)\oplus l_{\beta_k}^m(g)\leq l_{\beta_k}^m(f\oplus g)+\frac{\log 2}{\beta_k}$.
Recall that $\lim_{k\to+\infty}\beta_k=+\infty$ and that $l_{\beta_k}^m$ pointwise converges to $l$ as $k\to+\infty$. It follows that $l(f\oplus g)=l(f)\oplus l(g)$ for all $f,\, g\in C(X,\R)$. Now the claim is verified.
\smallskip

It follows from Propositions~\ref{p:compactness implies continuity} and \ref{p:dual space equals completion} that $l$ is tropically continuous and has a unique density $b$ in $\density(X)$. Now we show that $b$ is in $\es{\cL_A^{\ostar}}{\density(X)}{\energy}$. Since $\energy\in\R$ (cf.~(\ref{eq: defi of energy})), it follows that $\energy\otimes b\in\density(X)$. Note that $\cL_{A}^{\ostar}(b)\in\density(X)$ (cf.~(\ref{eq:definition of tropical adjoint})), and that for every $b_1\:X\to\Rm$, there exists a unique $b_2\in\density(X)$ equivalent to $b_1$ (cf.~Remark~\ref{r:density and functional}). Thus, it suffices to show that $\cL_{A}^{\ostar}(b)$ is equivalent to $\energy\otimes b$. By (\ref{eq: reason for adjoint}), it suffices to prove that for all $f\in C(X,\R)$,
\begin{equation*}
l(\cL_A(f))=l(f)\otimes\energy.
\end{equation*}

Fix $f\in C(X,\R)$. It follows from (\ref{ineq:regularity of eigenmeasure_functionals at positive temperature}), (\ref{ineq: difference of several beta sum}), and the definitions of the two operators ((\ref{eq: defi of the Ruelle operator}) and (\ref{eq: defi of the Bousch operator})) that for all $k\in\N$,
\begin{equation}\label{ineq:difference between Ruelle and Bousch operator}
 \begin{aligned}
 	 \Absbig{l_{\beta_k}^m(\cL_{A}(f))-l_{\beta_k}^m\bigl(\beta_k^{-1}\log \cR_{\beta_k A} \bigl( e^{\beta_k f}\bigr)\bigr)}
 	& \leq\Normbig{\cL_A(f)-\beta_k^{-1}\log\cR_{\beta_k A}\bigl(e^{\beta_k f}\bigr)}_{C^0} \\
 	& \leq\frac{\log N}{\beta_k},
 \end{aligned}
\end{equation}
where $N$ is the constant in Lemma~\ref{l:expanding systems}.

Note that
\begin{equation}\label{eq:eigenmeasure_functional at positive temperature}
\begin{aligned}
l_{\beta_k}^m\bigl(\beta_k^{-1}\log\cR_{\beta_k A}\bigl(e^{\beta_k f} \bigr)\bigr)
&=\frac{1}{\beta_k}\log\int \! \cR_{\beta_k A} \bigl(e^{\beta_k f} \bigr) \, \mathrm{d} m_{\beta_k A}\\
&=\frac{1}{\beta_k}\log\int \! e^{P(T,\beta_k A)}\cdot e^{\beta_k f} \, \mathrm{d} m_{\beta_k A}\\
&=\frac{P(T,\beta_k A)}{\beta_k}\otimes\frac{1}{\beta_k}\log\int \! e^{\beta_k f} \, \mathrm{d} m_{\beta_k A}\\
&=\frac{P(T,\beta_k A)}{\beta_k}\otimes l_{\beta_k}^m(f),
\end{aligned}
\end{equation}
where the second equality holds since $m_{\beta A}$ is the eigenmeasure of  $\cR_{\beta A}^*$ associated with eigenvalue $e^{P(T, \beta A)}$. Recall $\lim\limits_{k\to+\infty}\beta_k=+\infty$,
$\lim\limits_{k\to+\infty}l_{\beta_k}^m(f)=l(f)$, $\lim\limits_{k\to+\infty}l_{\beta_k}^m(\cL_A(f))=l(\cL_A(f))$ (here $\cL_A(f)\in C(X,\R)$ by Proposition~\ref{p:L_continuous}), and $\lim\limits_{\beta\to+\infty}\beta^{-1}P(T,\beta A)=\energy$ (cf.~\cite[Proposition~2.11]{BLL13}). Combining (\ref{ineq:difference between Ruelle and Bousch operator}) and (\ref{eq:eigenmeasure_functional at positive temperature}) and letting $k\to+\infty$, we conclude that $l(\cL_A(f))=l(f)\otimes\energy$, and (i) is now verified.

\smallskip
(ii) It follows from (\ref{ineq:regularity of equilibrium_functionals at positive temperature}) that $\bigl\{ l_{\beta}^\mu|_{C(X,\R)} \bigr\}_{\beta\in(1,+\infty)}$ is normal. By (\ref{eq: l_beta at infinity}), if $\bigl\{l_{\beta_k}^\mu|_{C(X,\R)}\bigr\}_{k\in\N}$ is pointwise convergent as $k\to+\infty$, then so is $\bigl\{l_{\beta_k}^\mu\bigr\}_{k\in\N}$.

Now suppose $\widehat{l}$ is the pointwise limit of a convergent subsequence $\bigl\{l_{\beta_k}^\mu\bigr\}_{k\in\N}$ with $\beta_k\to+\infty$ as $k\to+\infty$. By Lemma~\ref{l: zero tem limits of eigenfunctions}, (i), and Remark~\ref{r: l_beta at infinity}, we take a subsequence $\{\beta'_k\}_{k\in\N}$ from $\{\beta_k\}_{k\in\N}$ with $\beta'_k\to+\infty$ as $k\to+\infty$ such that as $k\to+\infty$, $\bigl\{{\beta'_k}^{-1}\log u_{\beta'_k A}\bigr\}_{k\in\N}$ uniformly converges to $v$ in $\es{\cL_A}{C(X,\R)}{\energy}$ and $\bigl\{l_{\beta'_k}^{m}\bigr\}_{k\in\N}$ pointwise converges to $l$. 

Note that for all $\beta>0$ and $f\in C(X,\R)$, $l_{\beta}^{\mu}(f)=l_{\beta}^m \bigl( f+\beta^{-1}\log u_{\beta A} \bigr)$ since $\mu_{\beta A}=u_{\beta A}\cdot m_{\beta A}$. Thus, for all $k\in\N$ and $f\in C(X,\R)$,
\begin{equation*}
	\begin{aligned}
		\Absbig{l(v\otimes f)-l_{\beta'_k}^{\mu}(f)}
		&=\Absbig{l(v+f)-l_{\beta'_k}^m \bigl( f+{\beta'_k}^{-1}\log u_{\beta'_k A} \bigr) }\\
		&\leq\Absbig{l(v+f)-l_{\beta'_k}^m(v+f)}+\Absbig{l_{\beta'_k}^m(v+f)-l_{\beta'_k}^m \bigl( f+{\beta'_k}^{-1} \log u_{\beta'_k A} \bigr) }\\
		&\leq\Absbig{l(v+f)-l_{\beta'_k}^m(v+f)}+\Normbig{v-{\beta'_k}^{-1}\log u_{\beta'_k A}}_{C^0},
	\end{aligned}
\end{equation*} where the last inequality follows from (\ref{ineq:regularity of eigenmeasure_functionals at positive temperature}). By the choice of $\{\beta'_k\}_{k\in\N}$, it follows that $l(v\otimes f)=\widehat{l}(f)$ for all $f\in C(X,\R)$. Since $v\in C(X,\R)$, it follows from (\ref{eq: l_beta at infinity}) that $l(v\otimes f)=\widehat{l}(f)$ for all $f\in C(X,\R)\cup\{\tinftyf,\,\taif\}$. 

By (i), $l$ is tropically continuous linear and its density $b$ in $\density(X)$ is in $\es{\cL_A^{\ostar}}{\density(X)}{\energy}$. It follows that $\widehat{l}(f)=\oplusu{x\in X}(f(x)\otimes v(x)\otimes b(x))$ for all $f\in C(X,\R )\cup\{\tinftyf,\,\taif\}$, i.e., $\widehat{l}$ is tropically continuous linear and $v\otimes b$ is its density in $\density(X)$. Now (ii) is verified.
\end{proof}

Recall that $u_A$ is the unique eigenfunction of the Ruelle operator $\cR_A$ satisfying $\int \! u_A \, \mathrm{d} m_A=1$ associated with eigenvalue $e^{P(T,A)}$. Note that for each $\beta>0$, $\cR_{\widetilde{\beta A}}$ with
\begin{equation*}
	\widetilde{\beta A}=\beta A+\log u_{\beta A}-\log u_{\beta A}\circ T-P(T,\beta A)
\end{equation*}
satisfies $\cR_{\widetilde{\beta A}}(\mathbbold{1})=\mathbbold{1}$ and $\cR_{\widetilde{\beta A}}^*(\mu_{\beta A})=\mu_{\beta A}$ (see e.g.~\cite[Section~5.4]{PU10}). So considering the logarithmic-type zero-temperature limit, we predict that if $\twidetilde{A}$ is the limit of $\widetilde{\beta A}\big/\beta$ and $\twidetilde{b}$ is the density in $\density(X)$ of the limit of $l_{\beta}^\mu$, then $\cL_{\twidetilde{A}}(\tmif)=\tmif$ and $\cL_{\twidetilde{A}}^{\ostar}\bigl(\twidetilde{b}\bigr)=\twidetilde{b}$. 

\begin{proof}[\bf Proof of Theorem~\ref{t:Log type limit for equilibrium states and altered potentials}]
(i) Since $\widetilde{\beta A}=\beta A+\log u_{\beta A}-\log u_{\beta A}\circ T-P(T,\beta A)$, $\energy=\lim_{\beta\to +\infty}\frac{P(T,\beta A)}{\beta}$ (cf.~\cite[Proposition~2.11]{BLL13}), and $\bigl\{\frac{1}{\beta}\log u_{\beta A}\bigr\}_{\beta\in(1,+\infty)}$ is a normal family (cf.~Lemma~\ref{l: zero tem limits of eigenfunctions}), it immediately follows that $\bigl\{\widetilde{\beta A}\big /\beta\bigr\}_{\beta\in(1,+\infty)}$ is a normal family. It has been verified in Theorem~\ref{t:zero-temperature limit}~(ii) that $\bigl\{l_{\beta}^{\mu}|_{C(X,\R)}\bigr\}_{\beta\in(1,+\infty)}$ is normal.

\smallskip
(ii) Recall $\widetilde{\beta A}=\beta A+\log u_{\beta A}-\log u_{\beta A}\circ T-P(T,\beta A)$ and $\lim_{\beta\to+\infty}\beta^{-1}P(T,\beta A)=\energy$ (cf.~\cite[Proposition~2.11]{BLL13}). By the assumption that $\beta_k^{-1}\log u_{\beta_kA}$ pointwise converges to $v$ as $k\to+\infty$, we see that $\widetilde{\beta_k A}\big/\beta_k$ pointwise converges to $A+v-v\circ T-\energy$ as $k\to+\infty$.

\smallskip
(iii) By Theorem~\ref{t:zero-temperature limit}~(ii), $l$ is a tropically continuous linear functional. Let $b$ be the density of $l$ in $\density(X)$ (cf.~Proposition~\ref{p:compactness implies continuity} and Remark~\ref{r:density and functional}). Now we show that $\cL_{\twidetilde{A}}(b)=b$. Since $\twidetilde{A}\in C(X,\R)$, it follows from (\ref{eq:definition of tropical adjoint}) that $\cL_{\twidetilde{A}}(b)\in\density(X)$. Recall that for every $b_1\:X\to\Rm$, there exists a unique $b_2\in\density(X)$ equivalent to $b_1$ (cf.~Remark~\ref{r:density and functional}). It suffices to show that $\cL_{\twidetilde{A}}(b)$ is equivalent to $b$. Thus, by (\ref{eq: reason for adjoint}), it suffices to show that 
$l\bigl(\cL_{\twidetilde{A}}(f)\bigr)=l(f)$ for all $f\in C(X,\R)$.

Fix $f\in C(X,\R)$. Recall $\cR_{\widetilde{\beta A}}^*(\mu_{\beta A})=\mu_{\beta A}$. It follows that for all $\beta>0$,
	\begin{equation} \label{eq:equilibrium is eigenmeasure of altered Ruelle}
		\begin{aligned}
		        l_{\beta}^{\mu}\biggl(\frac{1}{\beta}\log\cR_{\widetilde{\beta A}}\bigl(e^{\beta f}\bigr)\biggr)
                &=\frac{1}{\beta}\log\int\!\cR_{\widetilde{\beta A}}\bigl(e^{\beta f}\bigr)\,\mathrm{d}\mu_{\beta A}\\
                &=\frac{1}{\beta}\log\int\! e^{\beta f}\,\mathrm{d}\cR_{\widetilde{\beta A}}^{*}(\mu_{\beta A})
                =l_{\beta}^{\mu}(f).
		\end{aligned}
	\end{equation}
	Now we compare $\cL_{\twidetilde{A}}(f)$ with $\frac{1}{\beta}\log\cR_{\widetilde{\beta A}}\bigl(e^{\beta f}\bigr)$. Similar to (\ref{ineq:difference between Ruelle and Bousch operator}), it  follows from the definitions of the Ruelle operator and the Bousch operator (cf.\ (\ref{eq: defi of the Ruelle operator}) and (\ref{eq: defi of the Bousch operator})) that for all $\beta>0$, 
	\begin{equation*}
		\begin{aligned}
			\cL_{\twidetilde{A}}(f) &\fless\beta^{-1}\log \cR_{\beta\twidetilde{A}}\bigl(e^{\beta f}\bigr)\fless\cL_{\twidetilde{A}}(f)+\beta^{-1}\log N,\\
			\Normbig{\twidetilde{A}-\widetilde{\beta A}\big/\beta}_{C^0}&\geq\Normbig{\beta^{-1}\log \cR_{\beta\twidetilde{A}}\bigl(e^{\beta f}\bigl)-\beta^{-1}\log \cR_{\widetilde{\beta A}}\bigl(e^{\beta f}\bigr)}_{C^0},
		\end{aligned}
	\end{equation*}
	where $N$ is the constant in Lemma~\ref{l:expanding systems}. We conclude that for all $\beta>0$.
	\begin{equation}\label{ineq:difference of altered Ruelle and altered Bousch}
		\Normbig{\cL_{\twidetilde{A}}(f)-\beta^{-1}\log \cR_{\widetilde{\beta A}}\bigl(e^{\beta f}\bigr)}_{C^0}
		\leq\beta^{-1}\log N+\Normbig{\twidetilde{A}-\widetilde{\beta A}\big/\beta}_{C^0}.
	\end{equation}
	Thus, for all $k\in\N$,
	\begin{equation}\label{ineq:conclusion for invariantness of equilibrium_functional}
		\begin{aligned}
			&\Absbig{l(\cL_{\twidetilde{A}}(f))-l(f)}  \\
			&\qquad\leq \Absbig{l(f)-l_{\beta_k}^{\mu}(f)}+\Absbig{l_{\beta_k}^{\mu}(f)-l_{\beta_k}^{\mu}\bigl(\beta_k^{-1} \log\cR_{\widetilde{\beta_k A}}\bigl(e^{\beta_k f}\bigr)\bigr)}\\
			&\qquad\qquad+\Absbig{l_{\beta_k}^{\mu}\bigl(\beta_k^{-1} \log \cR_{\widetilde{\beta_k A}}\bigl(e^{\beta_k f}\bigr)\bigr)-l_{\beta_k}^{\mu}\bigl(\cL_{\twidetilde{A}}(f)\bigr)}
			+\Absbig{l_{\beta_k}^{\mu}\bigl(\cL_{\twidetilde{A}}(f)\bigr)-l\bigl(\cL_{\twidetilde{A}}(f)\bigr)}\\
			&\qquad\leq \Absbig{l(f)-l_{\beta_k}^{\mu}(f)} + 0 \\
			&\qquad\qquad+ \Normbig{\cL_{\twidetilde{A}}(f) - \beta_k^{-1} \log \cR_{\widetilde{\beta_k A}}\bigl(e^{\beta_k f}\bigr) }_{C^0} 
			+\Absbig{l_{\beta_k}^{\mu}\bigl(\cL_{\twidetilde{A}}(f)\bigr)-l\bigl(\cL_{\twidetilde{A}}(f)\bigr)}\\
			&\qquad\leq \Absbig{l(f)-l_{\beta_k}^{\mu}(f)}+\beta_k^{-1}\log N+\Normbig{\twidetilde{A}-\widetilde{\beta _k A}\big/ \beta_k}_{C^0}+\Absbig{l_{\beta_k}^{\mu}\bigl(\cL_{\twidetilde{A}}(f)\bigr)-l\bigl(\cL_{\twidetilde{A}}(f)\bigr)},
		\end{aligned}
	\end{equation}
	where the second inequality follows from (\ref{eq:equilibrium is eigenmeasure of altered Ruelle}) and (\ref{ineq:regularity of equilibrium_functionals at positive temperature}), and the third inequality follows from (\ref{ineq:difference of altered Ruelle and altered Bousch}).
	
	Recall $\lim_{k\to+\infty}l_{\beta_k}^{\mu}(f)=l(f)$, $\lim_{k\to+\infty}l_{\beta_k}^{\mu}\bigl(\cL_{\twidetilde{A}}(f)\bigr)=l\bigl(\cL_{\twidetilde{A}}(f)\bigr)$ (here $\cL_{\twidetilde{A}}(f)\in C(X,\R)$ by Proposition~\ref{p:L_continuous}), $\lim_{k\to+\infty}\widetilde{\beta_k A}\big/\beta_k=\twidetilde{A}$, and $\lim_{k\to+\infty}\beta_k=+\infty$. As $k\to+\infty$ in (\ref{ineq:conclusion for invariantness of equilibrium_functional}), we conclude that $l(f)=l\bigl(\cL_{\twidetilde{A}}(f)\bigr)$.
\end{proof}

\subsection{Proofs of Theorems~\ref{t: uniquely maximizing implies the large deviation principle} and \ref{t: equi conditions}}\label{ss: proof of A and B}
The framework that we have established now enables us to generalize the main results of \cite{BLT06} and \cite{Me18} to Corollary~\ref{c:uniquely maximizing implies LDP} below and Theorem~\ref{t: equi conditions} beyond symbolic dynamics, and then to achieve a further strengthening in Theorem~\ref{t: uniquely maximizing implies the large deviation principle} by dropping the transitivity assumption of Corollary~\ref{c:uniquely maximizing implies LDP}.

\begin{cor}\label{c:uniquely maximizing implies LDP}
Let $T\:X \rightarrow  X$ be a transitive expanding covering map, and $A \in \aholder$ with $\alpha \in (0,1]$ with a unique maximizing measure. Then the family of equilibrium states $\{\mu_{\beta A}\}_{\beta\in(1,+\infty)}$ satisfies the large deviation principle as $\beta\to+\infty$ with rate function $-(v_0\otimes b_0)$, where $b_0$ is the unique element of $\es{\cL_A^{\ostar}}{\density(X)}{\energy}$ satisfying $\oplusu{x\in X}b_0(x)=0$ and $v_0$ is the unique element of $\es{\cL_A}{C(X,\R)}{\energy}$ satisfying $\oplusu{x\in X}(v_0(x)\otimes b_0(x))=0$.
\end{cor}

\begin{proof}
Note that for all $b\in\density(X)\smallsetminus\{\tinftyf,\,\taif\}$, $\oplusu{x\in X}b(x)\in\R$. Since $A$ is uniquely maximizing, it follows from Proposition~\ref{p:condition_for_uniqueness} that $\tdim\es{\cL_A^{\ostar}}{\density(X)}{\energy}=1$. Thus, there exists a unique $b_0\in\es{\cL_A^{\ostar}}{\density(X)}{\energy}$ such that $\oplusu{x\in X}b_0(x)=0$.

Note that for all $\beta>0$, $l_{\beta}^m(\tmif)= \beta^{-1}\log 1=0$ since $m_{\beta A}$ is a probability measure. It then follows from Theorem~\ref{t:zero-temperature limit}~(i) that every pointwise convergent subsequence of the family $\bigl\{l_{\beta}^m\bigr\}_{\beta\in(1,+\infty)}$ as $\beta\to+\infty$ must converge to some tropical linear functional $l$ whose density $b$ in $\density(X)$ is in $\es{\cL_A^{\ostar}}{\density(X)}{\energy}$ and $\oplusu{x\in X}b(x)=l(\tmif)=0$, which implies that $b=b_0$. Recall from Theorem~\ref{t:zero-temperature limit}~(i) that the family $\bigl\{l_\beta^m|_{C(X,\R)}\bigr\}_{\beta\in(1,+\infty)}$ is normal. By Remark~\ref{r: l_beta at infinity}, we conclude that as $\beta\to+\infty$, $l_\beta^m$ must pointwise converge to $l$ whose density in $\density(X)$ is $b_0$. 

Since $A$ is uniquely maximizing, it follows from Proposition~\ref{p:condition_for_uniqueness} that 
\begin{equation*}
	\tdim\es{\cL_A}{C(X,\Rminus)}{\energy}=1.
\end{equation*}
Thus, there exists a unique $v_0\in\es{\cL_A}{\density(X)}{\energy}$ such that $\oplusu{x\in X}(v_0(x)\otimes b_0(x))=l(v_0)=0$.

\smallskip
\emph{Claim.} As $\beta\to+\infty$, $\beta^{-1}\log u_{\beta A}$ must uniformly converge to $v_0$.
\smallskip

Recall that $\int\! u_{\beta A}\,\mathrm{d}m_{\beta A}=1$ for all $\beta>0$. It follows that for all $\beta>0$,
\begin{equation*}
	l_{\beta}^{m}\bigl(\beta^{-1}\log u_{\beta A}\bigr)
	=\frac{1}{\beta}\log \int\! u_{\beta A}\,\mathrm{d}m_{\beta A}
	=0.
\end{equation*} 
Suppose $v\in\es{\cL_A}{C(X,\R)}{\energy}$ is the uniform limit of a convergent subsequence $\bigl\{ \beta_k^{-1}\log u_{\beta_k A}\bigr\}_{k\in\N}$ with $\beta_k\to+\infty$ as $k\to +\infty$ according to Lemma~\ref{l: zero tem limits of eigenfunctions}. It follows from~(\ref{ineq:regularity of eigenmeasure_functionals at positive temperature}) that for all $k\in\N$,
\begin{equation*}
\Absbig{l_{\beta_k}^m(v)-l_{\beta_k}^m\bigl(\beta_k^{-1}\log u_{\beta_k A}\bigr)}
\leq\Normbig{v-\beta_k^{-1} \log u_{\beta_k A}}_{C^0}.
\end{equation*}
We conclude that for all $k\in\N$,
\begin{equation}\label{ineq:intermidate for linear normalizing condition}
\begin{aligned}
\abs{l(v)}
\leq\Absbig{l(v)-l_{\beta_k}^m(v)}+\Absbig{l_{\beta_k}^m(v)}
&=\Absbig{l(v)-l_{\beta_k}^m(v)}+\Absbig{l_{\beta_k}^m(v)-l_{\beta_k}^m \bigl( \beta_k^{-1}\log u_{\beta_k A} \bigr)}\\
&\leq \Absbig{l(v)-l_{\beta_k}^m(v)}+\Normbig{v-\beta_k^{-1} \log u_{\beta_k A}}_{C^0}.
\end{aligned}
\end{equation}  

Since $\lim_{\beta\to +\infty}l_{\beta}^{m}(v)=l(v)$ and $\lim_{k\to+\infty}\beta_k^{-1}\log u_{\beta_k A}=v$, we get $l(v)=0$ as $k\to+\infty$ in (\ref{ineq:intermidate for linear normalizing condition}). Thus, by $v\in\es{\cL_A}{C(X,\R)}{\energy}$ and $l(v)=0$, we get $v=v_0$. Now our claim follows from the normality of $\bigl\{ \beta^{-1}\log u_{\beta A}\bigr\}_{\beta\in(1,+\infty)}$ (cf.~Lemma~\ref{l: zero tem limits of eigenfunctions}).

\smallskip

Finally, note that for all $\beta>0$ and $f\in C(X,\R)$, $l_{\beta}^{\mu}(f)=l_{\beta}^m \bigl( f+\beta^{-1}\log u_{\beta A} \bigr)$ since $\mu_{\beta A}=u_{\beta A}\cdot m_{\beta A}$. Thus, for all $\beta>0$ and $f\in C(X,\R)$,
\begin{equation}\label{ineq: connection between two kinds of functionals}
\begin{aligned}
	&\Absbig{l(v_0+f)-l_{\beta}^{\mu}(f)}\\
	&\qquad=\Absbig{l(v_0+f)-l_{\beta}^m \bigl( f+\beta^{-1}\log u_{\beta A} \bigr) }\\
	&\qquad\leq\Absbig{l(v_0+f)-l_{\beta}^m(v_0+f)}+\Absbig{l_{\beta}^m(v_0+f)-l_{\beta}^m \bigl( f+\beta^{-1} \log u_{\beta A} \bigr) }\\
	&\qquad\leq\Absbig{l(v_0+f)-l_{\beta}^m(v_0+f)}+\Normbig{v_0-\beta^{-1}\log u_{\beta A}}_{C^0},
\end{aligned}
\end{equation}where the last inequality follows from (\ref{ineq:regularity of eigenmeasure_functionals at positive temperature}).

 Since we have shown that $\frac{1}{\beta}\log u_{\beta A}$ uniformly converges to $v_0$ and that $l_{\beta}^m$ pointwise converges to $l$ as $\beta\to+\infty$, it follows from (\ref{ineq: connection between two kinds of functionals}) that $\lim_{\beta\to+\infty}l_{\beta}^{\mu}(f)=l(v_0+f)$ for all $f\in C(X,\R)$. Since $b_0$ is the density of $l$ in $\density(X)$, $l(v_0+f)=\oplusu{x\in X}(f(x)\otimes(v_0(x)\otimes b_0(x)))$
for all $f\in C(X,\R)$, i.e., $-(v_0\otimes b_0)$ is the rate function $I$ in (\ref{eq:LDP in functionals}). We conclude that $\{\mu_{\beta A}\}_{\beta\in(1,+\infty)}$ satisfies the large deviation principle as $\beta\to+\infty$ with rate function $-(v_0\otimes b_0)$. 
\end{proof}

Theorem~\ref{t: uniquely maximizing implies the large deviation principle} is derived as a consequence of Corollary~\ref{c:uniquely maximizing implies LDP}.
\begin{proof}[\bf Proof of Theorem~\ref{t: uniquely maximizing implies the large deviation principle}]
    Here we need to deal with the lack of the transitivity assumption. Denote the set of nonwandering points by $\Omega(T)$. It is well known that all finite $T$-invariant measures are supported on $\Omega(T)$ (cf.~\cite[Theorem~6.15]{Wal82}). Recall that equilibrium states are $T$-invariant Borel probability measures that maximize the measure-theoretic pressure (cf.~(\ref{eq: measure-theoretic pressure})). Thus, for all $\varphi\in C(X,\R)$, an equilibrium state for $\varphi$ is also an equilibrium state for $\varphi|_{\Omega(T)}$ (and $T|_{\Omega(T)}$). By \cite[Proposition~4.3.8]{PU10}, $\Omega(T)$ can be decomposed into finitely many disjoint compact sets $\Omega_j$, $1\leq j\leq J$, such that $(T|_{\Omega(T)})^{-1}(\Omega_j)=\Omega_j$ and $T|_{\Omega_j}$ is transitive for all $1\leq j\leq J$. 
    
    If for some $\beta>0$, there are at least two equilibrium states for $\beta A$, then by \cite[Proposition~3.6.3]{PU10}, there must be two equilibrium states supported on different $\Omega_j$. It follows from \cite[Theorem~4.1]{Je19} that every weak$^*$ accumulation point of an equilibrium state for $\beta A$ as $\beta\to+\infty$ is a maximizing measure for $A$. Let $\mu_{\max}$ be the unique maximizing measure of $A$. Without loss of generality, we can assume that $\supp\mu_{\max}\subseteq\Omega_1$ (since $\mu_{\max}$ must be ergodic). Consequently, there exists $r>0$ such that for all $\beta>r$, there exists a unique equilibrium state $\mu_{\beta A}$ for $\beta A$ and $\supp\mu_{\beta A}\subseteq\Omega_1$. 

    Now applying Corollary~\ref{c:uniquely maximizing implies LDP} to $T|_{\Omega_1}$ and $A|_{\Omega_1}$, we see that there exists a lower semi-continuous function $I_1\:\Omega_1\to[0,+\infty]$ such that for all $f\in C(X,\R)$,
    \begin{equation}\label{eq: ldp on some transitive part}
        \lim_{\beta\to+\infty}\frac{1}{\beta}\log\int_X e^{\beta f}\,\mathrm{d}\mu_{\beta A}=\lim_{\beta\to+\infty}\frac{1}{\beta}\log\int_{\Omega_1} e^{\beta f}\,\mathrm{d}\mu_{\beta A}=\sup_{x\in\Omega_1}(f(x)-I_1(x)).
    \end{equation}
    Consider 
    \begin{equation*}
        I(x)\=\begin{cases}
            I_1(x) &\text{if }x\in\Omega_1;\\
            +\infty &\text{if }x\in X\smallsetminus\Omega_1.
        \end{cases}
    \end{equation*}It is straightforward to check that $I\:X\to [0,+\infty]$ is lower semi-continuous since $I_1$ is lower semi-continuous and $\Omega_1$ is a compact subset of $X$. It follows from (\ref{eq: ldp on some transitive part}) that for all $f\in C(X,\R)$,
    \begin{equation*}
        \lim_{\beta\to+\infty}\frac{1}{\beta}\log\int_X e^{\beta f}\,\mathrm{d}\mu_{\beta A}=\sup_{x\in X}(f(x)-I(x)).
    \end{equation*}We conclude that the family $\{\mu_{\beta A}\}_{\beta\in(r,+\infty)}$ satisfies the large deviation principle with rate function $I$.

    In particular, it follows from \cite[Theorem~3.2]{Je06} that the set of potentials in $\aholder$ with a unique maximizing measure is generic. If, in addition, $T$ is Lipschitz continuous, then it follows from the reformulation of \cite[Theorem~A]{Co16} in \cite{Boc19} that the unique maximizing property holds for an open and dense subset of $\aholder$.
\end{proof}

We now prove Theorem~\ref{t: equi conditions}.
\begin{proof}[\bf Proof of Theorem~\ref{t: equi conditions}]
We have shown that the three families 
\begin{equation*}
\bigl\{\widetilde{\beta A}\big/\beta\bigr\}_{\beta\in(1,+\infty)}, \quad
\bigl\{ \beta^{-1}\log u_{\beta A}\bigr\}_{\beta\in(1,+\infty)}, \quad\text{ and }  
\bigl\{l_{\beta}^m|_{C(X,\R)}\bigr\}_{\beta\in(1,+\infty)}
\end{equation*} 
are normal in Theorem~\ref{t:Log type limit for equilibrium states and altered potentials}~(i), Lemma~\ref{l: zero tem limits of eigenfunctions}, and Theorem~\ref{t:zero-temperature limit}~(i), respectively. It suffices to show that the limit of every convergent subsequence must be the same function or functional. Since $\{\mu_{\beta A}\}_{\beta\in(1,+\infty)}$ satisfies the large deviation principle as $\beta\to+\infty$, it follows from the discussion above (\ref{eq:LDP in functionals}) and Remark~\ref{r: l_beta at infinity} that $\bigl\{l_{\beta}^\mu\bigr\}_{\beta\in(1,+\infty)}$ pointwise converges as $\beta\to+\infty$. Denote this limit by $l$. It follows from Theorem~\ref{t:zero-temperature limit}~(ii) that $l$ is tropical linear. Let $b'$ be the density of $l$ in $\density(X)$ (cf.~Proposition~\ref{p:compactness implies continuity} and Remark~\ref{r:density and functional}).

\smallskip
(i) Suppose that the subsequence $\bigl\{\widetilde{\beta_k A}\big/\beta_k\bigr\}_{k\in\N}$ uniformly converges to $\twidetilde{A}\in C(X,\R)$ with $\beta_k\to+\infty$ as $k\to+\infty$. Then it follows from Theorem~\ref{t:Log type limit for equilibrium states and altered potentials}~(iii) that 
\begin{equation}\label{eq:condition determining the altered potential}
	b'(T(x))+\twidetilde{A}(x)=b'(x)
\end{equation}
for all $x\in X$. If $b'(x_0)\in\R$ for some $x_0\in X$, then it follows from $\twidetilde{A}\in C(X,\R)$ that $b'(T(x_0))\in\R$ and $\twidetilde{A}(x_0)=b'(x_0)-b'(T(x_0))$. We conclude that the values of $\twidetilde{A}$ at points in $\{y\in X:b'(y)\in\R\}$ are determined by (\ref{eq:condition determining the altered potential}).

 Recall that $\mu_{\beta A}$ is a probability measure for all $\beta>0$. Note that 
\begin{equation*}
\oplusu{x\in X} (0\otimes b'(x) )=\lim_{\beta\to+\infty}l_{\beta}^{\mu}(\tmif)=\lim_{\beta\to+\infty}\frac{1}{\beta}\log\mu_{\beta A}(X)=0 .
\end{equation*}
It follows that $b':X \rightarrow \Rminus$.

We claim that $\{y\in X:b'(y)\in\R\}$ is dense in $X$. If the claim holds, then the values of $\twidetilde{A}$ on $X$ are all determined by (\ref{eq:condition determining the altered potential}) and the continuity of $\twidetilde{A}$. Thus, every uniformly convergent subsequence $\bigl\{\widetilde{\beta_k A}\big/\beta_k\bigr\}_{k\in\N}$ must converge to the same function $\twidetilde{A}\in C(X,\R)$, and (i) is verified.

Now we prove the claim by contradiction. Suppose that $\{y\in X:b'(y)\in\R\}$ is not dense, i.e., there is an open set $U\subseteq X$ such that $b'(y)=-\infty$ for all $y\in U$. Note that (\ref{eq:condition determining the altered potential}) implies that $b'(T(y))=-\infty$ if $b'(y)=-\infty$. It follows that for all $y\in U$ and $n\in\N$, $b'(T^n(y))=-\infty$.
Since $T$ is open, distance-expanding, and transitive, there exists a positive integer $M$ so that $X=\bigcup_{i=0}^MT^i(U)$ (cf.~\cite[Theorem~4.3.12]{PU10}). Thus, 
\begin{equation*}
0=\oplusu{x\in X}b'(x)=\oplusu{0\leq i\leq M}\oplusu{y\in T^i(U)}b'(y)=\oplusu{0\leq i\leq M}(-\infty)=-\infty.
\end{equation*}
This is a contradiction, and our claim follows.

\smallskip

(ii) We have proved in (i) that $\widetilde{\beta A}\big /\beta$ uniformly converges to $\twidetilde{A}$ as $\beta\to+\infty$.
Now suppose that the subsequence $\bigl\{ \beta_k^{-1}\log u_{\beta_k A}\bigr\}_{k\in\N}$ uniformly converges to $v\in C(X,\R)$ with $\beta_k\to+\infty$ as $k\to+\infty$. Then, by Theorem~\ref{t:Log type limit for equilibrium states and altered potentials}~(ii),
\begin{equation*}
\twidetilde{A}=A-\energy+v-v\circ T.
\end{equation*}
This implies that $v-v\circ T$ is uniquely determined. 

Recall $\int\! u_{\beta A}\,\mathrm{d}m_{\beta A}=1$ for all $\beta>0$. Suppose a subsequence $\{\beta'_k\}_{k\in\N}$ of the sequence $\{\beta_k\}_{k\in\N}$ satisfies $\lim_{k\to+\infty}\beta'_k=+\infty$ and that $l_{\beta'_k}^m$ pointwise converges to some $\check{l}$ as $k\to+\infty$.

Similar to the argument for the claim in the proof of Corollary~\ref{c:uniquely maximizing implies LDP} (cf.~(\ref{ineq:intermidate for linear normalizing condition})), we have
\begin{align*}
\Absbig{\check{l}(v)}
\leq \Absbig{\check{l}(v)-l_{\beta'_k}^m(v)}+\Absbig{l_{\beta'_k}^m(v)}
&=\Absbig{\check{l}(v)-l_{\beta'_k}^m(v)}+\Absbig{l_{\beta'_k}^m(v)-l_{\beta'_k}^m\bigl({\beta'_k}^{-1}\log u_{\beta'_k A}\bigr)}\\
&\leq\Absbig{\check{l}(v)-l_{\beta'_k}^m(v)}+\Normbig{v-{\beta'_k}^{-1}\log u_{\beta'_k A}}_{C^0},
\end{align*}
where the equality follows from  $\int\! u_{\beta A}\,\mathrm{d}m_{\beta A}=1$ and the second inequality follows from~(\ref{ineq:regularity of eigenmeasure_functionals at positive temperature}). As $k\to+\infty$ in the above inequalities, we have $\check{l}(v)=0$. Moreover, $\check{l}$ is tropical linear by Theorem~\ref{t:zero-temperature limit}~(i).

We claim that the uniqueness of $v-v\circ T$, together with $\check{l}(v)=0$, implies the uniqueness of $v$. Assume that there exists $v_1,\, v_2\in C(X,\R)$ such that $v_1-v_1\circ T=v_2-v_2\circ T$ and $\check{l}(v_1)=\check{l}(v_2)=0$. It follows that $v_1-v_2=(v_1-v_2)\circ T$ and $v_1-v_2\in C(X,\R)$. It then follows from the transitivity of $T$ that $v_1-v_2\equiv c$, where $c\in\R$. Thus,
\begin{equation*}
0=\check{l}(v_1)=\check{l}(v_2\otimes c)=\check{l}(v_2)\otimes c=0\otimes c=c,
\end{equation*}
i.e., $v_1=v_2$ and the claim follows. Now (ii) is verified.

\smallskip

(iii) Recall that $l_{\beta}^{\mu}(f)=l_{\beta}^{m}\bigl(f+\frac{1}{\beta}\log u_{\beta A}\bigr)$ for all $\beta>0$ and $f\in C(X,\R)$ since $\mu_{\beta A}=u_{\beta A}\cdot m_{\beta A}$. By (ii), $\frac{1}{\beta}\log u_{\beta A}$ uniformly converges to some $v\in C(X,\R)$ as $\beta\to+\infty$.  Recall that $l_{\beta}^\mu$ pointwise converges to $l$ as $\beta\to+\infty$. 

Now suppose that $\check{l}$ is the pointwise limit of $l_{\beta_k}^m$ with $\beta_k\to+\infty$ as $k\to+\infty$. It follows from (\ref{ineq:regularity of eigenmeasure_functionals at positive temperature}) that for all $f\in C(X,\R)$,
\begin{equation}\label{ineq: in_proof_equivalence_connect two kinds of fuctionals}
\begin{aligned}
\Absbig{l_{\beta_k}^m(f+v)-l(f)}&\leq\Absbig{l_{\beta_k}^m(f+v)-l_{\beta_k}^{\mu}(f)}+\Absbig{l_{\beta_k}^{\mu}(f)-l(f)}\\
&=\Absbig{l_{\beta_k}^m(f+v)-l_{\beta_k}^m(f+\beta_k^{-1}\log u_{\beta_k A})}+\Absbig{l_{\beta_k}^{\mu}(f)-l(f)}\\
&\leq\Normbig{v-\beta_k^{-1}\log u_{\beta_k A}}_{C^0}+\Absbig{l_{\beta_k}^{\mu}(f)-l(f)}.
\end{aligned}
\end{equation}
As $k\to+\infty$ in the above inequalities, we see that $l(f)=\check{l}(f+v)$ for all $f\in C(X,\R)$, i.e., $\check{l}(g)=l(g-v)$ for all $g\in C(X,\R)$. We conclude that $\check{l}$ is uniquely determined, and~(iii) is verified.

\smallskip
Finally, assume that $\frac{1}{\beta}\log u_{\beta A}$ uniformly converges to $v\in C(X,\R)$ as $\beta\to+\infty$ and that $l_{\beta}^m$ pointwise converges to $\check{l}$ as $\beta\to+\infty$ (by (\ref{eq: l_beta at infinity})). By Theorem~\ref{t:zero-temperature limit}~(i), $\check{l}$ is a tropical linear functional. Let $b$ be the density of $\check{l}$ in $\density(X)$ (cf.~Proposition~\ref{p:compactness implies continuity} and Remark~\ref{r:density and functional}). Similar to (\ref{ineq: in_proof_equivalence_connect two kinds of fuctionals}), it follows from (\ref{ineq:regularity of eigenmeasure_functionals at positive temperature}), (\ref{ineq:regularity of equilibrium_functionals at positive temperature}), and $l_\beta^\mu(f)=l_{\beta}^m\bigl(f+\frac{1}{\beta}u_{\beta A}\bigr)$ for all $\beta>0$ and $f\in C(X,\R)$ that $\lim_{\beta\to+\infty}l_{\beta}^\mu(f)=\check{l}(f+v)=\oplusu{x\in X}(f(x)\otimes(v(x)\otimes b(x)))$ for all $f\in C(X,\R)$. Thus, the family $\{\mu_{\beta A}\}_{\beta\in(1,+\infty)}$ satisfies the large deviation principle as $\beta\to+\infty$ with rate function $-(v\otimes b)$.
\end{proof}

\appendix

\section{}\label{ss: Appendix A}

For the convenience of the reader, this appendix includes proofs of a few results that may be known to experts.

\begin{proof}[\bf Proof of Proposition~\ref{p:L_continuous}]
Fix $x\in X$ and denote $T^{-1}(x)\=\{x_1,\,\dots,\,x_n\}$ where $n\in\N_0$. Let $\xi>0$ be the constant in Lemma~\ref{l:expanding systems}. For all $y\in B(x,\xi)$, denote $y_i\=T_{x_i}^{-1}(y)$ for each integer $1\leq i\leq n$ and consequently $T^{-1}(y)=\{y_1,\,\dots,\,y_n\}$ according to Lemma~\ref{l:expanding systems}. If $x\in X\smallsetminus T(X)$, i.e., $n=0$, then $\cL_A(u)(x)=\cL_A(u)(y)=-\infty$ for all $y\in B(x,\xi)$. 

Otherwise, $\cL_{A}(u)(x)=\max\limits_{1\leq i\leq n}\{u(x_i)+A(x_i)\}$ and $\cL_{A}(u)(y)=\max\limits_{1\leq i\leq n}\{u(y_i)+A(y_i)\}$ for all $y\in B(x,\xi)$. It follows that for all $y\in B(x,\xi)$,
\begin{equation}\label{ineq:continuity of L_A(u)}
\abs{\exp(\cL_{A}(u)(x))-\exp(\cL_{A}(u)(y))}\leq \oplusu{1\leq i\leq k}\Absbig{e^{u(x_{i})}e^{A(x_i)}-e^{u(y_{i})}e^{A(y_{i})}}.
\end{equation} 

Assume that $A\in C(X,\R)$ and $u\in C(X,\Rminus)$. Since $X$ is compact, it follows that $A$ and $e^u$ are both uniformly continuous.
Thus, for each $\epsilon>0$, there exists $\delta>0$ such that $\Absbig{e^{u(z_1)}e^{A(z_1)}-e^{u(z_2)}e^{A(z_2)}}<\epsilon$ for all $z_1,\, z_2\in X$ with $d(z_1,z_2)<\delta$. 

Then for all $y\in B(x,\min\{\xi,\,\lambda\delta\})$, Lemma~\ref{l:expanding systems} implies that $d(x_i,y_i)\leq\lambda^{-1}d(x,y)<\delta$ for each integer $1\leq i\leq n$. It follows that
$\Absbig{e^{u(x_{i})}e^{A(x_i)}-e^{u(y_{i})}e^{A(y_{i})}}<\epsilon$ for each integer $1\leq\ i\leq n$. Thus, it follows from (\ref{ineq:continuity of L_A(u)}) that $\abs{\exp(\cL_{A}(u)(x))-\exp(\cL_{A}(u)(y))}<\epsilon$ for all $y\in B(x,\min\{\xi,\,\lambda\delta\})$. We conclude that $\exp(\cL_A(u))\in C(X,\R)$ and thus $\cL_A(u)\in C(X,\Rminus)$.

Moreover, if $T$ is surjective and $u\in C(X,\R)$, then $\cL_A(u)(x)\in\R$ for all $x\in\R$ and consequently $\cL_A(u)\in C(X,\R)$ with $\norm{\cL_A(u)}_{C^0}<+\infty$ since $X$ is compact. If $T$ is surjective and $A,\, u\in \aholder$, then we see that for all $y\in B(x,\xi)$,
\begin{equation*}
    \abs{\cL_A(u)(x)-\cL_A(u)(y)}\leq \oplusu{1\leq i\leq k}\abs{u(x_i)-u(y_i)+A(x_i)-A(y_i)}.
\end{equation*}
Since $d(x_i,y_i)\leq\lambda^{-1}d(x,y)$ for each $y\in B(x,\xi)$ and each integer $1\leq i\leq n$ (by Lemma~\ref{l:expanding systems}), it follows that $\abs{\cL_{A}(u)}_{d^\alpha,\xi}\leq \lambda^{-\alpha}(\abs{u}_{d^\alpha}+\abs{A}_{d^\alpha})$. It immediately follows that $\cL_{A}(u)\in\aholder$.
\end{proof}

Recall the definition of $T_x^{-n}$ from Remark~\ref{r: defi of n times pull back}.
\begin{proof}[\bf Proof of Lemma~\ref{l:distortion}]
Let $\xi>0$ be the constant in Lemma~\ref{l:expanding systems}.
Fix $x,\, y\in X$ with $d(x,y)<\xi$ and $n\in\N$. Note that if $T^{-n}(x)=\emptyset$, then it follows from Lemma~\ref{l:expanding systems} that $T^{-n}(y)=\emptyset$.

Assume that $T^{-n}(x)\neq\emptyset$. Denote $T^{-n}(x)\=\{x_1,\,\dots,\,x_k\}$ and $y_i\=T_{x_i}^{-n}(y)$ for all $1\leq i\leq k$. Then Lemma~\ref{l:expanding systems} implies that $T^{-n}(y)=\{y_1,\,\dots,\,y_k\}$ and $d\bigl(T^l(x_i),T^l(y_i)\bigr)\leq\lambda^{l-n}d(x,y)$ for all $0\leq l\leq n$ and $1\leq i\leq k$. It follows that 
\begin{equation*}
\AbsBig{\oplusu{\ox\in T^{-n}(x)}S_nA(\ox)-\oplusu{\oy\in T^{-n}(y)}S_nA(\oy)}\leq\oplusu{1\leq i\leq k}\abs{S_nA(x_i)-S_nA(y_i)}.
\end{equation*} 
Since $A\in\aholder$, it follows that for all $1\leq i\leq k$,
\begin{equation*}
\begin{aligned}
\abs{S_nA(x_i)-S_nA(y_i)}
&\leq\abs{A}_{d^\alpha}\bigl(d(x_i,y_i)^\alpha+ \cdots +d\bigl(T^{n-1}(x_i),T^{n-1}(y_i)\bigr)^\alpha\bigr)\\
&\leq\abs{A}_{d^\alpha}d(x,y)^\alpha(\lambda^{-n\alpha}+ \cdots +\lambda^{-\alpha})
\leq\abs{A}_{d^\alpha}d(x,y)^\alpha(\lambda^{\alpha}-1)^{-1}.
\end{aligned}
\end{equation*}
Now (\ref{ineq:difference of the nearby trajectory sum}) is verified. It then follows from Lemma~\ref{l: generalized form of exchange lemma}~(ii) and $d(x_i,y_i)\leq\lambda^{-n}d(x,y)$ for all $1\leq i\leq k$ that for all $u\in C(X,\R)$,
\begin{equation*}
    \begin{aligned}
    \cL_A^n(u)(x)&=\oplusu{1\leq i\leq k}(u(x_i)\otimes S_nA(x_i))\\
    &\leq\oplusu{1\leq i\leq k}\bigl(u(y_i)\otimes \omega_{u}(\lambda^{-n}d(x,y))\otimes S_nA(y_i)\otimes\abs{A}_{d^\alpha}d(x,y)^\alpha(\lambda^{\alpha}-1)^{-1}\bigr)\\
    &=\cL_A^n(u)(y)+\omega_{u}(\lambda^{-n}d(x,y))+\abs{A}_{d^\alpha}d(x,y)^\alpha(\lambda^{\alpha}-1)^{-1}.
    \end{aligned}
\end{equation*}
Note that (\ref{ineq: equi estimate for L_A_n(u)}) trivially holds if $T^{-n}(x)=\emptyset$. We conclude that (\ref{ineq: equi estimate for L_A_n(u)}) is verified.

Now assume that $T$ is transitive. It directly follows that $T$ is surjective. Thus, we take $C_0\=\abs{A}_{d^\alpha}\xi^\alpha(\lambda^\alpha-1)^{-1}$ and conclude that for all $n\in\N$ and $x,\, y\in X$ with $d(x,y)<\xi$,
\begin{equation}\label{ineq:distortion for nearby points}
		\AbsBig{\oplusu{\ox\in T^{-n}(x)}S_nA(\ox)-\oplusu{\oy\in T^{-n}(y)}S_nA(\oy)}\leq C_0.
\end{equation}

\smallskip
\emph{Claim.} There exists $N_{\xi}\in\N$ such that for all $x,\, y\in X$, there exists an integer $m$ satisfying $0\leq m\leq N_{\xi}$ and $T^m(B(x,\xi))\cap B(y,\xi)\neq\emptyset$. 
\smallskip

Since $X$ is compact, there exists a finite set $\{z_1,\,\dots,\,z_s\}\subseteq X$ with $\bigcup_{i=1}^s B(z_i,\xi/2)=X$. For all $x,\, y\in X$, there exist $i,\, j\in\{1,\,\dots,\,s\}$ such that $d(x,z_i)<\xi/2$ and $d(y,z_j)<\xi/2$. Thus, $B(z_i,\xi/2)\subseteq B(x,\xi)$ and $B(z_j,\xi/2)\subseteq B(y,\xi)$. We conclude that if for some $m\in\N_0$, $T^m(B(z_i,\xi/2))\cap B(z_j,\xi/2)\neq\emptyset$, then $T^m(B(x,\xi))\cap B(y,\xi)\neq\emptyset$. 
 
It follows from the transitivity of $T$ that there exists $m_{ij}\in\N_0$ such that  for all $i,\, j\in\{1,\,\dots,\, s\}$, we have $T^{m_{ij}}(B(z_i,\xi/2))\cap B(z_j,\xi/2)\neq\emptyset$. Thus, we can take
\begin{equation}   \label{e:Def_N_xi}
	N_{\xi}\=\max\limits_{1\leq i,j\leq s}{m_{ij}} <+\infty.
\end{equation}
Now the claim is verified.
 
 \smallskip
	
Fix $x,\, y\in X$ and $n\in\N$. Then the claim implies that there exists $y'\in T^{-m}(B(y,\xi))\cap B(x,\xi)$ for some integer $m$ satisfying $0\leq m\leq N_{\xi}$. It follows from (\ref{ineq:distortion for nearby points}) that 
	\begin{equation*}
		\begin{aligned}
			\oplusu{\ox\in T^{-n}(x)}S_nA(\ox)&\leq C_0+\oplusu{\overline{y'}\in T^{-n}(y')}S_nA(\overline{y'})\\
			&=C_0-S_mA(y')+\oplusu{\overline{y'}\in T^{-n}(y')}S_{n+m}A(\overline{y'})\\
			&\leq C_0-m\inf_{x\in X} A(x)+\oplusu{\overline{y'}\in T^{-n-m}(T^m(y'))}S_{n+m}A(\overline{y'})\\
			&=C_0-m\inf_{x\in X} A(x)+\oplusu{\overline{y'}\in T^{-n-m}(T^m(y'))}(S_mA(\overline{y'})+S_nA(T^m(\overline{y'})))\\
			&\leq C_0-m\inf_{x\in X} A(x)+m\sup_{x\in X} A(x)+\oplusu{\overline{y'}\in T^{-n-m}(T^m(y'))}S_nA(T^m(\overline{y'}))\\
			&=C_0-m\inf_{x\in X} A(x)+m\sup_{x\in X} A(x)+\oplusu{z\in T^{-n}(T^m(y'))}S_nA(z)\\
			&\leq 2C_0+N_{\xi}\Bigl(\sup_{x\in X} A(x)-\inf_{x\in X} A(x)\Bigr)+\oplusu{\oy\in T^{-n}(y)}S_nA(\oy).
		\end{aligned}
	\end{equation*}
Thus, we can take $C_1\=2C_0+N_{\xi}(\sup_{x\in X} A(x)-\inf_{x\in X} A(x))$ and get that for all $x,\, y\in X$ and $n\in\N$,
\begin{equation}\label{ineq:distortion for all points}
	\AbsBig{\oplusu{\ox\in T^{-n}(x)}S_nA(\ox)-\oplusu{\oy\in T^{-n}(y)}S_nA(\oy)}\leq C_1 .
\end{equation}
	
Fix $A,\, u\in\aholder$. Proposition~\ref{p:L_continuous} implies that $\cL_A^n(u)\in\aholder$ for all $n\in\N$. Now it suffices to give the estimate for $\abs{\cL_{A}^n(u)}_{d^\alpha}$.  

It follows from (\ref{ineq: equi estimate for L_A_n(u)}) that for all $n\in\N$,
\begin{equation*}
	\abs{\cL_{A}^n(u)}_{d^\alpha,\xi}\leq\abs{A}_{d^\alpha}(\lambda^{\alpha}-1)^{-1}+\abs{u}_{d^\alpha}\lambda^{-n\alpha}\leq(\lambda^{\alpha}-1)^{-1}(\abs{A}_{d^\alpha}+\abs{u}_{d^\alpha}).
\end{equation*}
For $x,\, y\in X$ with $d(x,y)\geq\xi$, (\ref{ineq:distortion for all points}) implies that for all $n\in\N$,
\begin{align*}
		 \abs{\cL_{A}^n(u)(x)-\cL_{A}^n(u)(y)}
		 & \leq C_1+\sup_{x\in X} u(x)-\inf_{x\in X} u(x) \\
		 & \leq \Bigl(C_1+\sup_{x\in X} u(x)-\inf_{x\in X} u(x)\Bigr)  \xi^{ - \alpha} d(x,y)^{\alpha}   .
\end{align*}
We conclude that for all $n\in\N$,
\begin{equation}\label{ineq:Lipschitz constant estimate for L_A^n(u)}
\abs{\cL_A^n(u)}_{d^\alpha}\leq\max\Bigl\{(\lambda^{\alpha}-1)^{-1}(\abs{A}_{d^\alpha}+\abs{u}_{d^\alpha}),\, \Bigl(C_1+\sup_{x\in X} u(x)-\inf_{x\in X} u(x)\Bigr) \xi^{-\alpha}\Bigr\}.
\end{equation}
Now let $C_2(A,u)$ denote a positive constant satisfying $\abs{\cL_{A}^n(u)}_{d^\alpha}\leq C_2(A,u)(\abs{A}_{d^\alpha}+\abs{u}_{d^\alpha})$ for a specific pair $A,\, u\in\aholder$ and all $n\in\N$. 

If $A,\, u\in\aholder$ are two constant functions, then for each $n\in\N$, $\cL_{A}^n(u)$ is a constant function. Thus, $0=\abs{A}_{d^\alpha}=\abs{u}_{d^\alpha}=\abs{\cL_{A}^n(u)}_{d^\alpha}$ for all $n\in\N$ and, consequently, $C_2(A,u)$ can be an arbitrary positive number. 

Now suppose that there is a nonconstant function among $A$ and $u$, i.e., $\abs{A}_{d^\alpha}+\abs{u}_{d^\alpha}>0$. By (\ref{ineq:Lipschitz constant estimate for L_A^n(u)}), we can take $C_2(A,u)\=\max\bigl\{(\lambda^{\alpha}-1)^{-1},\,\frac{C_1+\sup_{x\in X} u(x)-\inf_{x\in X} u(x)}{\xi^\alpha(\abs{A}_{d^\alpha}+\abs{u}_{d^\alpha})}\bigr\}$.

Moreover, the fact that $A,\, u\in\aholder$ implies that 
	\begin{align*}
		\sup_{x\in X} A(x)-\inf_{x\in X} A(x) & \leq \abs{A}_{d^\alpha}(\diam X)^\alpha\quad\text{ and}\\
		\sup_{x\in X} u(x)-\inf_{x\in X} u(x) & \leq \abs{u}_{d^\alpha}(\diam X)^\alpha.
	\end{align*} 
	Recall that $C_1=2C_0+N_{\xi}(\sup_{x\in X} A(x)-\inf_{x\in X} A(x))$ and $C_0=\abs{A}_{d^\alpha}\xi^\alpha(\lambda^{\alpha}-1)^{-1}$. We conclude that
\begin{equation*}
\begin{aligned}
C_2(A,u)
&\leq \max\biggl\{(\lambda^{\alpha}-1)^{-1},\, \frac{2C_0+N_{\xi}\abs{A}_{d^\alpha}(\diam X)^\alpha+\abs{u}_{d^\alpha}(\diam X)^\alpha}{\xi^\alpha(\abs{A}_{d^\alpha}+\abs{u}_{d^\alpha})}\biggr\}\\
&\leq \max\biggl\{(\lambda^{\alpha}-1)^{-1},\, \frac{N_{\xi}(\diam X)^\alpha}{\xi^\alpha}+\frac{2C_0}{\xi^\alpha(\abs{A}_{d^\alpha}+\abs{u}_{d^\alpha})}\biggr\}\\
&\leq \max\biggl\{(\lambda^{\alpha}-1)^{-1},\, \frac{N_{\xi}(\diam X)^\alpha}{\xi^\alpha}+2(\lambda^{\alpha}-1)^{-1}\biggr\}.
\end{aligned}
\end{equation*} 
Thus, we take $C_2\=2(\lambda^{\alpha}-1)^{-1}+\frac{N_{\xi}(\diam X)^\alpha}{\xi^\alpha}$ and conclude that 
\begin{equation*}
\abs{\cL_A^n(u)}_{d^\alpha}\leq C_2(\abs{A}_{d^\alpha}+\abs{u}_{d^\alpha})
\end{equation*}
for all $A,\, u\in\aholder$ and $n\in\N$. The proof is now complete.
\end{proof}

\begin{proof}[\bf Proof of Lemma~\ref{l: maximizing measures supported on Aubry set}]
    Let $v_A\in \aholder$ be the sub-action for $A\in \aholder$ from Proposition~\ref{p: mane lemma}. It follows that $B\=A+v_A-v_A\circ T\fless\energy$ and $Q(T,B)=Q(T,A)$. Thus, an invariant probability measure $\mu$ is a maximizing measure for $A$ if and only if $\mu$ is a maximizing measure for $B$ if and only if $\mu$ is supported on $B^{-1}(\energy)$. Denote $K\=B^{-1}(\energy)$. Now it suffices to show that an invariant probability measure $\mu$ is supported on $K$ if and only if $\mu$ is supported on $\Omega_A$.

    Note that $S_n B(x)=S_nA(x)+v_A(x)-v_A(T^n(x))$ for all $n\in\N$ and $x\in X$. It then follows from Definition~\ref{d:Aubry set} and $v_A\in C(X,\R)$ that $\Omega_{B}=\Omega_A$. Fix an arbitrary $x_0\in \Omega_{B}$ and $\epsilon>0$. By Definition~\ref{d:Aubry set}, there exists $y\in B(x_0,\epsilon)$ and $n\in\N$ such that $d(x_0,T^n(y))\leq\epsilon$ and $\abs{S_n(B-Q(T,B))(y)}\leq\epsilon$. Since $B\fless\energy= Q(T,B)$, we see that 
    \begin{equation*}
        0\geq B(y)-Q(T,B)\geq S_nB(y)-nQ(T,B)\geq -\epsilon.
    \end{equation*}
    It then follows from $d(x_0,y)\leq\epsilon$ and $v_A\in\aholder$ that 
    \begin{equation*}
        \abs{B(x_0)-Q(T,B)}
        \leq \abs{B}_{d^\alpha}d(x_0,y)+\abs{B(y)-Q(T,B)}
        \leq \epsilon \abs{B}_{d^\alpha}+\epsilon.
    \end{equation*}
    We conclude $B(x)=Q(T,B)=\energy$ for all $x\in\Omega_{B}$, i.e., $\Omega_A=\Omega_{B}\subseteq K$. It follows that if a probability measure $\mu$ is supported on $\Omega_A$, then $\mu$ is supported on $K$.

    Assume that an invariant probability measure $\mu$ is supported on $K$. Since $\mu$ is invariant, we see that $\mu$ is supported on $\bigcap_{n\in\N_0}T^{-n}(K)$. Denote the set of nonwandering points by $\Omega(T)$. Note that the invariant measure $\mu$ is also supported on $\Omega(T)$ (cf.~\cite[Theorem~6.15]{Wal82}). We conclude that $\supp\mu\subseteq\bigcap_{n\in\N_0}T^{-n}(K)\cap\Omega(T)$. Fix $x_1\in\bigcap_{n\in\N_0}T^{-n}(K)\cap\Omega(T)$ and $\epsilon>0$. Since $x_1\in \Omega(T)$, there exists $n\in\N$ such that $d(T^{n}(x_1),x_1)\leq\epsilon$. Since $x_1\in \bigcap_{n\in\N_0}T^{-n}(K)$ and $\energy=Q(T,B)$, we see that $S_n(B)(x_1)-nQ(T,B)=0$. It then follows from Definition~\ref{d:Aubry set} that $x_1\in\Omega_{B}$. We conclude that $\bigcap_{n\in\N_0}T^{-n}(K)\cap\Omega(T)\subseteq\Omega_{B}=\Omega_A$ and $\mu$ is supported on $\Omega_A$. The proof is now complete.
\end{proof}

\begin{lemma}\label{l: mane potential is upper semi-continuous}
    Let $T\: X\rightarrow X$ be an open continuous distance-expanding map on a compact metric space $(X,d)$, and $A\in\aholder$ with $\alpha \in (0,1]$. Then the Ma\~{n}\'{e} potential $\phi_A$ is upper semi-continuous. 
\end{lemma}
\begin{proof}
    It suffices to show that for all $x,\, y\in X$,
    \begin{equation}\label{ineq: upper semi-continuity of mane potential}
        \lim_{\epsilon\to 0^+}\oplusu{z\in B(x,\epsilon)\atop w\in B(y,\epsilon)}\phi_A(z,w)\leq \phi_A(x,y).
    \end{equation}
    Fix $\delta>0$, $z\in B(x,\delta/2)$, and $w\in B(y,\delta/2)$. It immediately follows from Definition~\ref{d:Mane potential} that 
    \begin{equation*}
      \phi_A(z,w)\leq\oplusu{n\in\N}  \oplusu{d(z_1,z)\leq\delta/2\atop d(T^n(z_1),w)\leq\delta/2}S_n\oA(z_1)\leq \oplusu{n\in\N}\oplusu{d(z_1,x)\leq\delta\atop d(T^n(z_1),y)\leq\delta}S_n\oA(z_1).
    \end{equation*}
    Thus, 
    \begin{equation*}
        \oplusu{z\in B(x,\delta/2)\atop w\in B(y,\delta/2)}\phi_A(z,w)\leq \oplusu{n\in\N}\oplusu{d(z_1,x)\leq\delta\atop d(T^n(z_1),y)\leq\delta}S_n\oA(z_1).
    \end{equation*}
As $\delta\to 0^+$ in the above inequality, we see that (\ref{ineq: upper semi-continuity of mane potential}) follows from Definition~\ref{d:Mane potential}.
\end{proof}

The following lemma is used in the proof of Proposition~\ref{p:Mane_potential_properties}~(v).
\begin{lemma}\label{l:long trajectory for Aubry point}
Let $T\:X \rightarrow  X$ be an open continuous distance-expanding map on a compact metric space $(X,d)$, and $A \in \aholder$ with $\alpha \in (0,1]$. Let $\xi>0$ be the constant in Lemma~\ref{l:expanding systems}. For all $x_0\in\Omega_A$, $\epsilon\in(0,\xi)$, and $l\in\N$, there exists $x_1\in B(x_0,\epsilon)$ and $n>l$ such that $T^n(x_1)=x_0$ and $\abs{S_n\toverline{A}(x_1)}<\epsilon$.
\end{lemma}
\begin{remark}
This lemma is slightly different from \cite[Corollary~4.5]{Ga17}. In this lemma, we assume $A\in\aholder$ and require $T^n(x_1)=x_0$ while \cite[Corollary~4.5]{Ga17} assumes $A\in C(X,\R)$ and only requires $d(T^n(x_1),x_0)\leq\epsilon$. Although \cite[Corollary~4.5]{Ga17} is stated for subshifts of finite type, its proof is applicable to our setting. Thus, we directly use it in the following proof.
\end{remark}
\begin{proof}
Fix $\epsilon\in(0,\xi)$, $l\in\N$, and $x_0\in\Omega_A$. Then fix $\delta\in (0,\epsilon/2)$ satisfying 
\begin{equation*}
	\delta+(\lambda^{\alpha}-1)^{-1}\abs{A}_{d^\alpha}\delta^{\alpha}<\epsilon<\xi.
\end{equation*}

It follows from \cite[Corollary~4.5]{Ga17} that there exists $x_2\in B(x_0,\delta)$ and $n>l$ such that $d(T^n(x_2),x_0)\leq \delta$ and $\abs{S_n\toverline{A}(x_2)}\leq\delta$.

Since $d(T^n(x_2),x_0)\leq\delta<\xi$ and $x_2\in B(x_0,\delta)$,  Lemma~\ref{l:expanding systems} implies that
\begin{equation*}
	d\bigl(T_{x_2}^{-n}(x_0),x_0\bigr)
	\leq d\bigl(T_{x_2}^{-n}(x_0),x_2\bigr)+d(x_2,x_0)
	\leq \lambda^{-n}d(x_0,T^n(x_2))+d(x_2,x_0)
	\leq 2\delta
	<\epsilon.
\end{equation*}
Denote $x_1\=T_{x_2}^{-n}(x_0)$. It follows from (\ref{ineq:difference of the nearby trajectory sum}) that 
\begin{equation*}
	\abs{S_n\toverline{A}(x_1)-S_n\toverline{A}(x_2)}\leq(\lambda^{\alpha}-1)^{-1}\abs{A}_{d^\alpha}\delta^{\alpha}. 
\end{equation*}
We conclude that $d(x_1,x_0)<\epsilon$, $T^n(x_1)=x_0$, and 
\begin{equation*}
	\abs{S_n\toverline{A}(x_1)}\leq \delta+(\lambda^{\alpha}-1)^{-1}\abs{A}_{d^\alpha}\delta^{\alpha}<\epsilon,
\end{equation*}where the last inequality follows from our choice of $\delta$.
\end{proof}

Now we provide a proof of Proposition~\ref{p:Mane_potential_properties}.

\begin{proof}[\bf Proof of Proposition~\ref{p:Mane_potential_properties}]
(i) Fix $u\in\es{\cL_A}{C(X,\Rminus)}{\energy}$. Then $u(T(x))\geq u(x)\otimes\toverline{A}(x)$ for all $x\in X$. Thus, $u(z)\otimes S_n\toverline{A}(z)\leq u(T^n(z))$ for all $z\in X$ and $n\in\N$. 

Fix $x,\, y\in X$. Since $u\in C(X,\Rminus)$, for every $\epsilon>0$, there exists $\eta(\epsilon)\in(0,\epsilon)$ such that 
\begin{equation*}
	u(x)\leq u(z)+\epsilon\qquad\text{ and }\qquad u(w)\leq\log(\exp (u(y))+\epsilon)
\end{equation*} for all $z\in B(x,\eta(\epsilon))$ and $w\in B(y,\eta(\epsilon))$. By Definition~\ref{d:Mane potential}, we see that
\begin{align*}
u(x)\otimes\phi_A(x,y)&=\lim_{\epsilon\to 0^+}u(x)\otimes\Bigl(\oplusu{n\in\N}\oplusu{d(z,x)\leq\eta(\epsilon)\atop d(T^n(z),y)\leq\eta(\epsilon)}S_n\toverline{A}(z)\Bigr)\\
&\leq\limsup_{\epsilon\to 0^+}\Bigl(\oplusu{n\in\N}\oplusu{d(z,x)\leq\eta(\epsilon)\atop d(T^n(z),y)\leq\eta(\epsilon)}(S_n\toverline{A}(z)\otimes u(z)\otimes\epsilon)\Bigr)\\
&\leq\limsup_{\epsilon\to 0^+}\Bigl(\oplusu{n\in\N}\oplusu{d(z,x)\leq\eta(\epsilon)\atop d(T^n(z),y)\leq\eta(\epsilon)}(u(T^n(z))\otimes\epsilon)\Bigr)\\
&\leq\limsup_{\epsilon\to 0^+}\log(\exp (u(y))+\epsilon)\otimes\epsilon\\
&=u(y).
\end{align*}

(ii)--(iv) Statements~(ii), (iii), and (iv) are already established in~Subsection~\ref{ss:Mane_potential}.

\smallskip

(v) Let $\xi>0$ be the constant in Lemma~\ref{l:expanding systems}, and fix $x_0\in\Omega_A$.

\smallskip
Step~1. We first establish the claim below, which will also be useful in the following steps.

\smallskip

\emph{Claim~1}. For all $z\in X$ and every $y_0\in B(x_0,\xi)$ satisfying $T^n(y_0)=z$ for some $n\in\N_0$, we have
\begin{equation}\label{ineq: Mane(x,) is real when x is Aubry}
\phi_A(x_0,z)\geq S_n\toverline{A}(y_0)-(\lambda^{\alpha}-1)^{-1}\abs{A}_{d^\alpha}d(x_0,y_0)^{\alpha}.
\end{equation}

Now assume that $z\in X$ and $y_0\in B(x_0,\xi)$ with $T^n(y_0)=z$ for some $n\in\N_0$. 

Fix an arbitrary $\epsilon_1\in(0,\xi)$. Fix some $l\in\N$ depending on $\epsilon_1$ and satisfying $\lambda^{-l}\xi<\epsilon_1/2$. By Lemma~\ref{l:long trajectory for Aubry point}, there exists $x_1\in B(x_0,\epsilon_1/2)$ and $n_0>l$ such that $T^{n_0}(x_1)=x_0$ and $\abs{S_{n_0}\toverline{A}(x_1)}<\epsilon_1$.

It follows from Lemma~\ref{l:expanding systems} and $y_0\in B(x_0,\xi)$ that
\begin{align*}
	d\bigl(T_{x_1}^{-n_0}(y_0),x_0\bigr)
	& \leq d\bigl(T_{x_1}^{-n_0}(y_0),x_1\bigr)+d(x_1,x_0) \\
	& \leq \lambda^{-n_0}d(y_0,x_0)+d(x_1,x_0)
	\leq \lambda^{-l}\xi+\epsilon_1/2
	< \epsilon_1.
\end{align*} 
Since $y_0\in B(x_0,\xi)$, it follows from (\ref{ineq:difference of the nearby trajectory sum}) that
\begin{equation}\label{ineq: the difference of nearby trajectory sum in step2}
\Absbig{S_{n_0}\toverline{A}(x_1)-S_{n_0}\toverline{A}\bigl(T_{x_1}^{-n_0}(y_0)\bigr)}\leq(\lambda^{\alpha}-1)^{-1}\abs{A}_{d^\alpha}d(x_0,y_0)^{\alpha}.
\end{equation}
By (\ref{ineq: the difference of nearby trajectory sum in step2}) and the definition of $x_1$, we conclude that
\begin{align*}
\oplusu{m\in\N}\oplusu{d(y_1,x_0)\leq\epsilon_1\atop T^m(y_1)=z}S_m\toverline{A}(y_1)
&\geq S_{n_0+n}\toverline{A}\bigl(T_{x_1}^{-n_0}(y_0)\bigr)\\
&\geq S_{n_0}\toverline{A}(x_1)-(\lambda^{\alpha}-1)^{-1}\abs{A}_{d^\alpha}d(x_0,y_0)^{\alpha}+S_n\toverline{A}(y_0)\\
&\geq -\epsilon_1+S_n\toverline{A}(y_0)-(\lambda^{\alpha}-1)^{-1}\abs{A}_{d^\alpha}d(x_0,y_0)^{\alpha}.
\end{align*}
As $\epsilon_1\to0^+$ in the above inequality, (\ref{ineq: Mane(x,) is real when x is Aubry}) follows, and Claim~1 is verified.

\smallskip

Recall that $\phi_A(\cdot,\cdot) \: X\times X \rightarrow \Rminus$. Note that the transitivity of $T$ implies the existence of $y_0\in B(x_0,\xi)$ and $n\in\N_0$ satisfying $T^n(y_0)=z$. We conclude that if $T$ is transitive, then $\phi_A(x_0,z)\in\R$ for all $x_0\in\Omega_A$ and $z\in X$. 

\smallskip 
Step~2. We show that $\phi_A(x_0,z)=\oplusu{y\in T^{-1}(z)}(\phi_A(x_0,y)\otimes \toverline{A}(y))$ for all $z\in X$. 

Fix $z\in X$. It immediately follows from Definition~\ref{d:Mane potential} that $\phi_A(x,T(x))\geq \toverline{A}(x)$ for all $x\in X$. It then follows from Proposition~\ref{p:Mane_potential_properties}~(i) that 
\begin{equation*}
	\oplusu{y\in T^{-1}(z)}(\phi_A(x_0,y)\otimes \toverline{A}(y))\leq\oplusu{y\in T^{-1}(z)}(\phi_A(x_0,y)\otimes\phi_A(y,z))\leq \phi_A(x_0,z).
\end{equation*}
Now it suffices to show that $\phi_A(x_0,z)\leq\oplusu{y\in T^{-1}(z)}(\phi_A(x_0,y)\otimes\toverline{A}(y))$.

Fix arbitrary $\epsilon_2\in(0,\xi)$. Recall that we have proved 
\begin{equation*}
	\phi_A(x_0,z)=\lim_{\epsilon\to0^+}\oplusu{n\in\N}\oplusu{d(y_0,x_0)\leq\epsilon\atop T^n(y_0)=z}S_n\toverline{A}(y_0)
\end{equation*}as a claim in the proof of Proposition~\ref{p:Mane_potential_properties}~(ii). Note that the above limit is a decreasing limit. Thus, there exists $w\in X$ and $m\in\N$ such that $d(w,x_0)\leq\epsilon_2$, $T^m(w)=z$, and 
\begin{equation*}
	\phi_A(x_0,z)\leq\oplusu{n\in\N}\oplusu{d(y_0,x_0)\leq\epsilon_2\atop T^n(y_0)=z}S_n\toverline{A}(y_0)\leq S_m\toverline{A}(w)+\epsilon_2=S_{m-1}\toverline{A}(w)+\toverline{A}(w)+\epsilon_2.
\end{equation*}
Note that $m-1\in\N_0$. We now apply (\ref{ineq: Mane(x,) is real when x is Aubry}) with $T^{m-1}(w)$ in place of $z$, $w$ in place of $y_0$, and $m-1$ in place of $n$. It follows that 
\begin{equation*}
	S_{m-1}\toverline{A}(w)
 \leq \phi_A\bigl(x_0,T^{m-1}(w) \bigr)+(\lambda^{\alpha} -1)^{-1}\abs{A}_{d^\alpha}d(w,x_0)^\alpha.
\end{equation*}
Denote $G(\epsilon_2)\=\epsilon_2+(\lambda^{\alpha}-1)^{-1}\abs{A}_{d^\alpha}\epsilon_2^\alpha$. Recall $T^m(w)=z$. We conclude that 
\begin{align*}
	\phi_A(x_0,z)
 & \leq \phi_A \bigl(x_0,T^{m-1}(w) \bigr)+\toverline{A}(w)+G(\epsilon_2) \\
 & \leq\oplusu{y\in T^{-1}(z)}(\phi_A(x_0,y)\otimes\toverline{A}(y))+G(\epsilon_2).
\end{align*}
As $\epsilon_2\to0^+$ in the above inequality, we get $\phi_A(x_0,z)\leq\oplusu{y\in T^{-1}(z)}(\phi_A(x_0,y)\otimes\toverline{A}(y))$.

\smallskip
Step~3. We verify the regularity of $\phi_A(x_0,\cdot)$.

\smallskip

\emph{Claim~2}. For all $z_1,\, z_2\in X$ satisfying $d(z_1,z_2)<\xi$,
\begin{equation}\label{ineq: continuity of Mane(x_0) when x_0 is Aubry}
\phi_A(x_0,z_1)\leq\phi_A(x_0,z_2)+(\lambda^{\alpha}-1)^{-1}\abs{A}_{d^\alpha}d(z_1,z_2)^{\alpha}.
\end{equation}

Now fix $z_1,\, z_2\in X$ with $d(z_1,z_2)<\xi$.

Fix an arbitrary $\epsilon_3\in(0,\xi/2)$. Fix some $l\in\N$ depending on $\epsilon_3$ and satisfying $\lambda^{-l}\xi<\epsilon_3/2$. By Lemma~\ref{l:long trajectory for Aubry point}, there exists $x_1\in B(x_0,\epsilon_3/2)$ and $n_1>l$ such that $T^{n_1}(x_1)=x_0$ and $\abs{S_{n_1}\toverline{A}(x_1)}<\epsilon_3$.

Assume that there exists $n\in\N$ and $y_1\in T^{-n}(z_1)$ such that $y_1\in B(x_0,\epsilon_3/2)\subseteq B(x_0,\xi)$. Denote $y_2\=T_{x_1}^{-n_1}(y_1)$ and $n_2\=n+n_1$. Lemma~\ref{l:expanding systems} implies that
\begin{equation}\label{ineq: x_0 and y_2 is close}
\begin{aligned}
d(y_2,x_0)
& \leq d(y_2,x_1)+d(x_1,x_0) \\
& \leq \lambda^{-n_1}d(y_1,x_0)+d(x_1,x_0)
\leq \epsilon_3/2+\epsilon_3/2 
= \epsilon_3.
\end{aligned}
\end{equation}
Since $d(y_1,x_0)<\epsilon_3/2<\xi$, it follows from (\ref{ineq:difference of the nearby trajectory sum}) that 
\begin{equation}\label{ineq: add a long trajectory with small sum}
\begin{aligned}
\abs{S_{n_1}\toverline{A}(y_2)}
& \leq \abs{S_{n_1}\toverline{A}(x_1)}+(\lambda^{\alpha}-1)^{-1}\abs{A}_{d^\alpha}(\epsilon_3/2)^{\alpha}   \\
& \leq \epsilon_3+(\lambda^{\alpha}-1)^{-1}\abs{A}_{d^\alpha}(\epsilon_3/2)^{\alpha}.
\end{aligned}
\end{equation}

Note that $T^{n_2}(y_2)=T^n(y_1)=z_1$.
Since $d(z_1,z_2)<\xi$, it follows that
\begin{equation}\label{ineq: z_2 pull long enough is close to x_0}
\begin{aligned}
d\bigl(T_{y_2}^{-n_2}(z_2),x_0\bigr)
&\leq d\bigl(T_{y_2}^{-n_2}(z_2),y_2\bigr)+d(y_2,x_0) \\
&\leq \lambda^{-n_2}d(z_2,z_1)+\epsilon_3
< \lambda^{-l}\xi+\epsilon_3
<2\epsilon_3,
\end{aligned}
\end{equation}
where the second inequality follows from Lemma~\ref{l:expanding systems} and (\ref{ineq: x_0 and y_2 is close}), and the third follows from $n_2>n_1>l$. Moreover, it follows from (\ref{ineq:difference of the nearby trajectory sum}) that 
\begin{equation}\label{ineq: difference of z_1 and z_2 pull long back}
\Absbig{S_{n_2}\toverline{A}\bigl(T_{y_2}^{-n_2}(z_2)\bigr)-S_{n_2}\toverline{A}(y_2)}\leq(\lambda^{\alpha}-1)^{-1}\abs{A}_{d^\alpha}d(z_1,z_2)^{\alpha}.
\end{equation}
Denote $F(\epsilon_3)\=\epsilon_3+(\lambda^{\alpha}-1)^{-1}\abs{A}_{d^\alpha}(\epsilon_3/2)^{\alpha}$. By (\ref{ineq: difference of z_1 and z_2 pull long back}), (\ref{ineq: add a long trajectory with small sum}), and (\ref{ineq: z_2 pull long enough is close to x_0}), we conclude that 
\begin{align*}
S_n\toverline{A}(y_1)&=S_{n_2}\toverline{A}(y_2)-S_{n_1}\toverline{A}(y_2)\\
&\leq S_{n_2}\toverline{A}\bigl(T_{y_2}^{-n_2}(z_2)\bigr)+(\lambda^{\alpha}-1)^{-1}\abs{A}_{d^\alpha}d(z_1,z_2)^{\alpha}+\epsilon_3+(\lambda^{\alpha}-1)^{-1}\abs{A}_{d^\alpha}(\epsilon_3/2)^{\alpha}\\
&\leq \Bigl(\oplusu{m\in\N}\oplusu{d(z,x_0)\leq 2\epsilon_3\atop T^m(z)=z_2}S_m\toverline{A}(z)\Bigr)+(\lambda^{\alpha}-1)^{-1}\abs{A}_{d^\alpha}d(z_1,z_2)^{\alpha}+F(\epsilon_3).
\end{align*}
Thus,
\begin{equation}\label{ineq: pre for regularity of mane potential}
\oplusu{n\in\N}\oplusu{d(y,x_0)<\epsilon_3/2\atop T^n(y)=z_1}S_n\toverline{A}(y)\leq  \oplusu{m\in\N}\oplusu{d(z,x_0)\leq 2\epsilon_3\atop T^m(z)=z_2}S_m\toverline{A}(z)+(\lambda^{\alpha}-1)^{-1}\abs{A}_{d^\alpha}d(z_1,z_2)^{\alpha}+F(\epsilon_3).
\end{equation}
Note that (\ref{ineq: pre for regularity of mane potential}) holds even if there does not exist $y_1\in\bigcup_{n\in\N}T^{-n}(z_1)$ with $y_1\in B(x_0,\epsilon_3/2)$.
As $\epsilon_3\to0^+$ in (\ref{ineq: pre for regularity of mane potential}), we get (\ref{ineq: continuity of Mane(x_0) when x_0 is Aubry}). 

Thus, Claim~2 follows. 

\smallskip

If $T$ is transitive, it follows from Claim~1 that $\phi_A(x_0,\cdot)\in C(X,\R)$ (see the discussion before Step~2) and thus $\phi_A(x_0,\cdot)\in\aholder$ with \begin{equation*}
	\abs{\phi_A(x_0,\cdot)}_{d^\alpha,\xi}\leq\abs{A}_{d^\alpha}(\lambda^{\alpha}-1)^{-1} .
\end{equation*}

Therefore, $\phi_A(x_0,\cdot)\in \es{\cL_A}{C(X,\Rminus)}{\energy}$, and if $T$ is transitive, then $\phi_A(x_0,\cdot)\in\aholder$. Now (v) is verified.
\end{proof}

\begin{proof}[\bf Proof of Proposition~\ref{p:eigenfunction}]
	Fix $u\in \es{\cL_A}{C(X,\Rminus)}{\energy}$.
	It follows from Proposition~\ref{p:Mane_potential_properties}~(i) that  $u(\cdot)\geq\oplusu{x\in\Omega_A}(u(x)\otimes\phi_A(x,\cdot))$. It suffices to find an Aubry point $x_y$ for each $y\in X$ such that $u(x_y)\otimes\phi_A(x_y,y)\geq u(y)$. Fix $y\in X$. If $u(y)=-\infty$, then the above inequality trivially holds for all $x_y\in\Omega_A$. Now assume that $u(y)\in\R$.
 
	 By Lemma~\ref{l:expanding systems}, $T^{-1}(x)$ is finite for all $x\in X$. Since $u(y)=\oplusu{z\in T^{-1}(y)}\bigl(u(z)+\toverline{A}(z))$ and $u(y)\in\R$, there exists $y_1\in T^{-1}(y)$ such that $u(y)=u(y_1)+\toverline{A}(y_1)$, and thus $u(y_1)\in\R$. Then since $u(y_1)=\oplusu{z\in T^{-1}(y_1)}(u(z)+\toverline{A}(z))$ and $u(y_1)\in\R$, there exists $y_2\in T^{-1}(y_1)$ such that $u(y_1)=u(y_2)+\toverline{A}(y_2)$, and thus $u(y_2)\in\R$. Repeating this process recursively, we get a sequence $\{y_k\}_{k\in\N}$ in $X$ satisfying $T(y_{k+1})=y_k$ and $u(y_k)=u(y_{k+1})+\toverline{A}(y_{k+1})$ for all $k\in\N$. 
	
	\smallskip
	\emph{Claim.} Every accumulation point $w$ of $\{y_k\}_{k\in\N}$ is an Aubry point, and satisfies $u(w)+\phi_A(w,y)\geq u(y)$.
	\smallskip
	
	Suppose that the subsequence $\{y_{n_k}\}_{k\in\N}$ converges to $x_y$ with $n_k\to+\infty$ as $k\to+\infty$. Without loss of generality, we assume that $n_{k+1}>n_k+k$ for all $k\in\N$. It follows from Lemma~\ref{l: uniform upper bound without transi and surje} and the choice of $y_k$ that there exists $D>0$ such that $u(y_k)=u(y)-S_k\toverline{A}(y_k)\geq u(y)-D$ for all $k\in\N$. Since $u\in C(X,\Rminus)$ and $y_{n_k}$ converges to $x_y$ as $k\to+\infty$, we see that $u(x_y)\in\R$.
	
	Fix $\epsilon>0$. The continuity of $u$, together with $u(x_y)\in\R$, implies that there exists $\eta\in(0,\epsilon/2)$ such that $\abs{u(x_y)-u(z)}\leq\epsilon/2$ for all $z\in B(x_y,\eta)$. Thus, there exists $N\in\N$ such that for all $k\geq N$, $d(y_{n_k},x_y)<\eta<\epsilon/2$ and $d(y_{n_{k+1}},x_y)<\eta<\epsilon/2$. Write $n_{k+1}-n_k\eqqcolon m_k\in\N$. It follows that for all $k\geq N$,
	 \begin{equation*}
		\abs{S_{m_k}\toverline{A}(y_{n_k})}=\abs{u(y_{n_k})-u(y_{n_{k+1}})}\leq2\cdot(\epsilon/2)=\epsilon.
	\end{equation*}	
	 Since $d(y_{n_k},y_{n_{k+1}})\leq d(y_{n_k},x_y)+d(x_y,y_{n_{k+1}})\leq\epsilon$, by Definition~\ref{d:Aubry set}, we conclude that $x_y$ is an Aubry point. Moreover, the inequalities $d(y_{n_k},x_y)<\epsilon/2$ and $d(y_{n_{k+1}},x_y)<\epsilon/2$ imply that for all $k\geq N$,
	 \begin{equation*}
	 	u(y)
        =u(y_{n_k})\otimes S_{n_k}\toverline{A}(y_{n_k})
        \leq (u(x_y)+\epsilon/2)\otimes\Bigl(\oplusu{n\in\N}\oplusu{d(z,x_y)
        \leq\epsilon/2\atop d(T^n(z),y)\leq\epsilon/2}S_n\toverline{A}(z)\Bigr).
	 	\end{equation*}
 	 As $\epsilon\to 0^+$ in the above inequality, we get $u(y)\leq u(x_y)\otimes\phi_A(x_y,y)$. Now the claim is verified, and it follows that $u(\cdot)=\oplusu{x\in\Omega_A}(u(x)\otimes\phi_A(x,\cdot))$.
\end{proof}

\end{document}